\documentclass[twoside,11pt]{article}

\usepackage[preprint]{jmlr2e}

\usepackage{xcolor}
\usepackage{booktabs}
\usepackage{makecell}
\usepackage{subcaption}
\usepackage{enumitem}

\usepackage{xspace}
\usepackage{bbm}
\input{mathlig}
\usepackage{mathtools}
\usepackage{relsize}
\usepackage{mathrsfs}
\usepackage{dsfont}
\DeclarePairedDelimiterX{\inp}[2]{\langle}{\rangle}{#1, #2}
\makeatletter
\newcommand*\bigcdot{\mathpalette\bigcdot@{.5}}
\newcommand*\bigcdot@[2]{\mathbin{\vcenter{\hbox{\scalebox{#2}{$\m@th#1\bullet$}}}}}
\makeatother

\newcommand{\muspace}{\mspace{1mu}}

\DeclareRobustCommand{\scond}{\mathchoice{\muspace\vert\muspace}{\vert}{\vert}{\vert}}
\mathlig{|}{\scond}

\DeclareRobustCommand{\discint}{\mathchoice{\mspace{-1.5mu}:\mspace{-1.5mu}}{\mspace{-1.5mu}:\mspace{-1.5mu}}{:}{:}}
\mathlig{::}{\discint}
\newcommand{\suchthat}{\mathchoice{\colon}{\colon}{:\mspace{1mu}}{:}}

\def\Var{\mathop{\rm Var}\nolimits}%

\newcommand{\Dc}{\mathcal{D}}

\newcommand{\Kc}{\mathcal{K}}
\newcommand{\Lc}{\mathcal{L}}

\newcommand{\Nc}{\mathcal{N}}

\newcommand{\Vc}{\mathcal{V}}

\newcommand{\Xc}{\mathcal{X}}
\newcommand{\Yc}{\mathcal{Y}}

\renewcommand{\Pr}{\mathscr{P}}

\newcommand{\Xv}{{\bf X}}

\newcommand{\ph}{{\hat{p}}}

\newcommand{\rb}{{\mathbf r}}

\newcommand{\Xb}{{\mathbf X}}

\newcommand{\Yb}{\mathbf{Y}}

\newcommand{\Yh}{{\hat{Y}}}

\newcommand{\gh}{{\hat{g}}}

\newcommand{\pt}{{\tilde{p}}}

\def\a{\alpha}
\def\b{\beta}
\def\g{\gamma}
\def\d{\delta}

\def\eps{\epsilon}

\def\Th{\Theta}

\DeclareMathOperator\E{\mathsf{E}}
\let\P\relax
\DeclareMathOperator\P{\mathsf{P}}

\DeclareMathOperator\Prob{\mathrm{Pr}}

\newcommand{\Binom}{\mathrm{Binom}}
\newcommand\eg{e.g.,\xspace}
\newcommand\ie{i.e.,\xspace}
\def\textiid{i.i.d.\@\xspace}
\newcommand\iid{\ifmmode\text{ i.i.d. } \else \textiid \fi}

\newcommand{\Real}{\mathbb{R}}
\newcommand{\Natural}{\mathbb{N}}
\newcommand{\Integer}{\mathbb{Z}}

\newcommand{\ones}{\mathds{1}}

\newcommand{\half}{\frac{1}{2}}%

\def\mathllap{\mathpalette\mathllapinternal}
\def\mathllapinternal#1#2{%
  \llap{$\mathsurround=0pt#1{#2}$}}

\def\clap#1{\hbox to 0pt{\hss#1\hss}}
\def\mathclap{\mathpalette\mathclapinternal}
\def\mathclapinternal#1#2{%
  \clap{$\mathsurround=0pt#1{#2}$}}

\let\oldstackrel\stackrel
\renewcommand{\stackrel}[2]{\oldstackrel{\mathclap{#1}}{#2}}

\DeclarePairedDelimiterX{\infdivx}[2]{(}{)}{%
  #1\;\delimsize\|\;#2%
}

\renewcommand{\hbar}{h\mathllap{\overline{\vphantom{h}\hphantom{\rule{4.6pt}{0pt}}}\mspace{0.77mu}}}

\catcode`~=11 %
\newcommand{\urltilde}{\kern -.06em\lower -.06em\hbox{~}\kern .02em}
\catcode`~=13 %

\DeclarePairedDelimiterX{\norm}[1]{\lVert}{\rVert}{#1}
\DeclarePairedDelimiterX{\abs}[1]{\lvert}{\rvert}{#1}

\usepackage{xparse}

\newcommand*\diff{\mathop{}\!\mathrm{d}}

\let\oldpartial\partial
\renewcommand*{\partial}{\mathop{}\!\oldpartial}

\newcommand{\defeq}{\mathrel{\mathop{:}}=}

\newcommand{\supp}{\textnormal{supp}}

\newcommand{\numberthis}{\addtocounter{equation}{1}\tag{\theequation}}

\makeatletter
\let\save@mathaccent\mathaccent
\newcommand*\if@single[3]{%
  \setbox0\hbox{${\mathaccent"0362{#1}}^H$}%
  \setbox2\hbox{${\mathaccent"0362{\kern0pt#1}}^H$}%
  \ifdim\ht0=\ht2 #3\else #2\fi
  }
\newcommand*\rel@kern[1]{\kern#1\dimexpr\macc@kerna}
\newcommand*\wideaccent[2]{\@ifnextchar^{{\wide@accent{#1}{#2}{0}}}{\wide@accent{#1}{#2}{1}}}
\newcommand*\wide@accent[3]{\if@single{#2}{\wide@accent@{#1}{#2}{#3}{1}}{\wide@accent@{#1}{#2}{#3}{2}}}
\newcommand*\wide@accent@[4]{%
  \begingroup
  \def\mathaccent##1##2{%
    \let\mathaccent\save@mathaccent
    \if#42 \let\macc@nucleus\first@char \fi
    \setbox\z@\hbox{$\macc@style{\macc@nucleus}_{}$}%
    \setbox\tw@\hbox{$\macc@style{\macc@nucleus}{}_{}$}%
    \dimen@\wd\tw@
    \advance\dimen@-\wd\z@
    \divide\dimen@ 3
    \@tempdima\wd\tw@
    \advance\@tempdima-\scriptspace
    \divide\@tempdima 10
    \advance\dimen@-\@tempdima
    \ifdim\dimen@>\z@ \dimen@0pt\fi
    \rel@kern{0.6}\kern-\dimen@
    \if#41
      #1{\rel@kern{-0.6}\kern\dimen@\macc@nucleus\rel@kern{0.4}\kern\dimen@}%
      \advance\dimen@0.4\dimexpr\macc@kerna
      \let\final@kern#3%
      \ifdim\dimen@<\z@ \let\final@kern1\fi
      \if\final@kern1 \kern-\dimen@\fi
    \else
      #1{\rel@kern{-0.6}\kern\dimen@#2}%
    \fi
  }%
  \macc@depth\@ne
  \let\math@bgroup\@empty \let\math@egroup\macc@set@skewchar
  \mathsurround\z@ \frozen@everymath{\mathgroup\macc@group\relax}%
  \macc@set@skewchar\relax
  \let\mathaccentV\macc@nested@a
  \if#41
    \macc@nested@a\relax111{#2}%
  \else
    \def\gobble@till@marker##1\endmarker{}%
    \futurelet\first@char\gobble@till@marker#2\endmarker
    \ifcat\noexpand\first@char A\else
      \def\first@char{}%
    \fi
    \macc@nested@a\relax111{\first@char}%
  \fi
  \endgroup
}
\makeatother

\usepackage{adjustbox}
\usepackage{multirow}

\usepackage{thm-restate}
\newtheorem{assumption}{Assumption}

\DeclareMathOperator{\BetaDist}{\mathsf{Beta}}
\DeclareMathOperator{\GammaDist}{\mathsf{G}}
\DeclareMathOperator{\InvGammaDist}{\mathsf{IG}}
\DeclareMathOperator{\Leb}{\lambda_{\mathsf{Leb}}}
\DeclareMathOperator*{\diam}{diam}
\DeclareMathOperator{\rvol}{\varrho}

\def\ups{\upsilon}
\newcommand{\bigO}{O}
\newcommand{\bigOtilde}{\tilde{\bigO}}
\newcommand{\Bb}{\mathbb{B}}
\newcommand{\Cab}{C}

\def\CIS{\mathop{\rm CIS}\nolimits}
\renewcommand{\Pr}{\mathsf{P}}
\renewcommand{\E}{\mathbb{E}}
\newcommand{\etah}{\hat{\eta}}

\newcommand{\tnn}{T_{\mathrm{NN}}}
\newcommand{\aH}{\a_{\mathrm{H}}}

\newcommand{\gbayes}{g^*}

\newcommand{\openball}{\Bb^o}
\newcommand{\closedball}{\Bb}
\newcommand{\ball}{\mathcal{B}}
\newcommand{\binomrv}{\mathsf{B}}
\newcommand{\dataset}{\mathscr{D}}
\newcommand{\datasetX}{\Xb}
\newcommand{\partitions}{\mathscr{P}}
\newcommand{\bad}{\mathsf{bad}}

\newcommand{\taubar}{\bar{\tau}}
\newcommand{\qbar}{\bar{q}}

\newcommand{\meanstd}[2]{#1 {\scriptsize$\pm #2$}}

\newcommand{\nnrule}[2]{{#2}_{#1}}
\newcommand{\splitrule}[3]{#3_{#1,#2}}
\newcommand{\selectrule}[4]{{#4}_{#1,#2,#3}}

\newcommand{\nnreg}[1]{\nnrule{#1}{\eta}}
\newcommand{\splitreg}[2]{\splitrule{#1}{#2}{\eta}}
\newcommand{\splitregexp}[2]{\splitrule{#1}{#2}{\overline{\eta}}}
\newcommand{\selectreg}[3]{\selectrule{#1}{#2}{#3}{\eta}}
\newcommand{\selectregexp}[3]
{\selectrule{#1}{#2}{#3}{\overline{\eta}}}

\newcommand{\nncls}[1]{\nnrule{#1}{g}}
\newcommand{\splitcls}[2]{\splitrule{#1}{#2}{g}}
\newcommand{\selectcls}[3]{\selectrule{#1}{#2}{#3}{g}}

\newcommand{\nndensity}[1]{\nnrule{#1}{p}}
\newcommand{\splitdensity}[3]{\splitrule{#1}{#2}{p^{#3}}}

\newcommand{\nndensitytr}[2]{\nnrule{#1}{\pt^{#2}}}
\newcommand{\splitdensitytr}[3]{\splitrule{#1}{#2}{\pt^{#3}}}

\newcommand{\nnftr}[1]{\nnrule{#1}{\widetilde{(f\circ p)}}}
\newcommand{\splitftr}[2]{\splitrule{#1}{#2}{\widetilde{(f\circ p)}}}

\newcommand{\HM}{\mathsf{HM}}
\newcommand{\AM}{\mathsf{AM}}
\newcommand{\GM}{\mathsf{GM}}

\renewcommand{\digamma}{\Psi}
\newcommand{\density}{\rho}

\newcommand{\qed}{\hfill\BlackBox}

\usepackage{scrwfile}
\TOCclone[\appendixname]{toc}{atoc}
\newcommand\StartAppendixEntries{}
\AfterTOCHead[toc]{%
  \renewcommand\StartAppendixEntries{\value{tocdepth}=-10000\relax}%
}
\AfterTOCHead[atoc]{%
  \edef\maintocdepth{\the\value{tocdepth}}%
  \value{tocdepth}=-10000\relax%
  \renewcommand\StartAppendixEntries{\value{tocdepth}=\maintocdepth\relax}%
}

\usepackage{lastpage}
\jmlrheading{}{2024}{1-\pageref{LastPage}}{xx/xx; Revised xx/xx}{xx/xx}{paper id}{J. Jon Ryu and Young-Han Kim}

\ShortHeadings{Minimax Optimal Fixed-$k$-Nearest-Neighbors Algorithms}{Ryu and Kim}
\firstpageno{1}

\begin{document}

\title{Minimax Optimal Algorithms\\ 
with Fixed-$k$-Nearest Neighbors}

\author{\name J. Jon Ryu \email jongha@mit.edu \\
       \addr Department of Electrical Engineering and Computer Science\\
       Massachussetts Institute of Technology\\
       Cambridge, MA 02139, USA
       \AND
       \name Young-Han Kim \email yhk@ucsd.edu \\
       \addr Department of Electrical and Computer Engineering\\
       University of California San Diego\\ La Jolla, CA 92093, USA}

\editor{}

\maketitle

\begin{abstract}%
This paper presents how to perform minimax optimal classification, regression, and density estimation based on fixed-$k$ nearest neighbor (NN) searches. 
We consider a distributed learning scenario, in which a massive dataset is split into smaller groups, where the $k$-NNs are found for a query point with respect to each subset of data. 
We propose \emph{optimal} rules to aggregate the fixed-$k$-NN information for classification, regression, and density estimation that achieve minimax optimal rates for the respective problems.
We show that the distributed algorithm with a fixed $k$ over a sufficiently large number of groups attains a minimax optimal error rate up to a multiplicative logarithmic factor under some regularity conditions. 
Roughly speaking, distributed $k$-NN rules with $M$ groups has a performance comparable to the standard $\Theta(kM)$-NN rules even for fixed $k$.

\end{abstract}

\begin{keywords}
nearest neighbors, classification, regression, density estimation, distributed learning.
\end{keywords}

\section{Introduction}
Arguably one of the most primitive yet powerful nonparametric approaches for various statistical problems, the $k$-nearest-neighbor ($k$-NN) algorithms have been an essential toolkit in data science since their inception.
These algorithms have been extensively studied and analyzed over several decades for canonical statistical procedures including classification~\citep{Fix--Hodges1951,Cover--Hart1967}, regression~\citep{Cover1968a,Cover1968b}, density estimation~\citep{Loftsgaarden--Quesenberry1965,Fukunaga--Hostetler1973,Mack--Rosenblatt1979}, and density functional estimation~\citep{Kozachenko--Leonenko1987,Leonenko--Pronzato--Savani2008,Ryu--Ganguly--Kim--Noh--Lee2022}.
These algorithms remain attractive even in this modern age due to their effective performance despite their simplicity, as well as the rich understanding of their statistical properties.

There exist, however, clear limitations that hinder the wider deployment of these algorithms in practice.
First and most importantly, standard $k$-NN algorithms are often considered inherently infeasible for large-scale data, as they require storing and processing the entire data set on a single machine for nearest neighbor (NN) search. Second, although the number of neighbors $k$ needs to grow to infinity with the sample size to achieve statistical consistency in general for such procedures~\citep{Biau--Devroye2015}, small values of $k$ are highly preferred in practice to avoid the potentially demanding time complexity of large-$k$-NN search; see Section~\ref{sec:computation} for an in-depth discussion.

Recently, specifically for regression and classification, a few ensemble based methods~\citep{Xue--Kpotufe2018,Qiao--Duan--Cheng2019,Duan--Qiao--Cheng2020} have been proposed to achieve the accuracy of the optimal standard $k$-NN regression and classification rules with less computational complexity; however, theoretical guarantees of those solutions require large-$k$-NN searches.
\citet{Xue--Kpotufe2018} proposed an idea dubbed as \emph{denoising}, which is to draw multiple subsamples and preprocess them with the standard large-$k$-NN rule \emph{over the entire data} in the training phase, so that the $k$-NN information can be hashed effectively by 1-NN searches in the testing phase.
Though the resulting algorithm is provably optimal with a small statistical overhead, the denoising step may still suffer prohibitively large complexity for large $N$ and/or large $k$ in principle.
Recently, to address the computational and storage complexity of the standard $k$-NN classifier with large $N$, \citet{Qiao--Duan--Cheng2019} proposed the \emph{bigNN classifier}, which splits data into subsets, applies the standard $k$-NN classifier to each, and aggregates the labels by a majority vote.
This ensemble method works without any coordination among data splits, and thus they naturally fit to large-scale data which may be inherently stored and processed in distributed machines.
However, they showed its minimax optimality only when both the number of splits $M$ and the base $k$ increase as the sample size $N$ increases but only a strictly suboptimal guarantee for fixed $k$'s; see Section~\ref{sec:bignn} for the details.
With the increasingly-large-$k$ requirement from their theory for the optimal performance, they suggested to use the bigNN classifier in the preprocessing phase of the denoising proposal of \citep{Xue--Kpotufe2018}.
A more recent work~\citep{Duan--Qiao--Cheng2020} on a optimally weighted version of the bigNN classifier still assumes $k$ to grow.

In this paper, we complete the missing theory for small, fixed $k$ and show that the bigNN classifier with $k=1$ suffices for minimax rate-optimal classification.
More generally, we analyze a variant of the bigNN classifier, called the \emph{$M$-split $k$-NN classifier} or \emph{$(k,M)$-NN classifier} in short, which is defined as the majority vote over the total $kM$ nearest-neighbor labels obtained after running $k$-NN search over the $M$ data splits.
In general terms, we show that the $(k,M)$-NN classification rule behaves almost equivalently to the standard $\Th(M)$-NN rules, for any fixed $k\ge 1$. 
In particular, the $(1,M)$-NN rule, which is equivalent to the bigNN classifier with $k=1$, is shown to attain the minimax optimal rate up to logarithmic factors under smooth measure conditions.
We also provide a minimax-rate-optimal guarantee for regression task with an analogously defined $(k,M)$-NN regression rule.

The key technique in our analysis is to analyze intermediate rules that selectively aggregates the $k$-NN labels from each data split based on the $k$-th-NN distances from a query point. 
The intuition is that these intermediate rules which average only neighbors close enough to a query point exactly behave like a standard $\Th(M)$-NN rule for any fixed $k$. 
We establish the performance of the $(k,M)$-NN rules by showing that its performance is approximated by the intermediate rules, with a small (logarithmic) approximation overhead in the convergence rate.
Indeed, these intermediate rules, which we call the \emph{distance-selective} rules, attain exact minimax optimal rates for respective problems at the cost of additional complexity for ordering the NN distances; see Section~\ref{sec:selective}.

To provide a complete picture on the theory of distributed fixed-$k$-NNs, we also propose and analyze optimal rules for density estimation. 
We note that, unlike the two supervised learning problems above, density estimation has not been studied in the distributed learning setup.
While the $k$-NN density estimator is designed based on a different statistical property of NNs, 
it is known that $k$ needs to grow to infinity as the data size grows similar to classification and regression, for the estimator to become asymptotically consistent~\citep{Dasgupta--Kpotufe2014}.
Due to its distinct unsupervised nature, however, we need a different approach to combine the $k$-NN statistics.
The key property we utilize is that the volume of the fixed-$k$-NN ball scaled by the sample size converges to a Gamma random variable in distribution in the population limit (Proposition~\ref{prop:knndist}). 
Based on this asymptotic behavior, we design various aggregation rules that lead to asymptotically unbiased density estimators, and establish their convergence rates.

The algorithms proposed and analyzed in this paper are simple in nature, but we believe their implications may be valuable for practitioners. Specifically, while the $(\text{fixed } k, \text{growing } M)$-NN rules run faster than the standard 1-NN rules by processing smaller datasets with small-$k$-NN searches performed in parallel, they can achieve the same statistical guarantees as the \emph{optimal} standard (growing $k$)-NN rules run over the entire dataset. Moreover, when deploying these rules in practice, our analyses suggest that tuning only the number of splits $M$ (while fixing $k$, such as $k=1$) is sufficient, rather than tuning both parameters over a grid. 
From an algorithmic perspective, this implies that optimizing the performance of the 1-NN search algorithm is sufficient, without concern for the loss of statistical power.
We experimentally demonstrate that the $(1, M)$-NN rules indeed perform on par with the optimal standard $M$-NN rules as expected by theory, while running faster than the standard 1-NN rules.

\paragraph*{Organization}
The rest of the paper is organized as follows.
In Section~\ref{sec:overview}, we motivate the high-level ideas for the proposed rules and discuss the computational benefit of the data-splitting rules. 
We first present the main results for regression and classification in Section~\ref{sec:regression_classification} and then study density estimation in Section~\ref{sec:density}.
We then discuss related work in Section~\ref{sec:related}.
All the proof are deferred to Appendix.

\section{Overview: Learning with Distributed, Fixed-\texorpdfstring{$k$}{k}-Nearest-Neighbors}
\label{sec:overview}
Before we delve into the formal discussion to be followed, here we provide intuitions for the limitations of the standard fixed-$k$-NN rules and motivate how we can overcome these issues in the distributed learning setup.
We then discuss the computational benefit of the distributed learning.

\subsection{High-Level Intuitions for the \texorpdfstring{$M$}{M}-Split \texorpdfstring{$k$}{k}-NN Rules}
We will first consider the supervised learning problems of classification and regression, and the unsupervised problem of density estimation next.
In both cases, our intuitive arguments will be grounded in the consideration of the population limit.

\subsubsection{Classification and Regression}
\label{sec:intuition_classification}
Consider a binary classification problem.
Given \iid samples $\dataset=\{(X_i,Y_i)\}_{i=1}^N$ drawn from a distribution $\P$ over $\Xc\times \Yc$, where $\Yc=\{0,1\}$, the (standard) $k$-NN classifier, denoted as $\nncls{k}(x;\dataset)$, returns the majority vote of the labels of the $k$-NN instances from $\dataset$ to the query point $x$.
\citet{Cover--Hart1967} showed that the simplest 1-NN rule asymptotically achieves at most twice of the Bayes optimal error:
\begin{theorem}[\citealp{Cover--Hart1967}]
\label{thm:cover_hart}
For a metric $\rho$ defined on $\Xc$,
if $(\Xc,\rho)$ is a separable metric space, we have
\begin{align*}
\lim_{N\to\infty}\E_{(X,Y)\sim\P}[\Prob(\nncls{1}(X;\dataset)\neq Y|\dataset)]
&\le 2 \Prob(\gbayes(X)\neq Y),
\end{align*}
where $\gbayes(x)\defeq  \ones\{\eta(x)\ge 1/2\}$ denotes the Bayes optimal classifier for $\eta(x)$ denoting the conditional probability of the label $y$ being 1 given $X=x$.
\end{theorem}
The following lemma is the crucial observation to prove this theorem.
\begin{lemma}[\citealp{Cover--Hart1967}]
\label{lem:cover_hart_convergence}
Let $X_{(1)}(x)$ be the nearest neighbor of $x$ from independent and identically distributed (\iid) samples $\{X_1,\ldots,X_N\}$. If $(\Xc,\rho)$ is a separable metric space,
\[
\lim_{N\to\infty} \rho(X_{(1)}(x), x)=0 ~\text{with probability 1}.
\]    
\end{lemma}
With this lemma, the consequence is immediate: in the population limit, the 1-NN label for $x$ from the data set is essentially a random label drawn from the underlying distribution $p(y|x)$. 
Noting that the random guess incurs an error at most twice the Bayes error concludes the proof.

For $k\ge 1$, a similar convergence as in Lemma~\ref{lem:cover_hart_convergence} can be argued for the $k$-NN with any fixed $k$. This readily leads to an explicit expression of the asymptotic error probability of the $k$-NN rule, which is exactly the error probability with the majority voting rule based on $k$ random coin flips from the same label distribution~\citep[Section 5.4]{Devroye--Gyorfi--Lugosi1996}.
As one can expect the majority voting over \emph{random guesses} to converge to the Bayes rule, it can be shown that that with $k=k_N\to\infty$ and $k_N/N\to 0$ as $N\to\infty$, the $k$-NN classification rule is asymptotically consistent. 
We remark in passing that an exponential convergence of the majority voting rule with multiple random guesses to the Bayes rule was established by \citet{Bhatt--Huang--Kim--Ryu--Sen2018b}, extending the analysis of Theorem~\ref{thm:cover_hart}; see Variation 4 therein.

This asymptotic argument explains why the standard $k$-NN classifier fails with a fixed $k$ and converges to the Bayes optimal rule as $k$ grows, \ie with infinitely many random guesses, the majority voting rule converges to the Bayes optimal rule. 
In the distributed learning setup, this suggests a natural algorithm: if we are given a set of $k$-NN labels from $M$ different data splits, regardless of the size of $k$, the majority voting over the entire $kM$ labels is expected to converge to the Bayes classifier as long as the number of random guesses $kM$ grows appropriately. This is precisely the $(k,M)$-NN classifier we propose and analyze in this paper; we formally justify this intuition in our analyses. We also examine an analogous $(k,M)$-NN regression rule, which returns the mean of the $kM$ noisy labels.

\subsubsection{Density Estimation}
\label{sec:intuition_density}
For an integer $k\ge 2$ and $x\in\Xc=\Real^d$, \citet{Loftsgaarden--Quesenberry1965} proposed the $k$-NN density estimate at $x$ with respect to the sample $\Xb=X_{1:N}$ of size $N$ as
\begin{align}
\nndensity{k}(x;\Xb)
\defeq 
\frac{k-1}{U_k(x;\Xb)},
\label{eq:base_knn_density_estimate}
\end{align}
where we define $U_k(x;\Xb)\defeq N\Leb(\openball(x,r_k(x;\Xb)))$, which is a \emph{normalized} Lebesgue measure $\Leb$ over $\Real^d$ of the $k$-NN ball centered at $x$ with respect to sample $\Xb$. Here, $r_k(x;\Xb)$ denotes the distance from $x$ to its $k$-th NN in $\Xb$ and $\openball(x,r)\defeq \{y\in\Real^d\suchthat \|x-y\|_2< r\}$ denotes the open ball of radius $r>0$ centered at $x\in\Real^d$.

\citet{Loftsgaarden--Quesenberry1965} showed its weak consistency given that $k$ grows to infinity sublinearly with respect to the sample size.
\begin{restatable}[{\citealp{Loftsgaarden--Quesenberry1965}}]{theorem}{ThmKnnDensityEstimateWeakConsistency}
\label{thm:knn_density_estimate_weak_consistent}
Suppose that $p(x)$ is continuous and positive at $x$.
If $k=k_N$ satisfies that $k_N\to\infty$ as $N\to\infty$ with $k_N/N\to 0$, then $\nndensity{k}(x)$ converges to $p(x)$ in probability, denoted as $\nndensity{k}(x)\to_p p(x)$, as $n\to \infty$.
\end{restatable}
Due to this increasing-$k$ requirement for consistency, the $k$-NN density estimator is often defined as $\frac{k}{U_k(x;\Xb)}$ instead of $\frac{k-1}{U_k(x;\Xb)}$.
As will become clear below, however, for a fixed-$k$ case, the factor of $(k-1)$ is the right choice that leads to an \emph{unbiased} estimator.

To build an intuition for why the $k$-NN estimator is inconsistent for fixed $k$'s and to design an optimal rule with fixed $k$-NN statistics, we recall a fundamental and very useful property of the fixed $k$-NN: in words, a properly normalized volume of the $k$-NN ball converges to a Gamma random variable in distribution, whose shape parameter is $k$ and rate parameter is the density at the query point.
Formally, let $\GammaDist(\a,\b)$ denote a Gamma distribution with shape parameter $\a>0$ and rate parameter $\b>0$.
The following property has played a pivotal role in designing $L_2$-consistent estimators for density functionals using fixed $k$-NNs~\citep{Leonenko--Pronzato--Savani2008,Ryu--Ganguly--Kim--Noh--Lee2022}.
\begin{proposition}%
\label{prop:knndist}
Suppose that $k\ge 1$ is a fixed integer, and let $\Xb=X_{1:N}$ be \iid samples drawn from $p$ on $\Real^d$.
Then, for almost every $x$,
$U_{k}(x;\Xb)$ converges to a random variable $U_{k\infty}(x)\sim \GammaDist(k,p(x))$ in distribution as $N\to\infty$.
\end{proposition}

This asymptotic behavior can explain why the fixed-$k$-NN density is inherently inconsistent as follows.
Let $\InvGammaDist(\a,\b)$ denote a Gamma distribution with shape parameter $\a>0$ and scale parameter $\b>0$.
It is known that the reciprocal $1/U$ of a Gamma random variable $U\sim \GammaDist(\a,\b)$ follows the inverse Gamma distribution $ \InvGammaDist(\a,\b)$.
Hence, by the continuous mapping theorem and Proposition~\ref{prop:knndist}, the standard $k$-NN density estimate $\nndensity{k}(x;\Xb)$ converges to a random variable $\frac{k-1}{U_{k\infty}(x)}\sim \InvGammaDist(k,p(x))$ as $n\to\infty$.
Since the inverse Gamma distribution $\InvGammaDist(\a,\b)$ has mean $\frac{\b}{\a-1}$ if $\a>1$ and variance $\frac{\b}{(\a-1)^2(\a-2)}$ if $\a>2$, we expect to have that
\begin{align}
\label{eq:heuristic_bias}
\lim_{n\to\infty}\E[\nndensity{k}(x;X_{1:n})]
&= \E\Bigl[\frac{k-1}{U_k^\infty(x)}\Bigr]
= p(x),\quad\text{and}\\
\lim_{n\to\infty} \Var(\nndensity{k}(x;X_{1:n}))
&= \Var\Bigl(\frac{k-1}{U_k^\infty(x)}\Bigr)
= \frac{p(x)^2}{k-2}
\label{eq:heuristic_variance}
\end{align}
for $k\ge 3$.
This shows that while the $k$-NN density estimate is asymptotically unbiased for any $k\ge 2$ with the $(k-1)$ factor, the variance of $\nndensity{k}(x;\Xb)$ to vanishes if and only if $k$ grows to infinity.

Based on this observation, in Section~\ref{sec:density}, we construct a $M$-split fixed-$k$-NN density estimator by simply taking an arithmetic average of the $k$-NN density estimators over the data splits.
This allows the estimator to remain asymptotically unbiased, while at the same time its variance diminishes as the number of splits $M$ grows as $O(M^{-1})$, even when $k\ge 3$ is fixed.
We will show that this simple aggregation rule is nearly minimax optimal.
We will also discuss a class of its variants which can be constructed based on the asymptotic behavior in Proposition~\ref{prop:knndist} which could also work for any fixed $k\ge 1$, and establish their convergence rates.

\subsection{Reduced Computational Complexity with Distributed Learning}
\label{sec:computation}
As alluded to above, the standard $k$-NN rules are known to be asymptotically consistent only if $k\to\infty$ as $N\to\infty$.
Specifically to attain minimax rate-optimality,  $k=\Th(N^{\frac{2\aH}{2\aH+d}})$ is required under measures are H\"older continuity of order $\aH$; see Theorems~\ref{thm:regression}, \ref{thm:classification_rate_margin}, and \ref{thm:agg_density_estimate_l2_rate}, and their following discussions.
As alluded to earlier, this large-$k$ requirement on the standard $k$-NN rules for statistical optimality may be problematic in practice.
The main claim of this paper is that the $M$-split $1$-NN rules replace the large-$k$ requirement of the standard $k$-NN rules with a large-$M$ requirement without almost no loss in the statistical performance, while providing a natural, distributed solution to large-scale data with a possible speed-up via parallel computation.

To examine the complexity more carefully, consider Euclidean space $\Real^d$ for a moment. 
Let $\tnn(k,N)$ denote the test-time complexity of a $k$-NN search algorithm for data of size $N$. The simplest baseline NN search algorithm is the brute-force search,
which has time complexity $\tnn(k,N)=O(Nd)$ regardless of $k$.\footnote{Given a query point, 
(1) compute the distances from the data set to the query ($\bigO(Nd)$); (2)
find the $k$-NN distance using introselect algorithm ($\bigO(N)$),
(3) pick the $k$-nearest neighbors; ($\bigO(N)$).%
}
For extremely large-scale data, however, even $\bigO(N)$ may be unwieldy in practice. 
To reduce the complexity, several alternative data structures specialized for NN search such as KD-Trees~\citep{Bentley1975} for Euclidean data, and Metric Trees~\citep{Uhlmann1991} and Cover Trees~\citep{Beygelzimer--Kakade--Langford2006} for non-Euclidean data have been developed; see \citep{Dasgupta--Kpotufe2019,Kibriya--Frank2007} for an overview and comparison of empirical performance of these specialized data structures for $k$-NN search.
These are preferred over the brute-force search for better test time complexity $\bigO(\log N)$ in a moderate size of dimension, say $d\le 10$, but for much higher-dimensional data, it is known that the brute-force search may be faster.
In particular, the most popular choice of a KD-Tree based search algorithm has time complexity $\tnn(1,N)=O(2^d\log N)$ for $k=1$. The time complexity of exact $k$-NN search is $\tnn(k,N)=\bigO(k)\tnn(1,N)$ for moderately small $k$, but for a large $k$ the time complexity could be worse than $\bigO(k)\tnn(1,N)$.\footnote{One possible implementation of exact $k$-NN search algorithm with KD-tree is to remove already found points and repeatedly find 1-NN points until $k$-NN points are found using KD-tree-based 1-NN search; after the search, the removed points may be reinserted into the KD-tree without affecting the overall complexity for a moderate size of $k$.} 

Thanks to the fully distributed nature, the $(k,M)$-NN classifier have computational advantage over the standard $\Th(kM)$-NN classifier of nearly same statistical power run over the entire data.
Suppose that we split data into $M$ groups of equal size $\lceil \frac{N}{M}\rceil$ and they can be processed by $S$ parallel processors, where each processor ideally manages $\lceil \frac{M}{S}\rceil$ data splits.
Given the time complexity $\tnn(k,N)$ of a base $k$-NN search algorithm, the $(k,M)$-NN algorithms have time complexity
\[
T_{M;S}(k,N) = \Bigl\lceil\frac{M}{S}\Bigr\rceil \tnn\Bigl(k,\Bigl\lceil\frac{N}{M}\Bigr\rceil\Bigr).
\]
As to be discussed in Sections~\ref{sec:regression_classification} and \ref{sec:density}, the $(k,M)$-NN rules with $S\le M$ parallel units may attain the performance of the standard $\Th(kM)$-NN rules in a single machine with the relative speedup of 
\[
\frac{T_{M;S}(k,N)}{\tnn(kM,N)}\sim \frac{1}{S}
\]
with a brute-force search, and
\[
\frac{T_{M;S}(k,N)}{\tnn(kM,N)}\sim\frac{\frac{kM}{S}\log\frac{N}{M}}{kM\log N} = \frac{1}{S}\Bigl(1-\frac{\log M}{\log N}\Bigr)
\]
with a KD-Tree based search algorithm assuming $\tnn(k,N)=\bigO(k\log N)$ for simplicity.
Hence, the most benefit of the proposed algorithms comes from their distributed nature which reduces both time and storage complexity.

\section{Regression and Classification}
\label{sec:regression_classification}
Let $(\Xc,\rho)$ be a metric space and let $\Yc$ be the outcome (or label) space, \ie $\Yc\subseteq\Real$ for regression and $\Yc=\{0,1\}$ for binary classification.
We denote by $\P$ a joint distribution over $\Xc\times \Yc$, by $\mu$ the marginal distribution on $\Xc$, and by $\eta$ the regression function $\eta(x)=\E[Y|X=x]$.

We denote an open ball of radius $r$ centered at $x\in\Xc$ by $\openball(x,r)\defeq \{x'\in\Xc\suchthat \rho(x,x')< r\}$ and the closed ball by $\closedball(x,r)\defeq \overline{\openball(x,r)}$.
The support of a measure $\mu$ is denoted as
$\supp(\mu)\defeq 
\{x\in\Xc\suchthat \mu(\openball(x,r))>0,~\forall r>0\}$.

Given sample $\dataset=(\Xb,\Yb)=\{(X_i,Y_i)\}_{i=1}^N$ and a point $x\in\Real^d$, we use $X_{(k)}(x;\Xb)$ to denote the $k$-th-nearest neighbor of $x$ from the sample instances $\Xb=X_{1:N}$
and use $Y_{(k)}(x;\dataset)$ to denote the corresponding $k$-th-NN label among $\Yb=Y_{1:N}$; any tie is broken arbitrarily.
The $k$-th-NN distance of $x$ is denoted as $r_k(x;\Xb)\defeq \rho(x,X_{(k)}(x;\Xb))$ for $k\le N$.
We will omit the underlying data $\dataset$ or $\Xb$ whenever it is clear from the context.

For the rest of the paper, we use $N$, $M$, and $n=N/M$ to denote the size of the entire data $\dataset$, the number of data splits, and the size of each data split, respectively, assuming that $M$ divides $N$ for simplicity.

\subsection{Regression}

Given paired data $\dataset=\{(X_i,Y_i)\}_{i=1}^N$ drawn independently from the underlying joint distribution $\P$ over $\Xc\times\Yc$, the goal of regression is to design an estimator $\etah=\etah(\cdot;\dataset)\suchthat \Xc\to\Yc$ based on the data such that the estimate $\etah(x)$ is \emph{close} to the conditional expectation $\eta(x)=\E[Y|X=x]$, where the closeness between $\eta$ and $\etah$ is typically measured by the $l_p$-norm under $\mu$,
$\|\etah-\eta\|_p\defeq (\int |\etah(x)-\eta(x)|^p\mu(\diff x))^{1/p}$
for $p=1,2,$ or the sup norm 
$\|\etah-\eta\|_\infty \defeq \sup_{x\in\Xc} |\etah(x)-\eta(x)|$.

\subsubsection{\texorpdfstring{$M$}{M}-Split \texorpdfstring{$k$}{k}-NN Regression Rule}
Given a query $x\in\Xc$, we first recall that the \emph{standard $k$-NN regression rule} outputs the average of the $k$-NN labels, \ie
\[
\nnreg{k}(x;\dataset)\defeq \frac{1}{k} \sum_{i=1}^k Y_{(i)}(x;\dataset).
\]
Instead of running $k$-NN search over the entire data, given the number of splits $M\ge 1$, we first split the data $\dataset$ of size $N$ into $M$ subsets of equal size at random. 
Let $\partitions=\{\Dc_1,\ldots,\Dc_M\}$ denote the random subsets, where $\Dc_m$ corresponds to the $m$-th split. %
After finding $k$-NN labels for each data split, 
the \emph{$M$-split $k$-NN} (or \emph{$(k,M)$-NN} in short) {regression rule} is defined as the average of all the $kM$ labels, \ie 
\begin{align*}
\splitreg{k}{M}(x)
\defeq\splitreg{k}{M}(x;\partitions)
\defeq \frac{1}{M}\sum_{m=1}^M \nnreg{k}(x;\Dc_{m}).
\end{align*}

\subsubsection{Performance Guarantees}
We claim that the proposed $(k,M)$-NN regression rule for any fixed $k\ge 1$ is nearly optimal in terms of error rate under a set of standard regularity conditions.
For a formal statement, we borrow some standard assumptions on the metric measure space in the literature on analyzing the $k$-NN algorithms~\citep{Dasgupta--Kpotufe2019}.
\begin{assumption}[Doubling and homogeneous measure]\label{assum:doubling_homogeneous}
The measure $\mu$ on metric space $(\Xc,\rho)$ is \emph{doubling with exponent $d$}, \ie for any $x\in\supp(\mu)$ and $r>0$,
\[
\mu(\openball(x,r))\le 2^d\mu\Bigl(\openball\bigl(x,\frac{r}{2}\bigr)\Bigr).
\]
The measure $\mu$ is \emph{$(C_d,d)$-homogeneous}, \ie for some $C_d>0$ for any $x\in\supp(\mu)$ and $r>0$,
\[
\mu(\openball(x,r))\ge \min\{C_dr^d, 1\}.
\]
\end{assumption}
Note that a measure $\mu$ is homogeneous if $\mu$ is doubling and $\supp(\mu)$ is bounded.
The doubling exponent $d$ can be interpreted as an intrinsic dimension of a measure space. 

\begin{assumption}[H\"older continuity]
\label{assum:holder}
The conditional expectation function $\eta(x)=\E[Y|X=x]$ is \emph{$(\aH,A)$-H\"older continuous} for some $0<\aH\le 1$ and $A>0$ in metric space $(\Xc,\rho)$, \ie
for any $x,x'\in \Xc$,
\[
|\eta(x)-\eta(x')|\le A\rho^{\aH}(x,x').
\]
\end{assumption}

\begin{assumption}[Bounded conditional expectation and variance]
\label{assum:boundedness}
The conditional expectation function $\eta(x)=\E[Y|X=x]$ and the conditional variance function $v(x)\defeq \Var(Y|X=x)$ are bounded, \ie 
\[
\sup_{x\in\Xc}|\eta(x)|<\infty \text{~~and~~}\sup_{x\in \Xc} v(x)<\infty.
\]
\end{assumption}

The following condition is borrowed from \citep{Xue--Kpotufe2018} to establish a high-probability bound.
\begin{assumption}
\label{assum:stronger}
The collection of closed balls in $\Xc$ has finite VC dimension $\Vc$ and the outcome space $\Yc\subset\Real$ is contained in a bounded interval of length $l_Y$. 
\end{assumption}

The main goal of this paper is to demonstrate that the distributed $(k,M)$-NN rules can attain almost statistically equivalent performance to the optimal $k$-NN rules. 
Hence, {our statements in what follows are written in parallel to the known results for the standard $k$-NN rules, to which we remark the pointers for the interested reader.} 
For example, the following statement is new and we refer to \citep[Theorem~1.3]{Dasgupta--Kpotufe2019} for an analogous statement for the standard $k$-NN regression algorithm.

\begin{restatable}[Regression]{theorem}{ThmRegression}
\label{thm:regression}
Suppose that Assumptions~\ref{assum:doubling_homogeneous} and \ref{assum:holder} hold.
Let $k\ge 1$ be fixed.
\begin{enumerate}[label=(\alph*)]
\item If Assumption~\ref{assum:boundedness} holds and the support of $\mu$ is bounded, for any $M\le N$ such that $N/M\ge k$ and for any $\eps>0$, we have
\begin{align*}
\E_{\partitions}\|\splitreg{k}{M}-\eta\|_2
&\le C_1\Bigl\{\Bigl(\frac{M\log M}{N}\Bigr)^{\frac{\aH}{d}} + \sqrt{\frac{(\log M)^{1+\eps}}{M}}\Bigr\}.
\end{align*}
\item If Assumption~\ref{assum:stronger} holds, for any $0<\d<1$ and $0<\kappa<1$,
if $M\ge \frac{2}{\kappa^2(1-\kappa)}\log\frac{1}{\d}$, then with probability at least $1-\d$ over $\partitions$, we have
\begin{align*}
\|\splitreg{k}{M}-\eta\|_{\infty}
&\le C_2\Bigl\{ \Bigl(\frac{M}{N}\Bigl(2\Vc\log \frac{N}{M} + \log \frac{(1-\kappa)M}{\log\frac{1}{\d}}\Bigr)\Bigr)^{\frac{\aH}{d}} 
+\sqrt{\frac{1}{(1-\kappa)M}\log\frac{N}{\d}} 
\Bigr\}.
\end{align*}
\end{enumerate}
In particular, $C_1$ and $C_2$ are constants and independent of the ambient dimension $D$.
\end{restatable}

In particular, if we set $M={\Theta}(N^{\frac{2\aH}{2\aH+d}})$, Theorem~\ref{thm:regression} gives
\begin{align*}
\E_{\partitions}\|\splitreg{k}{M}-\eta\|_2&=\bigO(N^{-\frac{\aH}{2\aH+d}}(\log N)^{\half(1+\eps)})\text{ and}\\
\|\splitreg{k}{M}-\eta\|_\infty&=\bigO\Bigl(N^{-\frac{\aH}{2\aH+d}}\bigl(\log\frac{N}{\d}\bigr)^{\half}\Bigr)\text{ with probability $\ge 1-\d$}.
\end{align*}
This rate is known to be minimax optimal (modulo the polylogarithmic multiplicative terms) under the H\"older continuity of order $\aH$; for the standard $k$-NN regression algorithm, this rate optimality is attained for $k=\Th(N^{\frac{2\aH}{2\aH+d}})$~\citep{Dasgupta--Kpotufe2019,Xue--Kpotufe2018}. 
In this view, the $(k,M)$-NN regression algorithm attains the performance of the standard $\Th(M)$-NN regression algorithm for any fixed $k$.

\subsection{Classification}
We consider the binary classification with $\Yc=\{0,1\}$.
Given paired data $\dataset=\{(X_i,Y_i)\}_{i=1}^N$ drawn independently from $\P$, the goal of binary classification is to design a (data-dependent) classifier $\gh(\cdot;\dataset)\suchthat \Xc\to \Yc$ such that it minimizes the classification error $\Pr(\gh(X;\dataset)\neq Y)$.
For a classifier $\gh\suchthat\Xc\to\Yc$, we define its \emph{pointwise risk} at $x\in\Xc$ as $R(\gh;x)\defeq\Pr(Y\neq \gh(x)|X=x)$, and define the \emph{(expected) risk} as $R(\gh)\defeq \E[R(\gh;X)]=\Pr(Y\neq \gh(X))$.
Let $\gbayes(x)$ denote the \emph{Bayes-optimal} classifier, \ie $\gbayes(x)\defeq \ones\{\eta(x)\ge 1/2\}$ for all $x\in\Xc$, and let $R^*(x)\defeq R(\gbayes;x)=\eta(x)\wedge (1-\eta(x))$ and $R^*\defeq R(\gbayes)$ denote the \emph{pointwise-Bayes risk} and the \emph{(expected) Bayes risk}, respectively.
The canonical performance measure of a classifier $\gh$ is its \emph{excess risk} $R(\gh)-R^*$.

Another important performance criterion is the \emph{classification instability} proposed by \citet{Sun--Qiao--Cheng2016}, which quantifies the stablility of a classification procedure with respect to independent realizations of training data.
Given $N\ge 1$, with a slight abuse of notation, denote $\gh$ as a classification procedure $\dataset\mapsto \gh(\cdot;\dataset)$ that maps a data set $\dataset$ of size $N$ to a classifier $\gh(\cdot;\dataset)$. 
The classification instability of the classification procedure is defined as 
\[
\CIS_N(\gh)\defeq \E[\Pr(\gh(X;\dataset)\neq \gh(X;\dataset')|\dataset,\dataset')],
\]
where $\dataset$ and $\dataset'$ are independent, \iid data of size $N$.

\subsubsection{\texorpdfstring{$M$}{M}-Split \texorpdfstring{$k$}{k}-NN Classification Rule}
The \emph{standard $k$-NN classifier} is defined as the plug-in classifier of the standard $k$-NN regression estimate:
\[
\nncls{k}(x;\dataset)\defeq \ones\Bigl(\nnreg{k}(x;\dataset)\ge \frac{1}{2}\Bigr).
\]
It can be equivalently viewed as the majority vote over the $k$-NN labels given a query.
Similarly, we define the \emph{$(k,M)$-NN classification rule} as the plug-in classifier of the $(k,M)$-NN regression rule:
\[
\splitcls{k}{M}(x)
\defeq \splitcls{k}{M}(x;\partitions)
\defeq \ones\Bigl(\splitreg{k}{M}(x;\partitions)\ge \frac{1}{2}\Bigr).
\]

\subsubsection{Performance Guarantees}
As shown in the previous section for regression, we can show that the $(k,M)$-NN classifier behaves nearly identically to the standard $\Theta(M)$-NN rules for any fixed $k\ge 1$.
Here, we focus on guarantees on convergence rates of excess risk and classification instability.

To establish rates of convergence for classification,
we recall the following notion of smoothness for the conditional probability $\eta(x)=\P(Y=1|X=x)$ defined in \citep{Chaudhuri--Dasgupta2014} that takes into account the underlying measure $\mu$ to better capture the nature of classification than the standard H\"older continuity in Assumption~\ref{assum:holder}.
\begin{assumption}[Smoothness]
\label{assum:smooth}
For $\a\in(0,1]$ and $A>0$, $\eta(x)$ is \emph{$(\a,A)$-smooth in metric measure space $(\Xc,\rho,\mu)$}, \ie for all $x\in\supp(\mu)$ and $r>0$,
\[
\abs{\eta(\closedball(x,r))-\eta(x)}
\le A\mu^\a(\openball(x,r)).
\]
\end{assumption}

The following condition on the behavior of the measure $\mu$ around the decision boundary of $\eta$ is a standard assumption to establish a fast rate of convergence~\citep{Audibert--Tsybakov2007}.

\begin{assumption}[Margin condition]
\label{assum:margin}
For $\b\ge 0$, $\eta$ satisfies the \emph{$(\b,B)$-margin condition} in $(\Xc,\rho,\mu)$, \ie there exists a constant $B>0$ such that
\[
\mu(\partial\eta_\Delta)\le B\Delta^{\b},
\]
where
$\partial\eta_\Delta\defeq \{x\in\supp(\mu)\suchthat |\eta(x)-1/2|\le \Delta\}$
denotes the decision boundary with margin $\Delta\in(0,1/2]$.
\end{assumption}

The following statement can be understood as a distributed counterpart to {\citep[Theorem~4]{Chaudhuri--Dasgupta2014}}.
\begin{restatable}[Classification]{theorem}{ThmClassification}
\label{thm:classification_rate_margin}
Under Assumptions~\ref{assum:smooth} and \ref{assum:margin}, the following statements hold for any fixed $k\ge 1$, where $M_o$, $C_o$, $C_o'$, and $C_o''$ are constants depending on $\a,A,\b,B$, and $k$.
\begin{enumerate}[label=(\alph*)]
\item Pick any $\d\in(0,1)$ and $M_o>0$ such that $M=M_oN^{\frac{2\a}{2\a+1}}(\log\frac{1}{\d})^{\frac{1}{2\a+1}}\le N$. 
If $N\ge M\{2k + \log (\frac{15}{2^6}M\log\frac{2}{\d})\}$,
with probability at least $1-\d$ over $\partitions$,
\begin{align*}
&\Pr(\splitcls{k}{M}(x)\neq \gbayes(X)|\partitions)
\le \d + 
C_o
\Bigl(\frac{\log N}{N}\log\frac{2}{\d}\Bigr)^{\frac{\a\b}{2\a+1}}.
\end{align*}
\item Pick any $M_o\in (0,N^{\frac{1}{2\a+1}}]$ and set $M=M_oN^{\frac{2\a}{2\a+1}}\le N$. 
Then, for $N\ge M\{2k + \log (\frac{15}{2^6}M\log\frac{2}{\d})\}$, we have
\begin{align*}
\E_\partitions[R(\splitcls{k}{M})]-R^* &\le C_o'N^{-\frac{\a(\b+1)}{2\a+1}} (\log N)^{\half(1+\eps)(\b+1)}\quad\text{and}\quad\\
\CIS_N(\splitcls{k}{M})&\le C_o''
N^{-\frac{\a\b}{2\a+1}} (\log N)^{\half(1+\eps)\b},
\end{align*}
where $\eps>0$ is arbitrary.
\end{enumerate}
\end{restatable}

Suppose that $\eta$ is $(\aH,A)$-H\"older continuous and $\mu$ has a density with respect to Lebesgue measure that is strictly bounded away from zero on its support.
Then, by \citep[Lemma~2]{Chaudhuri--Dasgupta2014}, $\eta$ is $(\frac{\aH}{d},A)$-smooth.
Hence, if we set $M={\Theta}(N^{\frac{2\aH}{2\aH+d}})$ in Theorem~\ref{thm:classification_rate_margin}(b), we have
\begin{align*}
\E_\partitions[R(\splitcls{k}{M})]-R^* &= \bigO(N^{-\frac{\aH(\b+1)}{2\aH+d}} (\log N)^{\half(1+\eps)(\b+1)})\quad\text{and}\\
\CIS_N(\splitcls{k}{M})&=\bigO(N^{-\frac{\aH\b}{2\aH+d}}(\log N)^{\half(1+\eps)\b}),
\end{align*}
which are known to be minimax optimal (modulo the multiplicative polylogarithmic factors) under the H\"older continuity assumption~\citep{Chaudhuri--Dasgupta2014,Sun--Qiao--Cheng2016}. 
In parallel to Theorem~\ref{thm:regression} and the following discussion, the standard $k$-NN classifier is known to achieve these rates for $k=\Th(N^{\frac{2\aH}{2\aH+d}})$, and thus the $(k,M)$-NN classifier attains the performance of a standard $\Th(M)$-NN classifier in this sense.

\begin{remark}[Reduction to regression]
For a regression estimate $\etah$, let $\gh$ be the plug-in classifier with respect to $\etah$. 
Then, via the inequality
\[
R(\gh)-R^*\le 2\|\etah-\eta\|_1,
\]
the guarantees for the $(k,M)$-NN regression rule in Theorem~\ref{thm:regression} readily imply convergence rates of the excess riskeven for a multiclass classification scenario, by adapting the guarantee for a multivariate regression setting~\citep{Dasgupta--Kpotufe2019}. 
The current statements, however, are more general results for binary classification that apply to beyond smooth distributions, following the analysis by \citet{Chaudhuri--Dasgupta2014}.
\end{remark}

\subsection{Discussions}
\label{SEC:DISCUSSION}
In the previous section, we present the convergence rate guarantees for the $(k,M)$-NN classifier. In this section, we remark the implication of the results compared to \citep{Qiao--Duan--Cheng2019} in Section~\ref{sec:bignn}.
We then discuss a refined aggregation scheme based on the idea of distance-based selection, which are not only used as a proof technique for analyzing the $(k,M)$-NN rules, but also achieve minimax optimal rates without logarithmic factors on their own (Section~\ref{sec:selective}).

\subsubsection{Comparison to the BigNN Classifier}
\label{sec:bignn}
The \emph{bigNN classifier} 
proposed by \citet{Qiao--Duan--Cheng2019} takes the majority vote over the $M$ labels, each of which is the output of the standard $k$-NN classifier from a data split.
Formally, it is defined as $g_{\mathsf{big};M}^{(k)}(x;\partitions)\defeq\ones(\eta_{\mathsf{big};M}^{(k)}(x;\partitions)\ge 1/2)$, where
$\eta_{\mathsf{big};M}^{(k)}(x;\partitions)\defeq \frac{1}{M}\sum_{m=1}^M 
\nncls{k}(x;\Dc_m)
$.
\citet{Qiao--Duan--Cheng2019} showed that the bigNN classifier is minimax rate-optimal, provided that $k$ grows to infinity with a certain speed that depends on the smoothness of the underlying distribution.
\begin{theorem}[{\citealp[Theorems~1 and 2]{Qiao--Duan--Cheng2019}}, rephrased]
Assume Assumptions~\ref{assum:smooth} and \ref{assum:margin}.
Pick $1\le M\le N$ and
set $k=k_o N^{\frac{2\a}{2\a+1}}M^{-\frac{1}{2\a+1}}$ for some constant $k_o\ge 1$, such that $1\le k\le N$ and $k\to \infty$ as $N\to\infty$.
Then, we have
\begin{align*}
\E_\partitions[R(g_{\mathsf{big};M}^{(k)})]-R^* &=O(N^{-\frac{\a(\b+1)}{2\a+1}})
\quad\text{and}\quad
\CIS_N(g_{\mathsf{big};M}^{(k)})
=O(N^{-\frac{\a\b}{2\a+1}}).
\end{align*}
Further, if $k\ge 1$ is fixed, then for $M=N^{\gamma}$ with $\gamma\in(0,\frac{2\a}{2\a+1})$, we have
\begin{align*}
\E_\partitions[R(g_{\mathsf{big};M}^{(k)})]-R^* &=O(N^{-\frac{\gamma(\b+1)}{2}})
\quad\text{and}\quad
\CIS_N(g_{\mathsf{big};M}^{(k)})
=O(N^{-\frac{\gamma\b}{2}}).
\end{align*}
\end{theorem}
We note that the second part of the statement is only informally alluded to in the section on experiments of \cite{Qiao--Duan--Cheng2019}.
In the first part of the statement, $k$ must grow to infinity as $N\to\infty$, and thus if we choose $M=N^\gamma$, $\gamma$ has to be strictly less than $\frac{2\a}{2\a+1}$. 
Further, the second part of the statement only guarantees strictly suboptimal rates for fixed $k$; note that the rate exponent $\frac{\gamma(\b+1)}{2}$ is strictly less then $\frac{\a(\b+1)}{2\a+1}$, since $0<\gamma<\frac{2\a}{2\a+1}$.
Their analysis relies on the assumption that the $k$-NN classifier becomes asymptotically consistent for each $k$, and it cannot properly handle the interesting case of fixed $k$'s.
Based on our asymptotic argument in Section~\ref{sec:intuition_classification}, the growing-$k$ requirement is not necessary.
It is also worth noting that the number of splits $M=N^{\gamma}$ is restricted to be strictly slower than $\Th(N^{\frac{2\a}{2\a+1}})$, which is allowed in our analysis as the optimal choice. 
Their technique is also not readily applicable for analyzing a regression algorithm.

In contrast, in the current paper, the $(k,M)$-NN classifier takes the majority over \emph{all the $kM$ returned labels} and we establish the (near) rate-optimality \emph{for any fixed $k\ge 1$}, as long as $M$ grows properly. This implies that the $M$ sets of $k$-NN labels over subsets are almost statistically equivalent to $\Th(M)$-NN labels over the entire data.
Our analysis is based on the \emph{refined aggregation scheme} to be discussed in the next section, which provides a careful control on the behavior of distributed nearest neighbors and is naturally compatible with the analysis of the regression rule.

We remark that the bigNN rule and the $(k,M)$-NN classifier are equivalent for the most practical case of $k=1$, and we observed in our experiments that both schemes showed similar performance even for small $k$'s (data not shown). 
However, we emphasize that the suboptimality in the small-$k$ regime of the bigNN classifier in their analysis suggests to preclude the use of small $k$ in practice, whereas our analysis shows that fixed $k$, even $k=1$, suffices for optimal inference.

\subsubsection{A Refined Aggregation Scheme with Distance-Based Selection}
\label{sec:selective}
As alluded to earlier, we can remove the logarithmic factors in the guarantees of Theorems~\ref{thm:regression} and \ref{thm:classification_rate_margin} with a refined aggregation scheme which we call the \emph{distance-selective aggregation}.
With an additional hyperparameter $L\in \Natural$ such that $1\le L\le M$,
we select $L$ data splits out of the total $M$ splits based on the $k$-th-NN distances from the query point to each data split instances. 
Formally, if $m_1,\ldots,m_L$ denote the $L$-smallest values out of the $(k+1)$-th-NN distances $(r_{k+1}(x;\Xb_m))_{m=1}^M$, we take the partial average of the corresponding regression estimates:
\[
\selectreg{k}{M}{L}(x)\defeq\selectreg{k}{M}{L}(x;\partitions)\defeq \frac{1}{L}\sum_{j=1}^L \nnreg{k}(x;\Dc_{m_j}).
\]

We call the resulting rule the \emph{$L$-selective $M$-split $k$-NN (or $(k,M,L)$-NN in short) regression rule} and analogously define the \emph{$(k,M,L)$-NN classifier} $\selectcls{k}{M}{L}(x)$ as the corresponding plug-in classifier, \ie
\[
\selectcls{k}{M}{L}(x)\defeq \ones\Bigl(\selectreg{k}{M}{L}(x)\ge \half\Bigr).
\]
Intuitively, it is designed to filter out some possible \emph{outliers} based on the $(k+1)$-th-NN distances, since a larger $(k+1)$-th-NN distance to the query point likely indicates that the returned estimate from the corresponding group is more unreliable.\footnote{We use the $(k+1)$-th-NN distance instead of $k$-th-NN distance due to a technical reason for classification; see Lemma~\ref{lem:label_concentration_new} in Appendix.
For regression, our analysis remains valid for the $k$-th-NN distance.}
Note that the $(k,M,L)$-NN rules become equivalent to $(k,M)$-NN rules if $L=M$.

We can prove minimax optimality of the refined rules without the extra logarithmic factors, as shown below.
For the sake of conciseness, we present more user-friendly corollaries of the full statements Theorem~\ref{thm:regression_selective} (regression) and Theorem~\ref{thm:classification_rate_margin_selective} (classification) in Appendix.

\begin{corollary}[Regression]%
\label{cor:regression_selective_simple}
Suppose that Assumptions~\ref{assum:doubling_homogeneous}, \ref{assum:holder}, and \ref{assum:boundedness} hold and the support of $\mu$ is bounded.
Then, there exists a fixed constant $c\in(0,1)$ such that 
for $L\defeq \lceil(1-c)M\rceil$, if $kM=\Th(N^{\frac{2\aH}{2\aH+d}})$,
\begin{align*}
\E_{\partitions}\|\selectreg{k}{M}{L}-\eta\|_2
&=O\bigl(N^{-\frac{\aH}{2\aH+d}}\bigr).
\end{align*}
\end{corollary}

\begin{corollary}[Classification]
\label{cor:classification_rate_margin_selective_simple}
Suppose that Assumptions~\ref{assum:smooth} and \ref{assum:margin} hold.
Then, there exists a fixed constant $c\in(0,1)$ such that 
for $L\defeq \lceil(1-c)M\rceil$, if $kM=\Th(N^{\frac{2\a}{2\a+1}})$, we have
\begin{align*}
\E_\partitions[R(\selectcls{k}{M}{L})]-R^*
&=O(N^{-\frac{\a(\b+1)}{2\a+1}})\quad\text{and}\quad
\\
\CIS_N(\selectcls{k}{M}{L})&= O(N^{-\frac{\a\b}{2\a+1}}).
\end{align*}
\end{corollary}

Note that, unlike the previous statements for the $(k,M)$-NN rules, \ie Theorems~\ref{thm:regression} and \ref{thm:classification_rate_margin}, which are only valid for fixed $k$'s, we provide analyses that hold for an arbitrary $k$. This enables a strong claim that $(k,M,L)$-NN rules behave same as $\Th(kM)$-NN rules. 
As predicted by theory, in our Gaussian experiments, the $(k,M,L)$-NN rules exhibited almost same rates as $(k,M)$-NN rules, but with slightly smaller errors ; see Fig.~\ref{fig:toy_mog}.
We summarize the convergence rate guarantees for the classifiers discussed so far in Table~\ref{table:comparison_smooth_classification}. 

\begin{table*}[t]
\small
\centering
    {%
    \begin{tabular}{l c c c}
    \toprule
        Algorithms
         & \makecell[c]{No. splits $M$} & \makecell[c]{Base $k$} & Convergence rates
         \\
    \midrule
        \makecell[l]{Standard $k$-NN classifier~\\\citep{Chaudhuri--Dasgupta2014}}  & 1 & $\Theta( N^{\frac{2\a}{2\a+1}})$ & $O(N^{-\frac{\a(\b+1)}{2\a+1}})$\\
        \midrule
        \multirow{ 2}{*}{\makecell[l]{Big $k$-NN classifier $(\gamma\in(0,\frac{2\a}{2\a+1}))$~\\\citep{Qiao--Duan--Cheng2019}}}
        & $\Theta(N^{\gamma})$ & $\Theta( N^{\frac{2\a}{2\a+1}-\gamma})$ & $O(N^{-\frac{\a(\b+1)}{2\a+1}})$\\
        & $\Theta(N^\gamma)$
        & $\Th(1)$
        & $O(N^{-\frac{\gamma(\b+1)}{2}})$\\
        \midrule
        \makecell[l]{$(k,M)$-NN classifier} &
        $\Th(N^{\frac{2\a}{2\a+1}})$
        & $\Th(1)$
        & $\bigOtilde(N^{-\frac{\a(\b+1)}{2\a+1}})$
        \\
        \makecell[l]{$(k,M,L)$-NN classifier~(Section~\ref{sec:selective})} &
        \multicolumn{2}{c}{$kM=\Theta(N^{\frac{2\a}{2\a+1}})$}
        & $O(N^{-\frac{\a(\b+1)}{2\a+1}})$\\
    \bottomrule
    \end{tabular}}
\caption{Summary of the choices of parameters $k$ and the number $M$ of splits with respect to the size $N$ of the entire data, for minimax optimal classification under an $(\a,A)$-smooth conditional probability $\eta$ (Assumption~\ref{assum:smooth}) in $(\Xc,\rho,\mu)$ with Tsybakov margin condition (Assumptions~\ref{assum:margin}).
Note that the choice of the growing $k$ for the big $k$-NN classifier is suggested by \citet{Qiao--Duan--Cheng2019}. 
When $k=1$, the big $k$-NN classifier and $(k,M)$-NN classifier become equivalent, and our tightened analysis shows that the $(k,M)$-NN classifier for any fixed $k$ is nearly minimax optimal up to polylogarithmic factors, and the $(k,M,L)$-NN classifier is even exactly minimax optimal.
}\label{table:comparison_smooth_classification}
\end{table*}

A practitioner may wonder about a theoretically suggested range of the truncation factor $c$ to guarantee the convergence rates when $k$ is fixed. 
In Fig.~\ref{fig:maximum_allowed_factor}, we visualize the maximum allowed selection ratio $1-c\approx\frac{L}{M}$ for each fixed $k$ to guarantee the established convergence rates in Corollaries~\ref{cor:regression_selective_simple} and \ref{cor:classification_rate_margin_selective_simple}.
\begin{figure}[htb]
\centering
\includegraphics[width=0.45\textwidth]{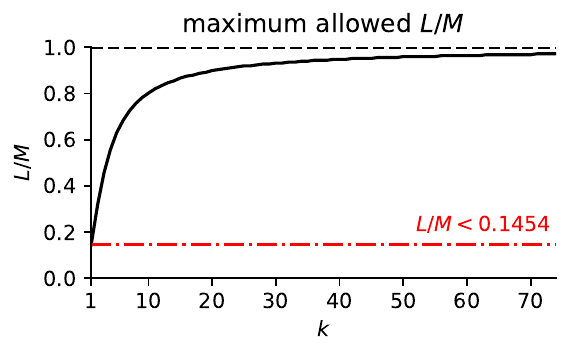}
\caption{Maximum allowed ratio $\frac{L}{M}$ indicated by our theory for different $k$'s, when $k$ is kept fixed. This plot summarizes the information in Fig.~\ref{fig:good_gamma_tau} in Appendix.}
\label{fig:maximum_allowed_factor}
\end{figure}

As one might expect, this plot shows that while only approximately $14.5\%$ of the $M$ batches of $k$-NN information need to be used for the theory to take effect when $k=1$, a larger fraction of the information can be retained and used with a larger $k$, \eg about $80\%$ with $k=10$.
In our experiments, however, we used $c=1/2$ with $k\in \{1,3\}$ with varying $M$, and they also exhibited optimal rates in a synthetic experiment. This suggest that the selection can be done slightly less aggressive in practice.
The rationale behind these number is based on our concentration bound for the distance-selective rules; see Lemma~\ref{lem:key_selective}. 
We refer an interested reader to its proof and the following discussion how the values in Fig.~\ref{fig:maximum_allowed_factor} can be computed.
Unfortunately, there is no closed form expression for the maximum allowed selection ratio.

Finally, we remark that we use the $(k,M,L)$-NN rules as a proof device for analyzing the $(k,M)$-NN rules as alluded to earlier. 
The difficulty in directly analyzing the $(k,M)$-NN rules is that we cannot control possible \emph{outliers} in the NNs from each split of data. 
To circumvent this, we use a $(k,M,L)$-NN rule with a carefully chosen $L$, which rejects possible outliers.
In Appendix, we first analyze the $(k,M,L)$-NN rules as these are more straightforward to analyze, and then present the analyses of the $(k,M)$-NN rules to highlight additional technicalities.

\section{Density Estimation}
\label{sec:density}
For density estimation, we assume $\Xc=\Real^d$ and the Euclidean distance $\rho(x,y)=\|x-y\|_2$ for simplicity, and that the underlying measure $\mu$ has density $p$.
Given data $\datasetX=X_{1:N}=\{X_i\}_{i=1}^N$ drawn \iid from $\mu$, the goal of density estimation is to design a density estimator (or a density estimation procedure) $\ph(\cdot)=\ph(\cdot;\datasetX)\suchthat \Xc\to\Real_+$ based on the data such that the estimate $\ph(x)$ is close to the true density $p(x)$ for any $x\in\Xc$ under a certain criterion, such as the mean squared error (MSE) $\E_{\datasetX}[(\ph(x;\datasetX)-p(x))^2]$.

\subsection{\texorpdfstring{$M$}{M}-Split \texorpdfstring{$k$}{k}-NN Density Estimation Rule}
We now propose a new density estimator based on distributed neighbors which is provably rate-optimal with fixed $k$'s. 
Suppose that we are given a data set $\Xb=X_{1:N}$ drawn \iid from $p(x)$ and randomly split data into disjoint subsets $\Xb_{1:M}=\{\Xb_1,\ldots,\Xb_M\}$, where each subset $\Xb_m=(X_{mi})_{i=1}^{n}$ contains $n$ sample points. Recall that we assume that the total sample size satisfies $N=Mn$.

The proposed estimator is the simple arithmetic average of the $k$-NN density estimators over the data splits, that is,
\begin{align*}
\splitdensity{k}{M}{\AM}(x;\Xb_{1:M})
&\defeq \frac{1}{M}\sum_{m=1}^M \nndensity{k}(x;\Xb_m)
=\frac{1}{M}\sum_{m=1}^M \frac{k-1}{U_k(x;\Xb_m)}.
\numberthis\label{eq:density_am}
\end{align*}
Note that this estimator requires at least $k\ge 2$ to be well-defined, but based on the asymptotic argument for the variance of the $k$-NN density estimator in Section~\ref{sec:intuition_density}, we need $k\ge 3$.

Note that the analysis is straightforward, since the estimator is an average of the truncated $k$-NN estimator: it inherits the same bias of the constituent estimators, which asymptotically vanishes even for a fixed $k\ge 3$, while the variance is $M$ times smaller. 
Hence, this estimator can emulates the growing-$k$ behavior of the standard $k$-NN density estimator by growing $M$ to let its variance vanish.

\subsection{Performance Guarantee}
As remarked by \citet{Singh--Poczos2016}, the standard $k$-NN density estimate $\nndensity{k}(x)$ without truncation is highly biased when $p(x)$ is low. 
For example, \citet{Fukunaga--Hostetler1973,Mack--Rosenblatt1979} showed that for $\sigma$-H\"older smooth densities, 
\begin{align*}
\abs{\E[\nndensity{k}(x)]-p(x)}\asymp \Bigl(\frac{k}{np(x)}\Bigr)^{\frac{\sigma}{d}}.
\end{align*}
Hence, for nonnegative sequences $(\tau_n)_{n\ge 1}$ and $(\nu_n)_{n\ge 1}$, we define a \emph{truncated} version of the density estimator
\begin{align}
\nndensitytr{k}{}(x)
\defeq \nndensitytr{k}{}(x;\Xb)\defeq \nndensity{k}(x;\Xb)\ones_{(\tau_n,\nu_n)}(U_{k}(x;\Xb)).
\nonumber
\end{align}
Here we note that the lower truncation is redundant and can be set always 0 for the estimator in \eqref{eq:density_am}.
However, for a more general treatment of a class of consistent $(k,M)$-NN density estimators in the next section, we will keep the lower truncation in the definition.

Accordingly, we also consider and analyze a truncated version of the $(k,M)$-density estimator defined as follows:
\begin{align}
\splitdensitytr{k}{M}{\AM}(x;\Xb_{1:M})
\defeq \frac{1}{M}\sum_{m=1}^M \nndensitytr{k}{}(x;\Xb_m).
\nonumber
\end{align}

For the convergence rate analysis of the density estimators, we consider a more general notion of $\sigma$-H\"older continuity than Assumption~\ref{assum:holder}, which allows the order $\sigma$ to be greater than 1.
\begin{definition}
\label{def:Holder}
For $\sigma>0$ and $S>0$, a function $h\suchthat\Real^d\to \Real$ is said to be \emph{$(\sigma,S)$-H\"older continuous}
over an open subset $\Omega\subseteq \Real^d$ if
$h$ is continuously differentiable over $\Omega$ up to order $\kappa\defeq \lceil\sigma\rceil-1$ and
\begin{align}
\sup_{\substack{\rb\in\Integer_+^d\\ 
				\abs{\rb}=\kappa}}
\sup_{\substack{y,z\in\Omega\\ 
                y\neq z}}
	\frac{\abs{\partial^{\rb}h(y)-\partial^{\rb}h(z)}}{\norm{y-z}^{\b}}
    \le S,
\nonumber
\end{align}
where $\b\defeq \sigma-\kappa$. 
Here we use a multi-index notation (see, \eg \cite[Ch.~8]{Folland2013}), 
that is, $\abs{\rb}\defeq r_1+\cdots+r_d$ for $\rb\in \Integer_+^d$ and $\partial^\rb h(x)\defeq \partial^\kappa h(x)/(\partial x_1^{r_1}\cdots\partial x_d^{r_d}).$ 
A function $h$ is said to be \emph{locally} $(\sigma,S)$-H\"older continuous, if $h$ is $(\sigma,S)$-H\"older continuous over some open neighborhood of $x$.
\end{definition}

\begin{restatable}{theorem}{ThmAMdensityRate}
\label{thm:agg_density_estimate_l2_rate}
For $x\in\supp(p)$, suppose that $p$ is locally $(\sigma,S)$-H\"older continuous for $\sigma\in(0,2]$.
Then for any fixed $k\ge 3$, $\tau_n=0$, $\nu_n=\Theta((\log n)^{1+\eps})$ for some $\eps>0$, we have, for $\zeta\defeq\frac{\sigma}{d}\wedge 1$, 
\begin{align}
\E_{\Xb_{1:M}}[(\splitdensitytr{k}{M}{\AM}(x;\Xb_{1:M})-p(x))^2]
&=\bigOtilde(n^{-2\zeta}+M^{-1}).
\nonumber
\end{align}
\end{restatable}
In particular, if we set $M=\Theta(N^{\frac{2\zeta}{1+2\zeta}})$, then
\begin{align}
\E[(\splitdensity{k}{M}{\AM}-p(x))^2]
&=\bigOtilde(N^{-\frac{2\zeta}{1+2\zeta}}).
\nonumber
\end{align}
If $d\le\sigma$, then $\zeta=1$, and thus the MSE rate becomes $\bigOtilde(N^{-\frac{2}{3}})$. This happens only if $d\in\{1,2\}$ and $d\le \sigma$.
If $d\ge\sigma$, then $\zeta=\frac{\sigma}{d}$, and thus the MSE rate becomes $\bigOtilde(N^{-\frac{2\sigma}{d+2\sigma}})$, which is minimax optimal for $\sigma$-H\"older smooth densities; see, \eg~\citep{Dasgupta--Kpotufe2014}.

We briefly remark a limitation of this analysis.
The bias rate under H\"older smoothness of order $\sigma>0$ in our analysis is at most $O(n^{-(\frac{\sigma}{d}\wedge 1}))$, which suffers the curse of dimensionality.
Moreover, this analysis cannot adapt to a higher-order smoothness for $\sigma>2$, as we rely on Lemma~\ref{lem:GOVLemma4_smoothing}, which cannot be improved for $\sigma>2$.
In general, this is an inherent limitation of estimation methods based on positive-valued kernels. We refer an interested reader to a more detailed discussion in \citep{Ryu--Ganguly--Kim--Noh--Lee2022} and references therein.

\subsection{Other Variants}
In the previous sections, we show that the $(k,M)$-NN estimator $\splitdensity{k}{M}{\AM}(x)$ enjoys a near minimax optimality under certain regularity conditions. 
However, this estimator requires the base $k$ to be at least 3. 
A natural question is: can we construct another $(k,M)$-NN density estimator, which is provably consistent even for $k=1$?
In this section, we answer the question in the affirmative, by constructing a family of consistent $(k,M)$-NN estimators based on the asymptotic behavior of the $k$-NN statistic $U_{kn}(x)$.
The key idea is that the statistic in the population limit behaves as a Gamma random variable as stated in Proposition~\ref{prop:knndist}, and there are various ways to combine Gamma random variables to relate the target density value with the expectation of the combined random variable.

As a first example, we can consider the following estimator
\begin{align}
\splitdensity{k}{M}{\HM}(x)\defeq \splitdensity{k}{M}{\HM}(x;\Xb_{1:M})
\defeq \frac{kM}{\sum_{m=1}^M U_k(x;\Xb_m)}.
\label{eq:density_hm}
\end{align}
It can be also viewed as a (bias-corrected) \emph{harmonic mean} of the standard $k$-NN density estimates $\{\nndensity{k}(x;\Xb_m)\}_{m=1}^M$, \ie
\[
\splitdensity{k}{M}{\HM}(x;\Xb_{1:M})
=\frac{kM-1}{(k-1)M}\Bigl(\frac{1}{M}\sum_{m=1}^M \nndensity{k}(x;\Xb_m)^{-1} \Bigr)^{-1},
\]
and thus the notation $\splitdensity{k}{M}{\HM}(x;\Xb_{1:M})$.
However, since $k=1$ is not allowed in this expression, it does not provide a correct intuition as $k=1$ is admissible in \eqref{eq:density_hm}.
Rather, this estimator can be justified by the following heuristic asymptotic argument.
For fixed $k$ and $M$, again from Proposition~\ref{prop:knndist}, we observe that 
$U_k(x;\Xb_m)$ converges to $U_{k\infty}^{(m)}(x)$ in distribution as $|\Xb_m|=n\to\infty$ for each $m\in\{1,\ldots,M\}$, where $(U_{k\infty}^{(m)}(x))_{m=1}^M$ are \iid random variables with distribution $\GammaDist(k,p(x))$.
Hence, by Slutsky's theorem, the sum of independent random variables $U_k(x;\Xb_1)+\ldots+U_k(x;\Xb_M)$ converges to the sum of independent Gamma random variables $U_{k\infty}^{(1)}(x)+\ldots+U_{k\infty}^{(M)}(x)\sim \GammaDist(kM,p(x))$ in distribution as $n\to\infty$.
Hence, similar to \eqref{eq:heuristic_bias} and \eqref{eq:heuristic_variance}, we expect to have
\begin{align*}
\lim_{n\to\infty} \E[\splitdensity{k}{M}{\HM}(x;\Xb_{1:M})]
&=\E\Bigl[\frac{kM-1}{U_{k\infty}^{(1)}(x)+\ldots+U_{k\infty}^{(M)}(x)}\Bigr]=p(x),\quad\text{and}
\\
\lim_{n\to\infty} \Var(\splitdensity{k}{M}{\HM}(x;\Xb_{1:M}))
&=\Var\Bigl(\frac{kM-1}{U_{k\infty}^{(1)}(x)+\ldots+U_{k\infty}^{(M)}(x)}\Bigr)=\frac{p(x)^2}{kM-2}.
\end{align*}
Hence, even if $k\ge 1$ is fixed, the proposed estimator is expected to be consistent as long as $M\to\infty$.

We can also consider a \emph{geometric-mean} version:
\begin{align}
\splitdensity{k}{M}{\GM}(x;\Xb_{1:M})
\defeq e^{\digamma(k)} \Bigl(\prod_{m=1}^M \frac{1}{U_k(x;\Xb_m)} \Bigr)^{\frac{1}{M}},
\nonumber
\end{align}
where $\digamma(x)$ denotes the digamma function~\citep{Korn--Korn2000}.
This design is heuristically justified as follows: since
$U_{kn}^{\infty}(x)\sim\GammaDist(k,p(x))$ and $\E[\log U_{kn}^\infty(x)]=\digamma(k)-\log p(x)$,
by the weak law of large number,
\[
\log\splitdensity{k}{M}{\GM}(x;\Xb_{1:M})
=\digamma(k)-\frac{1}{M}\sum_{m=1}^M \log U_{k\infty}^{(m)}(x)
\to \log p(x)
\]
in probability as $M\to\infty$, and by the continuity mapping theorem, $\splitdensity{k}{M}{\GM}(x;\Xb_{1:M})$ also converges to $p(x)$ in probability.

Indeed, the three, AM, GM, and HM, estimators constructed above can be understood in a unified way, borrowing the inverse Laplace transform framework from \citep{Ryu--Ganguly--Kim--Noh--Lee2022}.
Suppose that we wish to estimate a function of density value $f(p(x))$ given a one-to-one function $f\suchthat \Real_+\to\Real$ with fixed $k\ge 1$.
For example, the logarithmic function $f(p)=\log p$, polynomial functions $f(p)=p^{\a-1}$ ($\a\neq 1$), and the exponential function $f(p)=e^{-\b p}$ ($\b\neq 0$) are canonical examples; see Table~\ref{table:estimator_functions_single}.

\begin{table}[htb]
\small
\centering
\begin{tabular}{l l l}
\toprule
\makecell[l]{$f(p)$}
& \makecell[l]{$\displaystyle \phi_k^{(f)}(u)$} & Condition\\
\midrule
$\displaystyle
\frac{p^{\a}}{\a} ~(\a \neq 0)$
& \makecell[l]{$\displaystyle \frac{\Gamma(k)}{\Gamma(k-\a)} \frac{u^{-\a}}{\a}$
} & $k>\a$ \\
$\displaystyle
\log{p}$
& $\displaystyle -\log u+\digamma(k)$
& $k\ge 1$\\
$e^{-\b p}~ (\b>0)$
& \makecell[l]{$\displaystyle \Bigl(1-\frac{\b}{u}\Bigr)^{k-1}\ones_{[\b,\infty)}(u)$%
} & \makecell[l]{
$k\ge 1$
}\\
\bottomrule
\end{tabular}
\caption{Canonical examples of monotone functions $f(p)$ and their estimator functions $\phi_{k}(u)$; see \citep[Table~I]{Ryu--Ganguly--Kim--Noh--Lee2022} for other examples.
In the first row, $\digamma(\a)$ denotes the digamma function~\citep{Korn--Korn2000}. The last column specifies a sufficient condition of $k$ for the inverse Laplace transform in the definition of $\phi_k^{(f)}(u)$ to be well-defined.}
\label{table:estimator_functions_single}
\end{table}

Suppose that we are given $M$ independent Gamma random variables $U_{k\infty}^{(m)}\sim \GammaDist(k, p)$ for $m=1,\ldots,M$.
If there exists a function $\phi_k^{(f)}$ such that $\E[\phi_k^{(f)}(U_{k\infty}^{(1)})]=f(p)$, we can construct a corresponding $(k,M)$-NN density estimator as
\[
\splitdensity{k}{M}{(f)}(x;\Xb_{1:M})
\defeq
f^{-1}\Bigl(\frac{1}{M}\sum_{m=1}^M \phi_k^{(f)}(U_k(x;\Xb_m))\Bigr).
\]
\citet{Ryu--Ganguly--Kim--Noh--Lee2022} showed that 
such a function $\phi_k^{(f)}$, which is called the \emph{estimator function of order $k$ for the function $f$}, can be defined by the following formula:
\[
\phi_k^{(f)}(u)=\frac{\Gamma(k)}{u^{k-1}}\Lc^{-1}\Bigl\{\frac{f(p)}{p^k}\Bigr\}(u)
\quad \text{for }u>0,
\]
where $\Lc^{-1}(F(p))(u)$ denotes the inverse Laplace transform of a function $p\mapsto F(p)$.
It is immediate to check that the desired relation holds:
\begin{align*}
\E[\phi_k^{(f)}(U_{k\infty}^{(1)})]
&=\int_0^\infty\phi_k^{(f)}(u)\frac{p^k}{\Gamma(k)}u^{k-1}e^{-up}\diff u \\
&= \int_0^\infty \Bigr(\frac{\Gamma(k)}{u^{k-1}}
\Lc^{-1}\Bigl\{\frac{f(p)}{p^k}\Bigr\}(u) \Bigr)\frac{p^k}{\Gamma(k)}u^{k-1}e^{-up}\diff u \\
&=p^k \int_0^\infty e^{-up}
\Lc^{-1}\Bigl\{\frac{f(p)}{p^k}\Bigr\}(u)\diff u\\
&= f(p).
\end{align*}
Note that the AM, HM, and GM estimators introduced above are special instances of this general estimator for $f(p)=p$, $f(p)=p^{-1}$, and $f(p)=\log p$, respectively.
To handle all the three cases in a unified way, we define a function $f_\a\defeq\Real_+ \to \Real$ as
\begin{align}
\label{eq:poly_log_alpha_ftn}
f_\a(p)\defeq \begin{cases}
\frac{p^\a}{\a} & \text{if }\a\neq 0,\\
\log p & \text{if }\a=0.
\end{cases}
\end{align}
Note that for $p>0$, by L'H\^opital's rule, $\lim_{\a\to 0} f_\a(p) =f_0(p)$.

We can now state a more general guarantee on the convergence rate, which subsumes Theorem~\ref{thm:agg_density_estimate_l2_rate} as a corollary when $\a=1$.
For a technical reason, we provide a guarantee on the MSE for estimating $f(p(x))$ instead of the density value itself.
Similar to Theorem~\ref{thm:agg_density_estimate_l2_rate}, we consider the truncated version of the estimator for $f(p(x))$, which is defined as
\begin{align*}
\splitftr{k}{M}(x;\Xb_{1:M})
\defeq \frac{1}{M} \sum_{m=1}^M \phi_k^{(f)}(U_k(x;\Xb_m))\ones_{(\tau_n,\nu_n)}(U_{k}(x;\Xb_m)).
\end{align*}

\begin{restatable}{theorem}{ThmKnnFunctionOfDensityMSE}
\label{thm:function_of_density_knn_estimator_mse}
For $x\in\Xc$, assume that $p$ is locally $(\sigma,S)$-H\"older continuous for some $\sigma\in(0,2]$ at $x$.
Let $\bar{\sigma}_d\defeq \frac{\sigma}{d}$ denote the normalized order of smoothness.
For $\a\in\Real$, consider the function $f(p)=f_\a(p)$ defined in \eqref{eq:poly_log_alpha_ftn}.
For $k>2\a$ fixed, set $\nu_n=\Theta((\log n)^{1+\eps})$ for an arbitrary $\eps>0$.
Set $\tau_n=0$ if $\bar{\sigma}_d\ge \a-1$ and $\tau_n=\Th(n^{-\frac{\bar{\sigma}_d}{k-\bar{\sigma}_d-1}})$ otherwise.
If we define, $\zeta=\bar{\sigma}_d\wedge \frac{k-\a}{k-\bar{\sigma}_d-1} \wedge 1$,
we have
\begin{align*}
\abs{\E_{\Xb}[\splitftr{k}{M}(x;\Xb_{1:M})]-f(p(x))}
&= \bigOtilde\bigl(n^{-2\zeta} +M^{-1}\bigr).
\end{align*}
\end{restatable}

The optimal choice of $M$ remains the same $M=\Theta(N^{\frac{2\zeta}{1+2\zeta}})$ as before, which leads to the rate $\bigOtilde(N^{-\frac{2\zeta}{1+2\zeta}})$.
For the special cases $\a\in\{-1,0,1\}$, since $\sigma>0$, the bias rate exponent becomes $\zeta=\bar{\sigma}_d\wedge 1$ and no truncation is needed, \ie we can always set $\tau_n=0$.

\section{Related Work}
\label{sec:related}
The asymptotic-Bayes consistency and convergence rates of the $k$-NN classifier have been studied extensively in the last century~\citep{Fix--Hodges1951,Cover--Hart1967,Cover1968a,Cover1968b,Wagner1971,Fritz1975,Gyorfi1981,Devroye--Gyorfi--Krzyzak--Lugosi1994,Kulkarni--Posner1995}. 
More recent theoretical breakthroughs include a strongly consistent margin regularized 1-NN classifier~\citep{Kontorovich--Weiss2015}, a universally consistent sample-compression based $1$-NN classifier over a general metric space~\citep{Kontorovich--Sabato--Weiss2017,Hanneke--Kontorovich--Sabato--Weiss2020,Gyorfi--Weiss2021}, nonasymptotic analysis over Euclidean space~\citep{Gadat--Klein--Marteau2016} and over a doubling space~\citep{Dasgupta--Kpotufe2014}, optimal weighted schemes~\citep{Samworth2012}, stability~\citep{Sun--Qiao--Cheng2016}, robustness against adversarial attacks~\citep{Wang--Jha--Chaudhuri2018,Bhattacharjee--Chaudhuri2020}, and optimal classification with a query-dependent $k$~\citep{Balsubramani--Dasgupta--Freund--Moran2019}.
For NN-based regression~\citep{Cover1968a,Cover1968b,Dasgupta--Kpotufe2014,Dasgupta--Kpotufe2019}, we mostly extend the analysis techniques of \citep{Xue--Kpotufe2018,Dasgupta--Kpotufe2019}; we refer the interested reader to a recent survey of \citet{Chen--Shah2018} for more refined analyses. 
For a more comprehensive treatment on the $k$-NN based procedures, see \citep{Devroye--Gyorfi--Lugosi1996,Biau--Devroye2015} and references therein.

The most closely related work, especially considering the portion of this paper on the $(k,M)$-NN classifier, is \citep{Qiao--Duan--Cheng2019} as mentioned above. 
In a similar spirit, \citet{Duan--Qiao--Cheng2020} analyzed a distributed version of the optimally weighted NN classifier of \citet{Samworth2012}.
More recently, \citet{Liu--Xu--Shang2021} studied a distributed version of an adaptive NN classification rule of \citet{Balsubramani--Dasgupta--Freund--Moran2019}.

The idea of an ensemble predictor for enhancing statistical power of a base classifier has been long known and extensively studied; see, \eg \citep{Hastie--Tibshirani--Friedman2009} for an overview. Among many ensemble techniques, bagging~\citep{Breiman1996} and pasting~\citep{Breiman1999} are closely related to this work. 
The goal of bagging is, however, mostly to improve accuracy by reducing variance when the sample size is small and the bootstrapping step is computationally demanding in general; see \citep{Hall--Samworth2005,Biau--Cerou--Guyader2010} for the properties of bagged 1-NN rules.
The motivation and idea of pasting is similar to the $(k,M)$-NN rules, but pasting iteratively evolves an ensemble classifier based on an estimated prediction error based on random subsampling rather than splitting samples.
The $(k,M)$-NN rules analyzed in this paper are non-iterative and NN-based-rules-specific, and assume essentially no additional processing step beyond splitting and averaging.

Beyond ensemble methods, there are other attempts to make NN based rules scalable based on quantization~\citep{Kontorovich--Sabato--Weiss2017,Gottlieb--Kontorovich--Nisnevitch2018,Kpotufe--Verma2017,Xue--Kpotufe2018,Hanneke--Kontorovich--Sabato--Weiss2020} or regularization~\citep{Kontorovich--Weiss2015}, where the common theme there is to carefully select subsample and/or preprocess the labels. 
We remark, however, that they typically involve onerous and rather complex preprocessing steps, which may not be suitable for large-scale data.
Approximate NN (ANN) search algorithms~\citep{Indyk1998approximate,Slaney--Casey2008,Har-Peled--Indyk--Motwani2012} are yet another practical solution to reduce the query complexity, but ANN-search-based rules such as \citep{Alabduljalil--Tang--Yang2013,Anastasiu--Karypis2019} hardly have any statistical guarantee~\citep{Dasgupta--Kpotufe2019} with few exception~\citep{Gottlieb--Kontorovich--Krauthgamer2014,Efremenko--Kontorovich--Noivirt2020}.
\citet{Gottlieb--Kontorovich--Krauthgamer2014} proposed an ANN-based classifier for general doubling spaces with generalization bounds. More recently, \citet{Efremenko--Kontorovich--Noivirt2020} proposed a locality sensitive hashing~\citep{Datar--Immorlica--Indyk--Mirrokni2004} based classifier with Bayes consistency but a strictly suboptimal rate guarantee in $\Real^d$.
In contrast, this paper focuses on \emph{exact}-NN-search based algorithms.

There exists a seeming connection between the proposed distance-selective aggregation and the $k$-NN based outlier detection methods.
\citet{Ramaswamy--Rastogi--Shim2000} and \citet{Angiulli--Pizzuti2002} proposed to use the $k$-NN distance, or some basic statistics such as mean or median of the $k$-NN distances to a query point, as an outlier score; a recent paper by \citet{Gu--Akoglu--Rinaldo2019} analyzed these schemes.
In view of this line of work, the $(k,M,L)$-NN classification and regression rules can be understood as a selective ensemble of \emph{inliers} based on the $k$-NN distances.
It would be an interesting direction to investigate a NN-based outlier detection method for large-scale dataset, extending the idea of the distance-selective aggregation.

The NN-based density estimation approach has
a long history that started from the seminal work by \citet{Loftsgaarden--Quesenberry1965}, which established the consistency of the standard $k$-NN estimator for continuous densities.
\citet{Mack--Rosenblatt1979} analyzed the pointwise bias and variance of the estimator, in a similar spirit to what are established in this paper.
A stronger consistency such as asymptotic normality~\citep{Moore--Yackel1977} and strong uniform consistency~\citep{Devroye--Wagner1977} were also established in the early developments. Recently, \citet{Biau--Chazal--Cohen-Steiner--Devroye--Rodriguez2011} established the asymptotic normality of a weighted version of the estimator, and \citet{Dasgupta--Kpotufe2014} proved a high-probability convergence rate.
To the authors' knowledge, the NN-based density estimation has not been considered in a distributed learning setup.
For the analysis, the current work borrows the tool from a recent literature on fixed-$k$-NN-based density functional estimation, especially that of \citep{Ryu--Ganguly--Kim--Noh--Lee2022}.

\section{Experiments}
\label{sec:exp}
In this section, we numerically test the performance of the proposed $(k,M)$-NN rules for regression, classification, and regression.

\paragraph{Computing resources.}
For each experiment, we used a single machine with one of the following CPUs: (1) {Intel(R) Core(TM) i7-9750H CPU \@ 2.60GHz} with 12 (logical) cores or (2) {Intel(R) Xeon(R) CPU E5-2680 v4 @ 2.40GHz} with 28 (logical) cores.

\paragraph{Implementation.}
All implementations were based on Python 3.8 and we used the NN search algorithms implemented in {scikit-learn} package~\citep{scikit-learn} ver. 0.24.1 and utilized the multiprocessors using the python standard package \texttt{multiprocessing}.
The code can be found at \href{https://github.com/jongharyu/split-knn-rules}{\texttt{https://github.com/jongharyu/split-knn-rules}}.

\subsection{Classification and Regression}

We first present simulated convergence rates of the $(k,M)$-NN classification and regression rules for small $k$, say $k\in\{1,3\}$, are polynomial as predicted by theory with synthetic dataset. 
We then demonstrate that their practical performance is competitive against that of the standard $k$-NN rules with real-world datasets, while generally reducing both validation complexity for model selection and test complexity. 
In both experiments, we also show the performance of the $(k,M,\frac{M}{2})$-NN rules to examine the effect of the distance-selective aggregation.\footnote{As alluded to earlier, we used $k$-th-NN distance in experiments for the distance-selective classification rule instead of $(k+1)$-th-NN distance for simplicity.}

\subsubsection{Simulated Dataset}
\label{sec:exp_synthetic}
We first evaluated the performance of the proposed classifier with a synthetic data following \citet{Qiao--Duan--Cheng2019}.
We consider a mixture of two isotropic Gaussians of equal weight $\half \Nc(\mathbf{0},I_d) + \half\Nc(\mathbf{1},\sigma^2 I_d)$, where $\mathbf{1}\defeq[1,\ldots,1]^T\in\Real^d$ and $I_d\in\Real^{d\times d}$ denotes the identity matrix.
With $d=5$, we tested 3 different values of $\sigma\in\{0.5,2,3\}$ with 5 different sample sizes $N\in\{500,2500,12500,62500,312500\}$.
We evaluated the $(k,M)$-NN rule and $(k,M,\frac{M}{2})$-NN rule for $k\in\{1,3\}$ with $M=10N^{\frac{2\aH}{2\aH+d}}=10N^{\frac27}$ based on $\aH=1$ and $d=5$.
For comparison, we also ran the standard $k$-NN algorithm with $k\in\{1,3,10N^{\frac{7}{2}}\}$.
We repeated experiments with 10 different random seeds and reported the medians and (20\%,80\%) quantiles.

The excess risks are plotted in Figure~\ref{fig:toy_mog}.
We note that the $(1,M)$-NN classifier performs similarly to the baseline $k$-NN classifier across different values of $\sigma$, and the performance can be improved by the $(1,M,\frac{M}{2},M)$-NN classifier. This implies that discarding possibly noisy information in the aggregation could actually improve the performance of the ensemble classifier.
Note also that the convergence of the excess risks of the standard $M$-NN classifier and the $(k,M)$-NN classifiers with $k\in\{1,3\}$ is polynomial (indicated by the straight lines), as predicted by theory.

\begin{figure*}[htb]
\centering
\includegraphics[width=\textwidth]{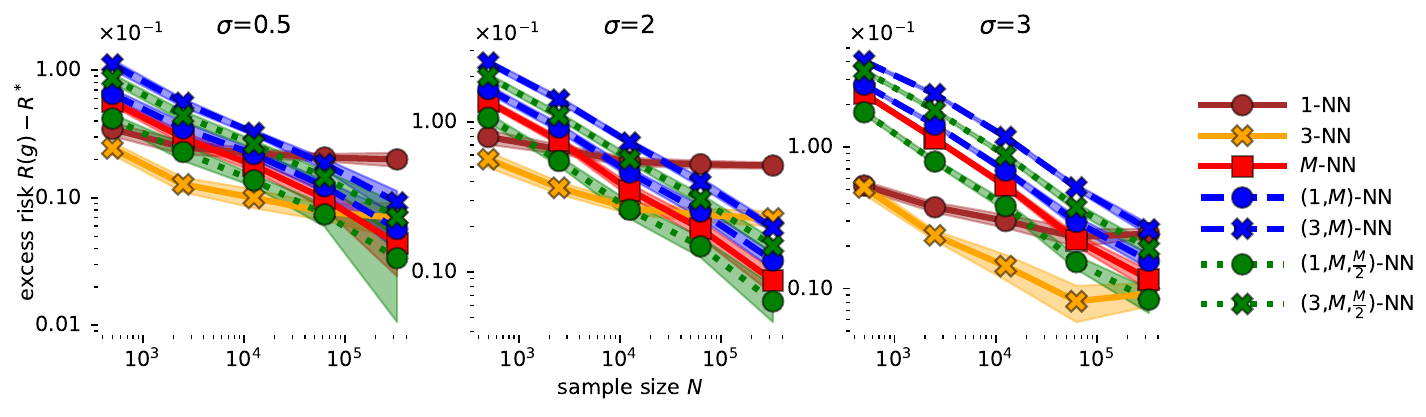}  %
\caption{Summary of the excess risks of the NN classifiers for the mixture of two Gaussians experiments in Section~\ref{sec:exp_synthetic}.}
\label{fig:toy_mog}
\end{figure*}

\subsubsection{Real-World Datasets}

We evaluated the proposed rules with publicly available benchmark datasets from the UCI machine learning repository~\citep{UCI2019} and the OpenML repository~\citep{OpenML2013}, which were also used in \citep{Xue--Kpotufe2018} and \citep{Qiao--Duan--Cheng2019}; see Table~\ref{table:dimension} in Appendix for size, feature dimensions, and the number of classes of each dataset. All data were standardized to have zero mean and unit variances.%

\begin{table}[htb]
\centering
\scriptsize
\begin{adjustbox}{center}
\begin{tabular}{l r r r}
\toprule
Data set & \# training & \# dim. & \# class. \\
\midrule
\makecell[l]{\href{https://archive.ics.uci.edu/ml/datasets/Gisette}{GISETTE}~\citep{Guyon--Gunn--Ben-Hur--Dror2004}} 
    & 7k & 5k & 2\\
\makecell[l]{\href{https://archive.ics.uci.edu/ml/datasets/HTRU2}{HTRU2}~\citep{Lyon--Stappers--Cooper--Brooke--Knowles2016}} 
    & 18k & 8 & 2\\
\makecell[l]{\href{https://archive.ics.uci.edu/ml/datasets/Credit+Approval}{Credit}~\citep{UCI2019}} 
    & 30k & 23 & 2\\
\makecell[l]{\href{https://archive.ics.uci.edu/ml/datasets/MiniBooNE+particle+identification}{MiniBooNE}~\citep{UCI2019}} 
    & 130k & 50 & 2\\
\makecell[l]{\href{https://archive.ics.uci.edu/ml/datasets/SUSY}{SUSY}~\citep{Baldi--Sadowski--Whiteson2014}} 
    & 5000k & 18 & 2 \\
\makecell[l]{\href{https://www.openml.org/t/7356}{BNG(letter,1000,1)}~\citep{OpenML2013}} 
    & 1000k &  17 & 26\\
\midrule
\makecell[l]{\href{https://archive.ics.uci.edu/ml/datasets/YearPredictionMSD}{YearPredictionMSD}~\citep{UCI2019}} 
    & 463k & 90 & 1\\
\bottomrule
\end{tabular}
\end{adjustbox}
\caption{Summary of dimensions of the benchmark datasets. }
\label{table:dimension}
\end{table}

We tested four algorithms.
The first two algorithms are 
(1) the standard 1-NN rule and
(2) the standard $k$-NN rule with 10-fold cross-validation (CV) over an exponential grid $k\in \Kc\defeq\{2^l-1\suchthat 2\le l\le \log_2(\min\{2^{10}, 1+N_{\mathsf{train}}/25\})\}$, where $N_{\mathsf{train}}$ denotes the size of training data.
The rest are (3) the $(1,M)$-NN rule and (4) the $(1,M,\frac{M}{2})$-NN rule both with 10-fold CV over $M\in\Kc$. 
We repeated with 10 different random (0.95,0.05) train-test splits and evaluated first $\min\{N_{\mathsf{test}},1000\}$ points from the test data to reduce the simulation time.
Table~\ref{table:results} summarizes the test errors, test times, and validation times.\footnote{Here, we used a KD-Tree based NN search by default. Since, however, a KD-Tree based algorithm suffers a curse of dimensionality (recall Section~\ref{sec:computation}), we ran additional trials with a brute-force search for high-dimensional datasets $\{\text{GISETTE, YearPredictionMSD}\}$, whose feature dimensions are 5000 and 90, respectively, and report the time complexities in the subsequent rows.}
The optimal $(1,M)$-NN rules consistently performed as well as the optimal standard $k$-NN rules, even running faster than the standard 1-NN rules in the test phase. 
We remark that the optimally tuned $(1,M,\frac{M}{2})$-NN rules (\ie with the distance-selective aggregation) performed almost identical to the $(1,M)$-NN rules, except slight error improvements observed in high-dimensional datasets $\{\text{GISETTE, YearPredictionMSD}\}$.

\begin{table*}[htb]
\small
    \centering
    \begin{adjustbox}{center}
    \resizebox{\textwidth}{!}{
    \begin{tabular}{l r r r r r r r r}
    \toprule
    \multirow{2}{.3in}{Dataset} & 
    \multicolumn{3}{c}{Error (\% for classification)} & \multicolumn{3}{c}{Test time (s)} & \multicolumn{2}{c}{Valid. time (s)} \\
    \cmidrule(lr){2-4}
    \cmidrule(lr){5-7}
    \cmidrule(lr){8-9}
        & 1-NN & $k$-NN & $(1,M)$-NN 
        & 1-NN & $k$-NN & $(1,M)$-NN
        & $k$-NN & $(1,M)$-NN \\
    \midrule
    \makecell[l]{\href{https://archive.ics.uci.edu/ml/datasets/Gisette}{GISETTE}
    } 
        & \meanstd{7.26}{1.65} 
            & \meanstd{\textbf{4.54}}{0.93} 
            & \meanstd{\textbf{5.11}}{1.01}
            (\meanstd{\textbf{4.86}}{0.86})
        & 6.13
            & \textbf{5.75}
            & 6.79 (6.18)
        & \textbf{52} & 262 (270) \\
    \qquad w/ brute-force %
        & - %
            & - %
            & - %
        & 0.30
            & \textbf{0.26}
            & 1.20 (2.06)
        & \textbf{38} & 200 (207) \\
    \makecell[l]{\href{https://archive.ics.uci.edu/ml/datasets/HTRU2}{HTRU2}
    } 
        & \meanstd{2.91}{0.40} & \meanstd{\textbf{2.18}}{0.44} & \meanstd{\textbf{2.08}}{0.28} (\meanstd{\textbf{2.28}}{0.37})
        & 0.18 & 0.18 & \textbf{0.04} (\textbf{0.04})
        & 18 & \textbf{8} (10)\\
    \makecell[l]{\href{https://archive.ics.uci.edu/ml/datasets/Credit+Approval}{Credit}
    } 
        & \meanstd{26.73}{0.99} & \meanstd{\textbf{18.68}}{1.01} & \meanstd{\textbf{18.65}}{1.05} (\meanstd{\textbf{18.93}}{0.95})
        & 0.85 & 1.2 & \textbf{0.2} (\textbf{0.2})
        & 122 & \textbf{25} (29)\\
    \makecell[l]{\href{https://archive.ics.uci.edu/ml/datasets/MiniBooNE+particle+identification}{MiniBooNE}
    } 
        & \meanstd{13.72}{1.57} & \meanstd{\textbf{10.63}}{0.76} & \meanstd{\textbf{10.69}}{0.86} (\meanstd{\textbf{10.62}}{0.64})
        & 1.68 & 2.42 & \textbf{0.98} (\textbf{0.94})
        & 264 & \textbf{88} (92)\\
    \makecell[l]{\href{https://archive.ics.uci.edu/ml/datasets/SUSY}{SUSY}
    } 
        & \meanstd{28.27}{1.50} 
            & \meanstd{\textbf{20.32}}{1.04} 
            & \meanstd{\textbf{20.55}}{1.35} (\meanstd{\textbf{20.52}}{1.31})
        & 32
            & 35
            & \textbf{14} (\textbf{13})
        & 3041 & \textbf{1338} (1362) \\
    \makecell[l]{\href{https://www.openml.org/t/7356}{BNG(letter,1000,1)}
    } 
        & \meanstd{46.13}{1.18} & \meanstd{\textbf{40.88}}{1.12} & \meanstd{\textbf{41.53}}{1.04} (\meanstd{\textbf{40.72}}{0.78})
        & 379 & 350 & 17 (\textbf{14})
        & 2868 & \textbf{619} (959)\\
    \midrule
    \makecell[l]{\href{https://archive.ics.uci.edu/ml/datasets/YearPredictionMSD}{YearPredictionMSD}
    } 
        & \meanstd{7.22}{0.34} & \meanstd{\textbf{6.72}}{0.25} & \meanstd{\textbf{6.79}}{0.22} (\meanstd{\textbf{6.75}}{0.27})
        & 33 & \textbf{31} & 40 (34)
        & 1616 & 431 (\textbf{412})\\
    \qquad w/ brute-force %
        & - %
            & - %
            & - %
        & 15
            & 18
            & \textbf{3.5} (\textbf{3.6})
        & 1529 & \textbf{300} (336) \\
    \bottomrule
    \end{tabular}
    }
    \end{adjustbox}
    \caption{Summary of experiments with benchmark datasets.
    YearPredictionMSD in the last row is a regression dataset.
    Recall that $(1,M)$-NN is a shorthand for the $M$-split 1-NN rules. 
    The values in the parentheses correspond to the $(1, M, \frac{M}{2})$-NN rules.
    The best values are highlighted in bold.}
    \label{table:results}
\end{table*}

Finally, we present Figure~\ref{fig:validation}, which summarizes the validation error profiles from the 10-fold CV procedures for the standard $k$-NN and $(1,M)$-NN classifiers. This shows that the optimal choices of $M$ are always larger, but consistently within a factor from the optimal $k$'s.

\begin{figure*}[!ht]
    \centering
    \includegraphics[width=.8\textwidth]{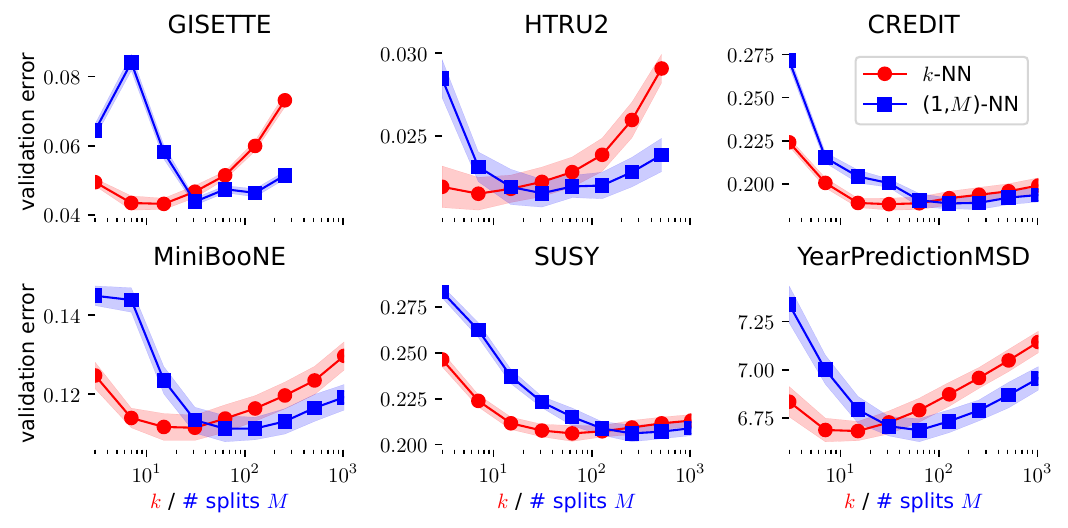}  %
    \caption{Validation error profiles from 10-fold cross validation. Here, as expected, the optimal $M$ chosen for $(1,M)$-NN rules is in the same order of the optimal $k$ for the standard $k$-NN rules.
}
    \label{fig:validation}
\end{figure*}

\newcommand{\kbase}{k_{\mathsf{base}}}
\subsection{Density Estimation}
We now examine the numerical performance of the $(k,M)$-NN density estimation rules.
For the evaluation, we randomly generated mixture of Gaussians (MoG) distributions as follows.
For a given dimension $d\in\{1,2,3,4,5\}$, we randomly generated $10$ centers from the normal distribution $\Nc(\mathbf{0},10I_d)$, and constructed a mixture of the standard normal distributions centered at these centers with equal weights.
We compared the performances of the $K$-NN estimator with the $(\kbase,\frac{K}{\kbase})$-NN (AM, GM, HM) estimators for $\kbase\in\{1,2,3,4,5\}$, except $\kbase\in\{3,4,5\}$ for the AM estimator.
We set $K$ as $K=\lceil \frac{1}{2}N^{\frac{4}{d+4}}\rceil$, which is the optimal choice for the $K$-NN estimator~\citep{Dasgupta--Kpotufe2014}. 
The sample size varied over $\{10^2,10^3,\ldots,10^6\}$, and we generated 1000 independent samples to estimate the mean squared error (MSE) $\E_{X}[(\ph(X;\dataset)-p(X))^2]$.
We repeated the experiments for 10 randomly constructed different MoG distributions.

Fig.~\ref{fig:exp_density} summarizes the MSEs with respect to varying sample sizes. Different dimensions are considered across the rows, and $\kbase$ varies over the columns. The shades indicate (20\%,80\%) quantiles over the 10 random experiments. 
Interestingly, the $(\kbase,\frac{K}{\kbase})$-NN HM estimator performs almost identically to the $K$-NN estimators, and the GM estimator behaves similarly, while the AM estimator performs the worst among the proposals.

\begin{figure}[htb]
\centering
\includegraphics[width=1.\linewidth]{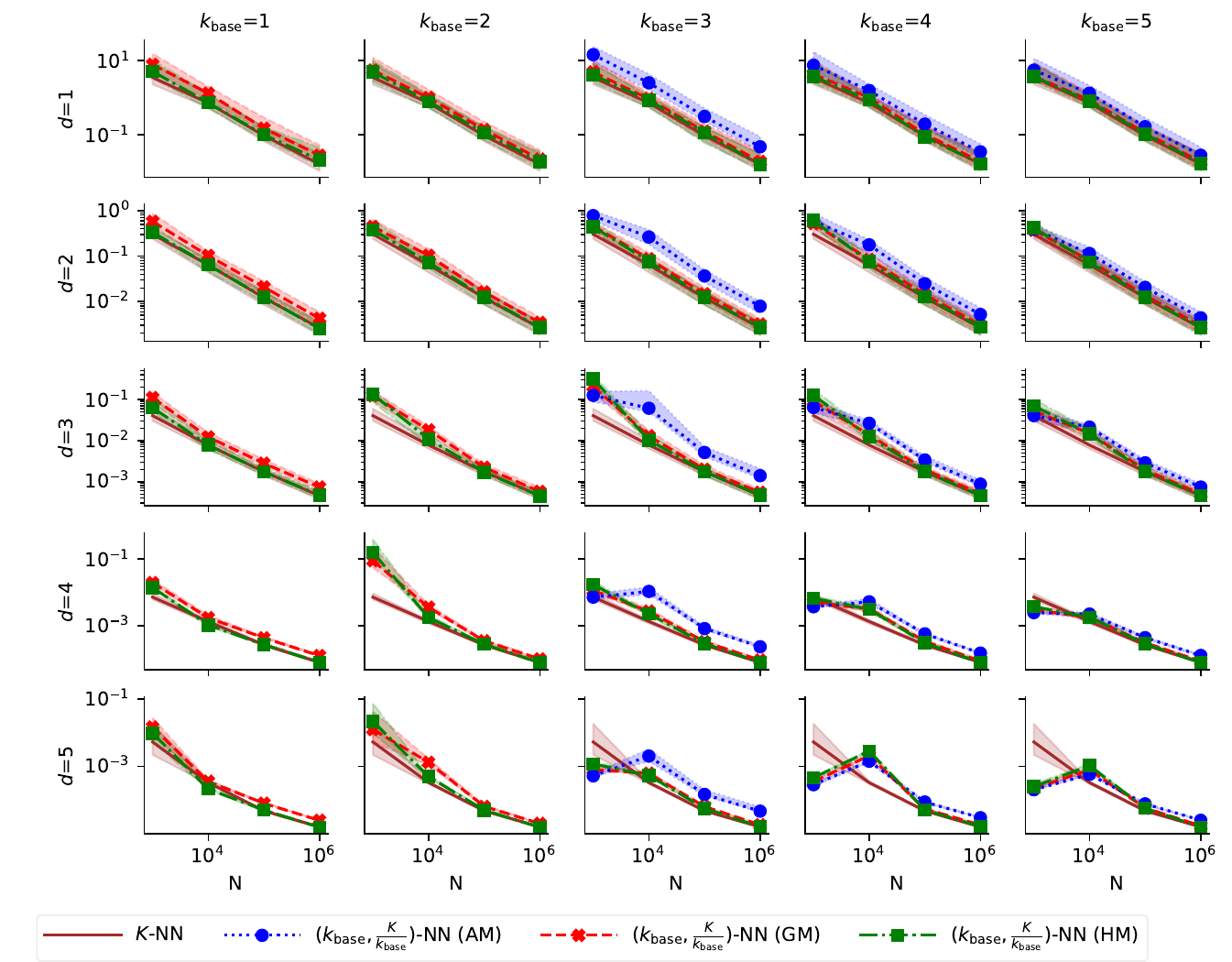}
\caption{Sample size vs. mean squared error plots for $K$-NN and $(k,K/k)$-NN density estimation rules for $k\in\{1,2,3,4,5\}$ over columns for mixture of Gaussian distributions of dimension $d\in \{1,2,3,4,5\}$. Here, $K$ was chosen as $\Th(N^{\frac{2\sigma}{d+2\sigma}})$, where $\sigma=2$ for mixture of Gaussians was plugged in.}
\label{fig:exp_density}
\end{figure}

\section{Concluding Remarks}
\label{sec:conclusion}
In this paper, we established the near statistical optimality of the $(k,M)$-NN rules when $k$ is fixed, which makes the sample-splitting-based NN rules more appealing for practical scenarios with large-scale data. 
For regression and classification, we also showed that the distance-selective rules enjoy exact minimax optimality and exhibited some level of performance boost in the experimental results.
In practice, our work suggests that $k$-NN search algorithms can be optimized only for $k=1$, without a concern of losing statistical efficiency. 
It is an open question whether the logarithmic factor is fundamental for the vanilla $(k,M)$-NN rules or can be removed by a tighter analysis.

As evidenced by both theoretical guarantees and empirical supports in this paper, we believe that the $(k,M)$-NN rules, especially for $k=1$, can be widely deployed in practical systems and deserve further study including an optimally weighted version of the classifier as studied in \citep{Duan--Qiao--Cheng2020}. 
For classification, it would be also interesting if the current divide-and-conquer framework can be modified to be universally consistent for any general metric space, whenever such a consistent rule exists~\citep{Hanneke--Kontorovich--Sabato--Weiss2020,Gyorfi--Weiss2021}.
Establishing a more general and stronger consistency result of the proposed $(k,M)$-NN density estimator is also an interesting research direction.

\acks{The authors appreciate insightful feedback from anonymous reviewers to improve earlier versions of the manuscript.
JR would like to thank Alankrita Bhatt, Sanjoy Dasgupta, Yung-Kyun Noh, and Geelon So for their discussion and comments on the manuscript.
This work was supported in part by the National Science Foundation under Grant CCF-1911238.}

\appendix

\addtocontents{toc}{\protect\StartAppendixEntries}
\listofatoc

\setcounter{figure}{0}
\renewcommand{\thefigure}{\thesection.\arabic{figure}}
\setcounter{table}{0}
\renewcommand{\thetable}{\thesection.\arabic{table}}

\section*{Overview of Appendix}
We provide the full proofs of the statements in the main text.
In Appendix~\ref{app:concentration_distributed_statistics}, we state and prove key technical lemmas for analyzing the distributed regression and regression rules, \ie $(k,M,L)$-NN rules and $(k,M)$-NN rules.

We first analyze the regression rules in Appendix~\ref{app:regression}, as they require less technicalities compared to the case of classification.
The analyses for classification rules then follow.
We remark that our analysis $(k,M,L)$-NN rules aim to handle any regime of $k$ (both fixed and growing), while we focus on the fixed-$k$ regime for the $(k,M)$-NN rules.
Appendix~\ref{app:density} presents the analyses for the proposed $(k,M)$-NN density estimation rules.

\section{Key Lemmas for Regression and Classification Rules}
\label{app:concentration_distributed_statistics}
We first restate a simple yet important observation on the $k$-nearest-neighbors by \citet{Chaudhuri--Dasgupta2014} that the $k$-nearest neighbors of $x$ lies in a ball of probability mass of $O(\frac{k}{n})$ centered at $x$, with high probability.
We define the \emph{probability radius} of mass $p$ centered at $x\in\Xc$ as the minimum possible radius of a closed ball containing probability mass at least $p$, that is,
\[
r_p(x)\defeq \inf\{r>0\suchthat \mu(\closedball(x,r))\ge p\}.
\]

\begin{lemma}[{\citealp[Lemma~8]{Chaudhuri--Dasgupta2014}}]
\label{lem:cd14_lemma8}
Pick any $x\in \Xc$, $0< p\le 1$, $0\le \g< 1$, and any positive integers $n$ and $k$ such that $1\le k\le (1-\g) np$. 
If $X_1,\ldots,X_n$ are drawn \iid from $\mu$, then
\begin{align*}
\P(
r_{k+1}(x;X_{1:n})
>r_p(x)) 
&\le e^{-\frac{\g^2}{2}np}
\le e^{-\frac{\g^2}{2}k}.
\end{align*}
\end{lemma}

\subsection{For \texorpdfstring{$(k,M,L)$}{(k,M,L)}-NN Rules}
We now state an analogous version (Lemma~\ref{lem:key_selective}) of the above lemma for our analysis of the $(k,M,L)$-NN rules. The following lemma quantifies that, with high probability (exponentially in $M$) over the split instances $\partitions_X=\{\Xv_1,\ldots,\Xv_M\}$, the the $k$-nearest neighbors of $x$ from the selected data splits based on the $(k+1)$-th-NN distances will likely lie within a small probability ball of mass $O(\frac{kM}{N})$ around the query point.

In the following, we need to invoke a more refined version of the multiplicative Chernoff bound, as the standard Chernoff bound cannot provide a guarantee for $k=1$.
In what follows, we let $\binomrv_{M,\a}\sim\Binom(M,\a)$ denote a binomial random variable with parameters $M$ and $\a\in[0,1]$.
\begin{lemma}[A refined multiplicative Chernoff bound]
\label{lem:refined_multiplicative_chernoff}
For any $q\in(0,1)$ and $\lambda\in(0,1)$,
the following holds:
\[
\Pr(\binomrv_{n,q}\le \lambda nq)
\le \Bigl(\frac{(1-q)^{1-\lambda q}}{q^{\lambda q}} \bigl(\frac{e}{\lambda}\bigr)^{\lambda q}\Bigr)^n.
\]
\end{lemma}
\begin{proof}
Consider
\begin{align*}
\Pr(\binomrv_{n,q}\le \lambda nq)
&=\sum_{i=0}^{\lfloor\lambda nq\rfloor} \binom{n}{i} q^i(1-q)^{n-i}\\
&\le (1-q)^{n(1-\lambda q)} \sum_{i=0}^{\lceil\lambda nq\rceil} \binom{n}{i}\\
&\le (1-q)^{n(1-\lambda q)} \Bigl(\frac{e}{\lambda q}\Bigr)^{n\lambda q}.\numberthis
\end{align*}
The last inequality follows from a well-known inequality $\sum_{i=0}^d \binom{n}{i}\le (\frac{en}{d})^d$ for $n\ge d$, which 
This concludes the proof.
\end{proof}

The following lemma is a counterpart of Lemma~\ref{lem:cd14_lemma8} for the $(k,M,L)$-NN rules.
\begin{lemma}
\label{lem:key_selective}
Pick any $x\in\Xc$, $0<p\le 1$, $0\le \g<1$, and any positive integers $n$ and $k$ such that $1\le k\le (1-\g)np$.
Further, for any integer $M\ge 1$, pick $\tau\in[0,1)$, so that $L=\lceil (1-\tau)(1-e^{-\frac{\g^2}{2}k}) M\rceil \ge 1$.
If the data splits $\Xb_1,\ldots,\Xb_M$ of size $n$ are drawn \iid from $\mu^{\otimes n}$, we have
\begin{align*}
\Pr\Bigl(\max_{j\in[L]} r_{k+1}(x;\Xb_{m_j})>r_p(x)\Bigr)
&\le \Bigl(q_{k;\g}^{1-\taubar\qbar_{k;\g}}\bigl(\frac{e}{\taubar\qbar_{k;\g}}\bigr)^{\taubar\qbar_{k;\g}} \Bigr)^{M}  
\defeq P_e(k,M;\g,\tau).
\end{align*}
Here, we define $q_{k;\g}\defeq e^{-\frac{\g^2}{2}k}$ and we use a shorthand notation $\bar{x}\defeq 1-x$ for $x\in[0,1]$.

In particular, either $k$ is any fixed integer $k\ge 1$ or growing, the bound is $e^{-O(kM)}$. More precisely:
\begin{enumerate}
\item For any $k\ge 1$ and $\g\in(0,1)$, there exists $(k,\g)\mapsto \tau_{k;\g}^{\inf} \in (0,1)$ such that for any $\tau\in(\tau_{k;\g}^{\inf}, 1)$, we have
\[
P_e(k,M;\g,\tau)
\le e^{-\phi_{k;\g,\tau} M}.
\]
Moreover, for any $\gamma\in(0,1)$, $\tau_{k;\gamma}^{\inf}$ monotonically decreases as $k$ increases.
\item For any $\g\in(0,1)$ and $\tau\in(0,1)$, define $k_{\g,\tau}^{\min}\defeq \frac{2-\tau^2}{\g^2\tau}$.
For any $c>1$, if $k\ge c k_{\g,\tau}^{\min}$, we have
\[
P_e(k,M;\g,\tau)\le e^{-\half(1-\frac{1}{c}) kM}.
\] 
\end{enumerate}
\end{lemma}

\begin{proof}
For each data split indexed by $m\in[M]$, we define a \emph{bad} event 
\[
E_m=\{r_{k+1}(x;\Xb_m)> r_p(x)\}.
\]
Observe that $E_m$ occurs if any only if the closed ball of probability mass $p$ contains less than $k$ points from $\Xb_m$. 
By Lemma~\ref{lem:cd14_lemma8},
the probability of the bad event $E_m$ is upper bounded by $e^{-\frac{\g^2}{2}k}$.
Now, since the data splits are independent, $(1(E_m))_{m=1}^M$ is a sequence of independent Bernoulli random variables with parameter $\Pr(E_1)\le \tau$. 
Hence, we have 
\begin{align*}
\Pr\Bigl(\max_{j\in[L]} r_{k+1}(x;\Xb_{m_j})>r_p(x)\Bigr)
&\le \Pr\Bigl(\sum_{m=1}^{M}1(E_m)> M-L\Bigr)
\\
&= \Pr\Bigl(\sum_{m=1}^{M}1(E_m^c)< L\Bigr)\\
&\le \Pr(\binomrv_{M,1-q} < (1-\tau)(1-q)M)\numberthis\label{eq:intrmdt_binom_bound}
\end{align*}
for $q=e^{-\frac{\g^2}{2}k}$, where $\binomrv_{M,\a}\sim\Binom(M,\a)$ denotes a binomial random variable with parameters $M$ and $\a\in[0,1]$.
We now apply Lemma~\ref{lem:refined_multiplicative_chernoff} for $(n,q,\lambda)\gets (M,1-e^{-\frac{\g^2}{2}k},1-\tau)$, which concludes the desired bound.

To prove the second part of the statement, we need to show that there exists a proper choice of $(\g,\tau)\in(0,1)\times(0,1)$, so that the error probability bound decays exponentially fast as $e^{-O(kM)}$ as $N=nM\to\infty$, regardless of whether $k=O(1)$ or $k\to\infty$ as $N\to\infty$.
We consider the two regimes separately.
\begin{enumerate}
\item If $k=O(1)$ as $N\to\infty$, we need to find a pair of $(\g,\tau)$ such that 
\[
P_e(k,M;\g,\tau)^{\frac{1}{M}}
=(1-\qbar_{k;\g})^{1-\taubar\qbar_{k;\g}}\bigl(\frac{e}{\taubar\qbar_{k;\g}}\bigr)^{\taubar\qbar_{k;\g}} <1,
\]
which is equivalent to
\[
M\log P_e(k,M;\gamma,\tau)
=f(\taubar\qbar_{k;\g}; \log(1-\qbar_{k;\g}))<0, \numberthis \label{eq:good_gamma_tau}
\]
where $f(x;a)\defeq a(1-x)+x(1-\log x)$ for $x\in(0,1]$.
Note that since $x\mapsto f(x;a)$ is strictly concave and its maximum is attained when $x=1$ as $f(1;a)=1$, it monotonically increases over the domain.
Further, since $f(0;a)\defeq \lim_{x\to 0} f(x;a)= a$, for $a=\log(1-\qbar_{k;\g})<0$, by the intermediate value theorem, there must exist a unique root $x_{\max}(a)\in(0,1)$ of the equation $f(x;a)=0$, which further satisfies that $f(x;a)<0$ for $x\in(0,x_{\max}(a))$.
This implies that for any $k\ge 1$ and $\g>0$, if
\[
\tau> 1-\frac{x_{\max}(\log(1-\qbar_{k;\g}))}{\qbar_{k;\g}},
\]
then $f(\taubar\qbar_{k;\g}; \log(1-\qbar_{k;\g}))=\frac{1}{M}\log P_e(k,M;\g,\tau) <0$. 
This implies that for any $k\ge 1$, $\g\in(0,1)$, and $\tau\in(1-\frac{x_{\max}(\log(1-\qbar_{k;\g}))}{\qbar_{k;\g}}, 1)$, we have 
\[
P_e(k,M;\g,\tau)\le e^{-\phi_{k;\g,\tau} M},
\]
with $\phi_{k;\g,\tau}\defeq -f(\taubar\qbar_{k;\g}; \log(1-\qbar_{k;\g}))>0$.

\item If $k\to\infty$ as $N\to\infty$, we can consider the following upper bound on the rate
\begin{align*}
P_e(k,M;\g,\tau) 
&\le 
\Bigl((1-\qbar_{k;\g})^{1-\taubar\qbar_{k;\g}}
e^{1-\half(1-\taubar\qbar_{k;\g})^2} \Bigr)^{M} \\
&= \Bigl(e^{-\frac{\g^2}{2}k(1-\taubar\qbar_{k;\g}) + 1 -\half(1-\taubar\qbar_{k;\g})^2}\Bigr)^{M}\\
&\le \bigl(e^{-\frac{\g^2}{2}k\tau + 1 -\frac{\tau^2}{2}}\bigr)^{M}.
\end{align*}
where the first inequality immediately follows from $(\frac{e}{x})^{x}\le e^{1-\frac{(1-x)^2}{2}}$ for $x\in (0,1]$.
In this case, for a given $\g$ and $\tau$, for any $k\ge K\ge \frac{2-\tau^2}{\g^2\tau}$, we have
\[
P_e(k,M;\g,\tau)\le e^{-\half\{\frac{1}{K}(\tau^2-2) + \g^2\tau\}kM} = e^{-O(kM)}.%
\] 
\end{enumerate}
This concludes the proof.
\end{proof}

A reader may wonder what values of $(\gamma,\tau)$ guarantee the exponential convergence $e^{-O(M)}$ for a fixed $k\ge 1$.
Note that since $L\defeq \lceil (1-\tau)(1-e^{-\frac{\g^2}{2}k}M\rceil$, larger $\gamma$ and/or smaller $\tau$ corresponds to less truncation (or equivalently, larger $L$), and vice versa.
To answer to the question, in Fig.~\ref{fig:good_gamma_tau}, we visualize the function $(\gamma,\tau)\mapsto M\log P_e(k,M;\gamma,\tau)=f(\taubar\qbar_{k;\g}; \log(1-\qbar_{k;\g}))$ in \eqref{eq:good_gamma_tau} for $k\in \{1,3,\ldots,13\}$, as well as the zero-level sets.
Note that for each $\gamma=1-\frac{k}{np}\in(0,1)$, the range of \emph{good} $\tau$'s becomes wider, since intuitively smaller $\tau$'s (less aggressive truncation) are allowed if using a larger base $k$.

\begin{figure*}[htb]
\centering

\subfloat[Visualization of the function $(\gamma,\tau)\mapsto M\log P_e(k,M;\gamma,\tau)=f(\taubar\qbar_{k;\g}; \log(1-\qbar_{k;\g}))$ in \eqref{eq:good_gamma_tau} for $k\in \{1,3,\ldots,13\}$.
The pairs resulting in the negative values (blue region) guarantee the exponential convergence. 
]{\includegraphics[width=\linewidth]{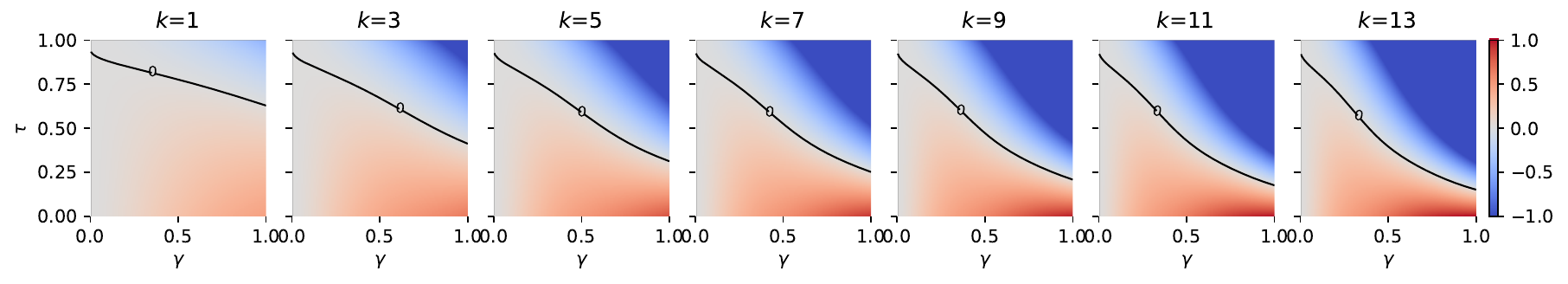}}

\subfloat[Visualization of the selection factor $(1-\tau)(1-e^{-\frac{\gamma^2}{2}k})$, which guarantees the exponential convergence for each $k\in \{1,3,\ldots,13\}$.
]{\includegraphics[width=\linewidth]{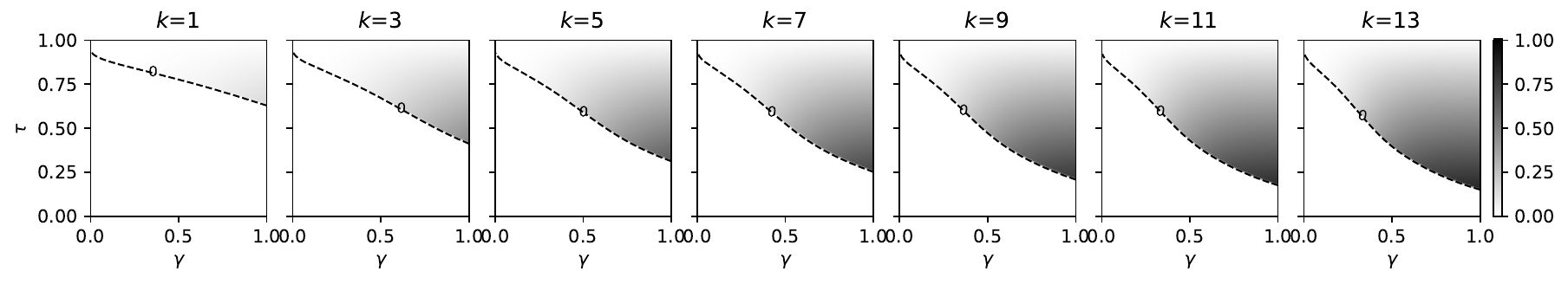}}

\caption{Visualization of (a) the allowed pairs of $(\gamma,\tau)$ for each fixed $k\ge 1$ and (b) the corresponding selection factors. 
The maximum values of the selection factors for each $k$ in (b) are summarized in Fig.~\ref{fig:maximum_allowed_factor}.}
\label{fig:good_gamma_tau}
\end{figure*}

\subsection{For \texorpdfstring{$(k,M)$}{(k,M)}-NN Rules}

In the following, the key idea for analyzing the $(k,M)$-NN rules
is to use a distance-selective $(k,M,L)$-NN rule with $L=\lceil(1-\tau)^2M\rceil$ as a proxy to the $(k,M)$-NN rule for $\tau$ sufficiently small.
For $(k,M)$-NN rules, we are interested in when $k=O(1)$ is fixed.
We need the following variant of Lemma~\ref{lem:key_selective}.
While we will use $q=\tau$ in our analysis for the $(k,M)$-NN rules, we keep $\tau$ and $q$ in the following statement to clarify the role of each variable in its proof.

\begin{lemma}\label{lem:key_original}
For any positive integer $k\ge 1$, pick $\tau\in(0,1]$, $q\in(0,1]$, and set $L=\lceil (1-\tau)(1-q) M\rceil$.
If the data splits $\Xb_1,\ldots,\Xb_M$ of size $n$ are independent, we have
\begin{align*}
\Pr\Bigl(\max_{j\in[L]} r_{k+1}(x;\Xb_{m_j})>r_p(x)\Bigr)
\le e^{-\frac{(1-q)\tau^2}{2}M}
\end{align*}
for $n\ge k+\log\frac{1}{q}+\sqrt{2k\log\frac{1}{q}+(\log\frac{1}{q})^2}$ and
$p\in[\frac{1}{n}(k+\log\frac{1}{q}+\sqrt{2k\log\frac{1}{q}+(\log\frac{1}{q})^2}),1]\subset (0,1]$.
\end{lemma}

\begin{proof}
Define
\[
\g =\frac{\log\frac{1}{q}+\sqrt{2k\log\frac{1}{q}+(\log\frac{1}{q})^2}}{k+\log\frac{1}{q}+\sqrt{2k\log\frac{1}{q}+(\log\frac{1}{q})^2}}
\numberthis\label{eq:def_gamma}
\]
so that we can write $p\ge\frac{1}{1-\g}\frac{k}{n}=\frac{1}{1-\g}\frac{kM}{ N}$ and $e^{-\frac{\g^2}{2(1-\g)}k}=q$.
Note that $\g\in[0,1)$ for any $k\ge 1$ and $\tau\in(0,1]$.

For each data split indexed by $m\in[M]$, we define a \emph{bad} event 
\[
E_m=\{r_{k+1}(x;\Xb_m)> r_p(x)\}.
\]
Observe that $E_m$ occurs if any only if the closed ball of probability mass $p$ contains less than $k$ points from $\Xb_m$. 
By Lemma~\ref{lem:cd14_lemma8},
the probability of the bad event $E_m$ is upper bounded by $e^{-\frac{\g^2}{2(1-\g)}k}=q$.

Now, since the data splits are independent, $(1(E_m))_{m=1}^M$ is a sequence of independent Bernoulli random variables with parameter $\Pr(E_1)\le \tau$. 
Hence, we have 
\begin{align*}
\Pr\Bigl(\max_{j\in[L]} r_{k+1}(x;\Xb_{m_j})>r_p(x)\Bigr)
&\le \Pr\Bigl(\sum_{m=1}^{M}1(E_m)> M-L\Bigr)
\\
&= \Pr\Bigl(\sum_{m=1}^{M}1(E_m^c)< L\Bigr)\\
&\le \Pr(\binomrv_{M,1-q} < (1-\tau)(1-q)M),\numberthis\label{eq:intrmdt_binom_bound_original}
\end{align*}
where $\binomrv_{M,q}\sim\Binom(M,q)$ denotes a binomial random variable with parameters $M$ and $q\in[0,1]$.
Another application of the multiplicative Chernoff bound to the right-hand side
concludes the desired bound.
\end{proof}

\section{Analyzing Regression Rules}
\label{app:regression}

In this section, we first analyze the $(k,M,L)$-NN regression rule and then the $(k,M)$-NN rule.
The analysis of the $(k,M)$-NN rule closely resembles that of the $(k,M,L)$-NN rule as we use the $(k,M,L)$-NN rule as a proxy to the $(k,M)$-NN rule, and we show how to carefully control the approximation error.

\subsection{Analysis of the \texorpdfstring{$(k,M,L)$}{(k,M,L)}-NN Regression Rule}

For the $(k,M,L)$-NN regression rule, we claim the following convergence guarantees. Note that Corollary~\ref{cor:regression_selective_simple} is an immediate corollary of the part (a) of this theorem.
\begin{theorem}
\label{thm:regression_selective}
Suppose that Assumptions~\ref{assum:doubling_homogeneous} and \ref{assum:holder} hold.
\begin{enumerate}[label=(\alph*)]
\item For any $1\le M\le N$, pick any $0< \g\le 1-\frac{M}{N}$, any integer $1\le k\le (1-\g)\frac{N}{M}$, and any $\tau\in(0,1)$, such that $L\defeq \lceil(1-\tau)(1-e^{-\frac{\g^2}{2}k}) M\rceil\ge 1$.
If Assumption~\ref{assum:boundedness} holds and the support of $\mu$ is bounded,
\begin{align*}
\E_{\partitions}\|\selectreg{k}{M}{L}-\eta\|_2
&=O\Bigl(\frac{1}{(1-\tau)(1-e^{-\frac{\g^2}{2}k}) kM} + \Bigl(\frac{kM}{(1-\g)N}\Bigr)^{\frac{2\aH}{d}}
+ P_e(k,M;\g,\tau)\Bigr).
\end{align*}

\item For any $1\le M\le N$, pick any integer $1\le k\le \frac{N}{M}$. 
Pick any $\sigma>0$ so that $L\defeq \lceil(1-\tau)(1-e^{-\frac{\sigma^2}{2}k}) M\rceil\ge 1$.
If Assumption~\ref{assum:stronger} holds, for any $0<\d<1$, 
we have
\begin{align*}
\|\selectreg{k}{M}{L}(\cdot;\partitions)-\eta\|_{\infty}
&=O\Bigl(\Bigl(\frac{M}{N}\Bigl(k\vee \Bigl(\log \frac{N}{M} + \sigma^2 k\Bigr)\Bigr)\Bigr)^{\frac{\aH}{d}}
\sqrt{\frac{1}{kM(1-\tau)(1-e^{-\frac{\sigma^2}{2}k})}\log\frac{N}{\d}}
\Bigr)
\end{align*}
with probability at least $1-\d-P_e(k,M;\sigma,\tau)$ over $\partitions$.
\end{enumerate}
Here, the hidden constants in the big-O notation are independent of the ambient dimension $D$.
\end{theorem}

\subsubsection{Proof of Theorem~\ref{thm:regression_selective}(a)}
Let $\partitions_X\defeq\{\Xb_m\}_{m=1}^M$ denote the set of splits of $\Xb$.
We let $V\defeq \sup_{x\in\Xc} v(x)<\infty$ and $H\defeq \sup_{x\in\Xc} |\eta(x)|<\infty$. 
Since the support of $\mu$ is bounded, we let $R\defeq \diam(\supp(\mu))<\infty$.

\paragraph*{Step 1. Error decomposition}
Recall that we wish to bound
\begin{align*}
\E_{\partitions}\|\selectreg{k}{M}{L}(\cdot;\partitions)-\eta\|_2 
&= \E_{\partitions}\sqrt{\E_X[(\selectreg{k}{M}{L}(X;\partitions)-\eta(X))^2]}\\
&\le \sqrt{\E_\partitions\E_X[(\selectreg{k}{M}{L}(X;\partitions)-\eta(X))^2]},
\end{align*}
where the upper bound follows by Jensen's inequality.
We will consider $L=\lceil(1-\tau)(1-e^{-\frac{\g^2}{2}k}) M\rceil$, where $\tau$ is to be determined at the end of the proof.
Pick any $x\in\Xc$.
We denote the conditional expectation of the $(k,M,L)$-NN regression estimate $\selectreg{k}{M}{L}(x;\partitions)$ by
\begin{align*}
\selectregexp{k}{M}{L}(x;\partitions_X)
&\defeq\E_{\Yb|\partitions_X}[\selectreg{k}{M}{L}(x;\partitions)]
\\&
=\frac{1}{kL}\sum_{j=1}^{L} \sum_{i=1}^k \eta(X_{(i)}(x;\Xb_{m_j})),
\end{align*}
where the expectation is over $Y$-values $\Yb$ given the data splits $\partitions_X$.
We decompose the squared error $(\selectreg{k}{M}{L}(x;\partitions)-\eta(x))^2$ as 
\begin{align*}
(\selectreg{k}{M}{L}(x;\partitions)-\eta(x))^2
&\le 2\Bigl\{(\selectreg{k}{M}{L}(x;\partitions)- \selectregexp{k}{M}{L}(x;\partitions_X))^2 
\\&\qquad 
+(\selectregexp{k}{M}{L}(x;\partitions_X)-\eta(x))^2\Bigr\},
\end{align*}
where we use the inequality $(a+b)^2\le 2(a^2+b^2)$.
Taking expectation over the $Y$ values given the data splits $\partitions_X$, we have
\begin{align*}
\E_{\Yb|\partitions_X}[(\selectreg{k}{M}{L}(x;\partitions)-\eta(x))^2]\numberthis\label{eq:error_decomposition}
&\le 2\Bigl\{\underbrace{\Var_{\Yb|\partitions_X}(\selectreg{k}{M}{L}(x;\partitions))}_{\rm(variance)}
\\&\qquad
+\underbrace{(\selectregexp{k}{M}{L}(x;\partitions_X)-\eta(x))^2}_{\rm(bias)}\Bigr\}.
\end{align*}
We now bound the three terms separately in the next steps.

\paragraph*{Step 2. Variance term}
Consider
\begin{align*}
\Var_{\Yb|\partitions_X}(\selectreg{k}{M}{L}(x;\partitions))
&=\E_{\Yb|\partitions_X}[(\selectreg{k}{M}{L}(x;\partitions)-\selectregexp{k}{M}{L}(x;\partitions_X))^2]\\
&=\E_{\Yb|\partitions_X}\Bigl[\Bigl(\frac{1}{kL}\sum_{i=1}^k \sum_{j=1}^L (Y_{(i)}(x;\dataset_{m_j}) - \E[Y_{(i)}(x;\dataset_{m_j})|\partitions_X])\Bigr)^2\Bigr]\\
&\stackrel{(a)}{=}\frac{1}{(kL)^2} \sum_{i=1}^k\sum_{j=1}^{L}\Var_{\Yb|\partitions_X}(Y_{(i)}(x;\dataset_{m_j}))\\
&= \frac{1}{(kL)^2} \sum_{i=1}^k \sum_{j=1}^L v(X_{(i)}(x;\Xb_{m_j}))
\stackrel{(b)}{\le} \frac{V}{kL}.
\numberthis\label{eq:l2_bias_bound_A}
\end{align*}
Here, (a) follows by the independence of $Y_i$'s conditioned on the splits $\partitions_X$ and (b) follows from the assumption $v(x)\le V$ for all $x\in\Xc$.

\paragraph*{Step 3. Bias term}
It only remains to bound the term $(C)$, which is the bias of the $(k,M,L)$-NN regression estimate $\selectreg{k}{M}{L}(x;\partitions)$.
Since $\eta$ is $(\aH,A)$-H\"older continuous, it immediately follows that
\begin{align*}
|\selectregexp{k}{M}{L}(x;\partitions_X)-\eta(x)|
&\le \frac{1}{kL}\sum_{i=1}^k\sum_{j=1}^{L} |\eta(X_{(i)}(x;\Xb_{m_j}))-\eta(x)|\\
&\le A\max_{j\in[L]} r_{k+1}^{\aH}(x;\Xb_{m_j}).
\end{align*}
Now, for any $p\in(0,1)$, we observe that by the homogeneity of $\mu$, we have
\[
C_d\Bigl(\frac{r_p(x)}{2}\Bigr)^d \le \mu\Bigl(\openball\Bigl(x,\frac{r_p(x)}{2}\Bigr)\Bigr)<p,
\]
which implies that $r_p(x)<(\frac{2^dp}{C_d})^{1/d}$.
For $p=\frac{1}{1-\g}\frac{kM}{N}\in(0,1)$,
by Lemma~\ref{lem:key_selective} and the boundedness of the support, \ie $\diam(\supp(\mu))\le R$, we then have
\begin{align*}
\E_{\partitions_X}[(\selectregexp{k}{M}{L}(x;\partitions_X)-\eta(x))^2]
&\le A^2\E_{\partitions_X}\Bigl[\max_{j\in[L]} r_{k+1}^{2\aH}(x;\Xb_{m_j})\Bigr]\\
&\le A^2 \Bigl\{r_p^{2\aH}(x) +
R^{2\aH}\P\Bigl(\max_{j\in[L]} r_{k+1}(x;\Xb_{m_j})>r_p(x)\Bigr)\Bigr\}
\\
&\le A^2\Bigl(\Bigl(\frac{2^d p}{C_d}\Bigr)^{\frac{2\aH}{d}} + R^{2\aH}P_e(k,M;\g,\tau)\Bigr).
\numberthis\label{eq:l2_bias_bound_C}
\end{align*}

\paragraph*{Step 4}
Plugging in \eqref{eq:l2_bias_bound_A} and \eqref{eq:l2_bias_bound_C} to the error decomposition~\eqref{eq:error_decomposition} leads to
\begin{align*}
\E[(\splitreg{k}{M}(x;\partitions)-\eta(x))^2]
&\le 2\Bigl\{
\underbrace{\frac{V}{kL}}_{\rm(variance)}
+ \underbrace{A^2 \Bigl(\frac{2^d p}{C_d}\Bigr)^{\frac{2\aH}{d}} 
+ A^2 R^{2\aH} P_e(k,M;\g,\tau)}_{\rm(bias)}\Bigr\}
\numberthis\label{eq:l2_error_regression_selective}
\\
&= O\Bigl(\frac{1}{(1-\tau)(1-e^{-\frac{\g^2}{2}k}) kM} + \Bigl(\frac{kM}{(1-\g)N}\Bigr)^{\frac{2\aH}{d}}
+ P_e(k,M;\g,\tau)\Bigr).
\end{align*}
This proves the desired rate.
\qed

\subsubsection{Proof of Theorem~\ref{thm:regression_selective}(b)}

This analysis adopts the proof technique of \citep[Proposition~1]{Xue--Kpotufe2018} and will invoke the following lemma therein.

\begin{lemma}[%
{\citep[Lemma~1]{Xue--Kpotufe2018}}]
\label{lem:unif_bound_knn_distances}
Assume that $\mu$ is a $(C_d,d)$-homogeneous measure and the collection of all closed balls in $\Xc$ has finite VC dimension $\Vc$.
Then, with probability at least $1-q$ over the sample $\Xb$ of size $n$, for any $k\in[n]$, we have
\[
\sup_{x\in\Xc} r_k(x;\Xb) \le \Bigl(\frac{3}{C_d n}\Bigl(k\vee\Bigl(\Vc\log 2n + \log\frac{8}{q}\Bigr)\Bigr)\Bigr)^{\frac{1}{d}}.
\]
\end{lemma}

To bound the sup-norm
$\|\selectreg{k}{M}{L}-\eta\|_\infty=\sup_{x\in\Xc} |\selectreg{k}{M}{L}(x;\partitions)-\eta(x)|$,
we consider the following bias-variance decomposition
\begin{align*}
|\selectreg{k}{M}{L}(x;\partitions)-\eta(x)|
&\le \underbrace{|\selectregexp{k}{M}{L}(x;\partitions_X)-\eta(x)|}_{\rm(bias)} + \underbrace{|\selectreg{k}{M}{L}(x;\partitions)-\selectregexp{k}{M}{L}(x;\partitions_X)|}_{\rm(variance)},
\numberthis\label{eq:approximation}
\end{align*}
where we define the conditional expectation $\selectregexp{k}{M}{L}(x;\partitions_X)\defeq \E[\selectreg{k}{M}{L}(x;\partitions)|\partitions_X]$ as in the proof of Theorem~\ref{thm:regression}(a).

\paragraph*{Step 1. Bias term}
The following variant of Lemma~\ref{lem:key_selective}, can be readily shown by invoking Lemma~\ref{lem:unif_bound_knn_distances} with $n\gets \frac{N}{M}$ and following the same line of the proof of Lemma~\ref{lem:key_selective}.

\begin{lemma}\label{lem:key_variant}
Assume that $\mu$ is a $(C_d,d)$-homogeneous measure and the collection of all closed balls in $\Xc$ has finite VC dimension $\Vc$.
Pick any any positive integers $n$, $k$, and $M$.
Further, pick any $\sigma>0$ and $\tau\in[0,1)$, so that $L=\lceil (1-\tau)(1-e^{-\frac{\sigma^2}{2}k}) M\rceil\ge 1$.
If the data splits $\Xb_1,\ldots,\Xb_M$ of size $n$ are drawn \iid from $\mu^{\otimes n}$, we have
\begin{align*}
\Pr\Bigl(\max_{j\in[L]} \sup_{x\in\Xc} r_{k}(x;\Xb_{m_j})
>h_{n,k}(e^{-\frac{\sigma^2}{2}k})\Bigr)
&\le P_e(k,M;\sigma,\tau),
\end{align*}
where we define
\[
h_{n,k}(a)
\defeq \Bigl(\frac{3}{C_d n} \Bigl(k\vee\Bigl(\Vc\log 2n + \log \frac{8}{a}\Bigr)\Bigr)\Bigr)^{\frac{1}{d}}.
\numberthis\label{eq:def_threshold}
\]
\end{lemma}
\begin{proof}[Proof of Lemma~\ref{lem:key_variant}]
Set $q=e^{-\frac{\sigma^2}{2}k}$ for some $\sigma>0$.
For each data split indexed by $m\in[M]$, we define a \emph{bad} event 
\[
F_m=\Bigl\{\sup_{x\in\Xc} r_{k}(x;\Xb_m)> h_{n,k}(q)\Bigr\}.
\]
By Lemma~\ref{lem:unif_bound_knn_distances},
the probability of the bad event $F_m$ is at most $q$.
Now, since the data splits are independent, $(1(F_m))_{m=1}^M$ is a sequence of independent Bernoulli random variables with parameter $\Pr(F_1)\le q$. 
Hence, we have 
\begin{align*}
\Pr\Bigl(\max_{j\in[L]} \sup_{x\in\Xc} r_{k}(x;\Xb_{m_j})>h_{n,k}(q)\Bigr)
&\le \Pr\Bigl(\sum_{m=1}^{M}1(F_m)> M-L\Bigr)
\\
&= \Pr\Bigl(\sum_{m=1}^{M}1(F_m^c)< L\Bigr)\\
&\le \Pr(\binomrv_{M,1-q} < (1-\tau)(1-q)M).
\end{align*}
Recall that $q=e^{-\frac{\sigma^2}{2}k}$.
We now apply Lemma~\ref{lem:refined_multiplicative_chernoff} for $(n,q,\lambda)\gets (M,1-q,1-\tau)$, we have
\begin{align*}
\Pr\Bigl(\max_{j\in[L]} \sup_{x\in\Xc} r_{k}(x;\Xb_{m_j})>h_{n,k}(q)\Bigr)
&\le P_e(k,M;\sigma,\tau),
\end{align*}
which concludes the proof.
\end{proof}

Then, Lemma~\ref{lem:key_variant} together with the H\"older continuity of $\eta$ implies that with probability at least $P_e(k,M;\sigma,\tau)$ over the data splits $\partitions_X$, we have
\begin{align*}
\sup_{x\in\Xc}
|\selectregexp{k}{M}{L}(x;\partitions_X)-\eta(x)|
&\le A\sup_{x\in\Xc}\max_{j\in[L]} r_k^{\aH}(x;\Xb_{m_j})\\
&\le A\Bigl(\frac{3M}{C_d N} \Bigl(k\vee\Bigl(\Vc\log \frac{2N}{M} + \frac{\sigma^2}{2}k + \log 8\Bigr)\Bigr)\Bigr)^{\frac{\aH}{d}}.
\numberthis\label{eq:bias_bound_whp}
\end{align*}

\paragraph*{Step 2. Variance term}
For any fixed $x\in\Xc$ and split instances $\partitions_X=\{\Xb_m\}_{m=1}^M$, Hoeffding's inequality guarantees that with probability at least $1-\d_o$ over the labels $\{\Yb_m\}_{m=1}^M$, we have 
\begin{align}\label{eq:hoeffding_fixed_x}
|\selectreg{k}{M}{L}(x;\partitions)-\selectregexp{k}{M}{L}(x;\partitions_X)|\le \sqrt{\frac{l_Y^2}{2kL}\log\frac{2}{\d_o}}.
\end{align}
Now, observe that given $\partitions_X$, the left hand side is a function of $x$ only via its nearest neighbors from $\Xb$, and thus only depends on a closed ball centered at $x$.
The finite VC dimensionality assumption then implies that if we vary $x\in\Xc$, there are at most $(\frac{eN}{\Vc})^{\Vc}$ different such inequalities~\eqref{eq:hoeffding_fixed_x}.
Hence, letting $\d=\d_o(\frac{eN}{\Vc})^{\Vc}$ and applying union bound, we have, with probability at least $1-\d$ over $\{\Yb_m\}_{m=1}^M$,
\begin{align}\label{eq:variance_bound_whp}
\sup_{x\in\Xc}|\selectreg{k}{M}{L}(x;\partitions)-\selectregexp{k}{M}{L}(x;\partitions_X)|
\le \sqrt{\frac{\Vc l_Y^2}{kL}\log\frac{N}{\d}}.
\end{align}
Since this inequality holds independent of $\partitions_X$, it also holds with probability at least $1-\d$ over the split data $\partitions$.

\paragraph*{Step 3}
Continuing from \eqref{eq:approximation} and combining the bias bound \eqref{eq:bias_bound_whp} and variance bound~\eqref{eq:variance_bound_whp} by union bound, we have with probability at least $1-\d-P_e(k,M;\sigma,\tau)$, 
\begin{align*}
\|\selectreg{k}{M}{L}-\eta\|_\infty
&\le \underbrace{A\Bigl(\frac{3M}{C_d N} \Bigl(k\vee\Bigl(\Vc\log \frac{2N}{M} + \frac{\sigma^2}{2}k + \log 8\Bigr)\Bigr)\Bigr)^{\frac{\aH}{d}}}_{\rm(bias)} 
+\underbrace{\sqrt{\frac{\Vc l_Y^2}{kL}\log\frac{N}{\d}}}_{\rm(variance)}\\
&=O\Bigl(\Bigl(\frac{M}{N}\Bigl(k\vee \Bigl(\log \frac{N}{M} 
+ \sigma^2 k\Bigr)\Bigr)\Bigr)^{\frac{\aH}{d}} 
+ \sqrt{\frac{1}{kM(1-\tau)(1-e^{-\frac{\sigma^2}{2}k})}\log\frac{N}{\d}}
\Bigr),
\end{align*}
which leads to the desired bound.
\qed

\subsection{Analysis of the \texorpdfstring{$(k,M)$}{(k,M)}-NN Regression Rule}

We now prove the main theorem for the $(k,M)$-NN regression rule.

\ThmRegression*

\subsubsection{Proof of Theorem~\ref{thm:regression}(a)}
Recall that in Assumption~\ref{assum:boundedness}, we let $V\defeq \sup_{x\in\Xc} v(x)<\infty$ and $H\defeq \sup_{x\in\Xc} |\eta(x)|<\infty$. 
Since the support of $\mu$ is bounded, we let $R\defeq \diam(\supp(\mu))<\infty$.
Recall that we wish to bound
\begin{align*}
\E_{\partitions}\|\splitreg{k}{M}-\eta\|_2 
&= \E_{\partitions}\sqrt{\E_X[(\splitreg{k}{M}(X;\partitions)-\eta(X))^2]}\\
&\le \sqrt{\E_\partitions\E_X[(\splitreg{k}{M}(X;\partitions)-\eta(X))^2]}.
\end{align*}
Here, the inequality follows by Jensen's inequality.
We will consider the $(k,M,L)$-NN regression rule with $L=\lceil(1-\tau)^2 M\rceil$ as a proof device, where $\tau\in(0,1)$ is to be determined at the end of the proof.
Pick any $x\in\Xc$.
We decompose the squared error $(\splitreg{k}{M}(x;\partitions)-\eta(x))^2$ as 
\begin{align*}
(\splitreg{k}{M}(x;\partitions)-\eta(x))^2
&\le 3\Bigl\{(\splitreg{k}{M}(x;\partitions)-\splitregexp{k}{M}(x;\partitions_X))^2 
\\&\qquad 
+(\splitregexp{k}{M}(x;\partitions_X) - \selectregexp{k}{M}{L}(x;\partitions_X))^2 
\\&\qquad 
+(\selectregexp{k}{M}{L}(x;\partitions_X)-\eta(x))^2\Bigr\},
\end{align*}
where we use the inequality $(a+b+c)^2\le 3(a^2+b^2+c^2)$.
Taking expectation over the $Y$ values given the data splits $\partitions_X$, we have
\begin{align*}
\E_{\Yb|\partitions_X}[(\splitreg{k}{M}(x;\partitions)-\eta(x))^2]\numberthis\label{eq:error_decomposition_}
&\le 3\Bigl\{
\underbrace{(\splitregexp{k}{M}(x;\partitions_X) - \selectregexp{k}{M}{L}(x;\partitions_X))^2}_{\rm(approximation)} 
\\&\qquad 
+\underbrace{\Var_{\Yb|\partitions_X}(\splitreg{k}{M}(x;\partitions))}_{\rm(variance)}
\\&\qquad 
+\underbrace{(\selectregexp{k}{M}{L}(x;\partitions_X)-\eta(x))^2}_{\rm(bias)}\Bigr\}.
\end{align*}
We now bound the three terms separately in the next steps.

\paragraph*{Step 1. Approximation term}
We claim that the first term, which is the approximation error between the $(k,M)$-NN rule and $(k,M,L)$-NN rule is bounded as $O(\tau^2)$. 
We first note that, by Jensen's inequality, we have 
\begin{align*}
|\splitregexp{k}{M}(x;\partitions_X)-\selectregexp{k}{M}{L}(x;\partitions_X)|
&\le \E_{\Yb|\partitions_X}\bigl[|\splitreg{k}{M}(x;\partitions)-\selectreg{k}{M}{L}(x;\partitions)|\bigr].
\end{align*}
The argument inside the expectation can be bounded pointwise as follows:
\begin{align*}
\bigl|\splitreg{k}{M}(x;\partitions)-\selectreg{k}{M}{L}(x;\partitions)\bigr| 
&\le
\Bigl|
\frac{1}{M}\sum_{m=1}^M \nnreg{k}(x;\dataset_m)
- \frac{1}{L}\sum_{j=1}^L \nnreg{k}(x;\dataset_{m_j})
\Bigr|
\\
&\le \Bigl(\frac{1}{L}-\frac{1}{M}\Bigr) \sum_{j=1}^L |\nnreg{k}(x;\dataset_{m_j})| + \frac{1}{M} \sum_{j=L+1}^M | \nnreg{k}(x;\dataset_{m_j})|\\
&\stackrel{(a)}{\le} \Bigl(1-\frac{L}{M}\Bigr)H + \frac{M-L}{M}H
\\
&= 2H\Bigl(1-\frac{L}{M}\Bigr)\\
&\stackrel{(b)}{\le} 4H\tau,
\numberthis\label{eq:l2_bias_bound_B_}
\end{align*}
where (a) follows by the assumption $|\eta(x)|\le H$ for all $x\in\Xc$ and (b) follows since $L=\lceil(1-\tau)^2M\rceil\ge (1-2\tau)M$.
Thus, the approximation error $|\splitregexp{k}{M}(x;\partitions_X)-\selectregexp{k}{M}{L}(x;\partitions_X)|$ is upper bounded by $4H\tau$. 

\paragraph*{Step 2. Variance term}
Similar to \eqref{eq:l2_bias_bound_A}, we can easily show that 
\begin{align*}
\Var_{\Yb|\partitions_X}(\splitreg{k}{M}(x;\partitions))
&= \frac{1}{(kM)^2} \sum_{i=1}^k \sum_{m=1}^M v(X_{(i)}(x;\Xb_m))
\stackrel{(b)}{\le} \frac{V}{kM}.
\numberthis\label{eq:l2_bias_bound_A_}
\end{align*}
Here, (a) follows by the independence of $Y_i$'s conditioned on the splits $\partitions_X$ and (b) follows from the assumption $v(x)\le V$ for all $x\in\Xc$.

\paragraph*{Step 3. Bias term}
In Step 3 in the proof of Theorem~\ref{thm:regression_selective}(a), in place of applying Lemma~\ref{lem:key_selective}, we now apply Lemma~\ref{lem:key_original} with $q=\tau$ and $p=\frac{1}{n}(k+\log\frac{1}{\tau}+\sqrt{2k\log\frac{1}{\tau}+(\log\frac{1}{\tau})^2})=O(\frac{M}{N}(k+\log\frac{1}{\tau}))$:
\begin{align*}
&\E_{\partitions_X}[(\selectregexp{k}{M}{L}(x;\partitions_X)-\eta(x))^2]
\le A^2\Bigl(\Bigl(\frac{2^dp}{C_d}\Bigr)^{\frac{2\aH}{d}} + R^{2\aH}e^{-\frac{(1-\tau)\tau^2}{2}M}\Bigr).
\end{align*}

\paragraph*{Step 4}
Plugging in \eqref{eq:l2_bias_bound_B_}, \eqref{eq:l2_bias_bound_A_}, and \eqref{eq:l2_bias_bound_C} to the error decomposition~\eqref{eq:error_decomposition_} leads to
\begin{align*}
\E[(\splitreg{k}{M}(x;\partitions)-\eta(x))^2]
&\le 3\Bigl\{
\underbrace{16H^2\tau^2}_{\rm(approximation)}
+ \underbrace{\frac{V}{kM}}_{\rm(variance)}
+ \underbrace{A^2 \Bigl(\frac{2^dp}{C_d}\Bigr)^{\frac{2\aH}{d}} 
+ A^2 R^{2\aH} e^{-\frac{(1-\tau)\tau^2}{2}M}\Bigr)}_{\rm(bias)}\Bigr\}\\
&= O\Bigl(\tau^2 + \frac{1}{kM} + \Bigl(\frac{M}{N}\bigl(k+\log\frac{1}{\tau}\bigr)\Bigr)^{\frac{2\aH}{d}} + e^{-\frac{(1-\tau)\tau^2}{2}M}\Bigr)\\
&= O\Bigl(\tau^2 + \frac{1}{M} + \Bigl(\frac{M}{N}\log\frac{1}{\tau}\Bigr)^{\frac{2\aH}{d}} + e^{-\frac{(1-\tau)\tau^2}{2}M}\Bigr),
\end{align*}
where in the last equality we assume $k=O(1)$ is fixed.
To handle the additional approximation term $O(\tau^2)$ compared to \eqref{eq:l2_error_regression_selective}, we need to set $\tau$ to decay to 0 as $M$ increases.
Specifically, we choose to set $\tau=\sqrt{\frac{(\log M)^{1+\eps}}{M}}$ for any $\eps>0$, then $e^{-\frac{(1-\tau)\tau^2}{2}M}=e^{-\half(\log M)^{1+\eps}(1-\tau)}$ decays faster than any polynomial rate, which leads to the final rate
\begin{align*}
\E[(\splitreg{k}{M}(x;\partitions)-\eta(x))^2]
&= O\Bigl(
\frac{(\log M)^{1+\eps}}{M} + \Bigl(\frac{M\log M}{N}\Bigr)^{\frac{2\aH}{d}}
\Bigr).\qed
\end{align*}

\subsubsection{Proof of Theorem~\ref{thm:regression}(b)}
The proof readily follows from a slight extension of the proof of Theorem~\ref{thm:regression_selective}(b), since we can also approximate the performance of $(k,M)$-NN regression rule in $\infty$-norm using a $(k,M,L)$-NN estimator $\selectreg{k}{M}{L}(x;\partitions)$ by the triangle inequality:
\begin{align*}
\|\splitreg{k}{M}-\eta\|_{\infty}
&\le \underbrace{\|\splitregexp{k}{M} - \selectregexp{k}{M}{L}\|_{\infty} }_{\rm(approximation)}
+\underbrace{\|\splitreg{k}{M}-\splitregexp{k}{M}\|_{\infty}}_{\rm(variance)}
+\underbrace{\|\selectregexp{k}{M}{L}-\eta\|_{\infty}}_{\rm(bias)}.
\end{align*}
Note that the first term, which is the approximation error in the sup norm, can be upper-bounded by $4H\tau$, invoking the same bound in \eqref{eq:l2_bias_bound_B_}.
For the second term, which is the variance of $\splitreg{k}{M}$, we can follow the same logic in the proof of Theorem~\ref{thm:regression_selective}(b) and can be bounded by $\sqrt{\frac{\Vc l_Y^2}{kM}\log\frac{N}{\d}}$ with probability $1-\d$.
The last term is the bias of the $(k,M,L)$-NN rule, and we invoke the following lemma which is a variant of Lemma~\ref{lem:key_variant}.
\begin{lemma}\label{lem:key_variant_original}
Assume that $\mu$ is a $(C_d,d)$-homogeneous measure and the collection of all closed balls in $\Xc$ has finite VC dimension $\Vc$.
Pick any any positive integers $n$, $k$, and $M$.
Further, pick any $q\in[0,1)$ and $\tau\in[0,1)$, so that $L=\lceil (1-\tau)(1-q) M\rceil\ge 1$.
If the data splits $\Xb_1,\ldots,\Xb_M$ of size $n$ are drawn \iid from $\mu^{\otimes n}$, we have
\begin{align*}
\Pr\Bigl(\max_{j\in[L]} \sup_{x\in\Xc} r_{k}(x;\Xb_{m_j})>h_{n,k}(q)\Bigr)
&\le e^{\frac{(1-q)\tau^2}{2}M},
\end{align*}
where $h_{n,k}(q)$ is defined in \eqref{eq:def_threshold}.
\end{lemma}
We omit the proof as it is a straightforward modification of that of Lemma~\ref{lem:key_variant}, replacing Lemma~\ref{lem:key_selective} with Lemma~\ref{lem:key_original}.
By applying this lemma for $q=\tau$, with probability $\ge 1-e^{-\frac{(1-\tau)\tau^2}{2}M}$.
Suppose that $M\ge \frac{2}{\kappa^2(1-\kappa)}\log\frac{1}{\d}$ for some $\kappa\in (0,1)$.
If we set $\tau=\sqrt{\frac{2}{(1-\kappa)}\frac{1}{M}\log\frac{1}{\d}}$, then $\tau\le \kappa$ by the choice of $M$, and $e^{-\frac{(1-\tau)\tau^2}{2}M}\le e^{-\frac{(1-\kappa)\kappa^2}{2}M}\le \d$.
Therefore, by a union bound, we can now bound the sup norm as follows: with probability at least $1-\d$, we have
\begin{align*}
\|\splitreg{k}{M}-\eta\|_\infty
&\le \underbrace{4H\tau}_{\rm(approximation)} 
+\underbrace{\sqrt{\frac{\Vc l_Y^2}{kM}\log\frac{2N}{\d}}}_{\rm(variance)}
+ \underbrace{A\Bigl(\frac{3M}{C_d N} \Bigl(k\vee\Bigl(\Vc\log \frac{2N}{M} + \log \frac{8}{\tau}\Bigr)\Bigr)\Bigr)^{\frac{\aH}{d}}}_{\rm(bias)}
\\
&=O\Bigl(\sqrt{\frac{1}{(1-\kappa)M}\log\frac{1}{\d}} 
+ \Bigl(\frac{M}{N}\Bigl(2\Vc\log \frac{N}{M} + \log \frac{(1-\kappa)M}{\log\frac{1}{\d}}\Bigr)\Bigr)^{\frac{\aH}{d}} 
+ \sqrt{\frac{1}{M}\log\frac{N}{\d}}
\Bigr),
\end{align*}
where in the last equality we assume $k=O(1)$ is fixed.
This concludes the desired bound.
\qed

\section{Analyzing Classification Rules}
\label{app:classification}
All theoretical guarantees on classifiers in this paper are analogous to the results for the standard $k$-NN classifier established in the seminal paper~\citep{Chaudhuri--Dasgupta2014}.

\subsection{Definitions}
We first review some technical definitions introduced in \citep{Chaudhuri--Dasgupta2014}.
For any $x\in\Xc$ and any $0\le p\le 1$, define the \emph{probability radius} of a ball centered at $x$ as
\begin{align*}
    r_p(x) = \inf\{r\suchthat \mu(\closedball(x,r))\ge p\}.
\end{align*}
One can show that $\mu(\openball(x,r_p(x)))\ge p$, and $r_p(x)$ is the smallest radius for which this holds.

The support of the distribution $\mu$ is defined as
\begin{align*}
    \supp(\mu)\defeq\{x\in\Xc\suchthat \mu(\closedball(x,r))>0,\forall r>0\}.
\end{align*}
In separable metric spaces, it can be shown that $\mu(\supp(\mu))=1$; see \citep{Cover--Hart1967} or  \citep[Lemma~24]{Chaudhuri--Dasgupta2014}.

We define for any measurable set $A\subset \Xc$ with $\mu(A)>0$,
\begin{align*}
\eta(A)
    \defeq p(y=1|A)
    =\frac{1}{\mu(A)}\int_A p(y=1|x)\diff\mu(x).
\end{align*}
This is the conditional probability of $Y$ being 1 given a point $X$ chosen at random from the distribution $\mu$ restricted to the set $A$.

Based on the definitions above, we now define the effective interiors of the two classes, and the effective boundary.
For $p\in [0,1]$ and $\Delta>0$, we define the \emph{effective interiors} for each class as
\begin{align*}
    \Xc_{p,\Delta}^+
    &\defeq\supp(\mu) 
    \\&\quad 
    \cap\Bigl\{x\in\Xc\suchthat \eta(x)>\half\Bigr\}
    \\&\quad 
    \cap\Bigl\{x\in\Xc\suchthat
        \eta(\closedball(x,r))\ge \half+\Delta, \forall r\le r_p(x)
        \Bigr\}
\end{align*}
and
\begin{align*}
    \Xc_{p,\Delta}^-
    &\defeq\supp(\mu)
    \\&\quad
    \cap\Bigl\{x\in\Xc\suchthat \eta(x)<\half\Bigr\}
    \\&\quad 
    \cap\Bigl\{x\in\Xc\suchthat
        \eta(\closedball(x,r))\le \half-\Delta, \forall r\le r_p(x)
        \Bigr\}.
\end{align*}

\iftrue
\begin{remark}[A Tie-breaking Mechanism]
\citet{Chaudhuri--Dasgupta2014} discuss how to handle ties in the nearest neighbor distances, which may occur with non-zero probability, \eg in a discrete instant space $\Xc$.
We assume that there is no distance tie in what follows or the sake of simplicity, but our analyses can be similarly modified for non-zero distance ties.

\end{remark}

\subsection{Analysis of the \texorpdfstring{$(k,M,L)$}{(k,M,L)}-NN Classification Rule}

We claim the following convergence guarantees for the $(k,M,L)$-NN classifier. 

\begin{theorem}
\label{thm:classification_rate_margin_selective}
Under Assumptions~\ref{assum:smooth} and \ref{assum:margin}, the following statements hold for any fixed $k\ge 1$, where $M_o$, $C_o$, $C_o'$, and $C_o''$ are constants depending on $k,\a,\b,$ and $B$.
In what follows, we assume 
either $kM\ge \frac{2c}{c-1}\log\frac{2}{\d^2}$ and $k\ge c\frac{2-\tau^2}{\g^2\tau}$ for some $c>1$, or $k\ge 1$, $M\ge \frac{1}{\phi_{k;\g,\tau}}\log\frac{2}{\d^2}$, and $\tau\in(\tau_{k;\g}^{\inf},1)$. Here, $\phi_{k;\g,\tau}$ and $\tau_{k;\g}^{\inf}$ are defined in Lemma~\ref{lem:key_selective}.
\begin{enumerate}[label=(\alph*)]
\item Pick any $\d\in(0,1)$ and $M_o>0$ such that $M=M_oN^{\frac{2\a}{2\a+1}}(\log\frac{1}{\d})^{\frac{1}{2\a+1}}\le N$. 
With probability at least $1-\d$ over $\partitions$,
\begin{align*}
&\Pr(\selectcls{k}{M}{L}(X)\neq g(X)|\partitions)
\le \d + 
C_o\Bigl(\frac{1}{N}\log\frac{1}{\delta}\Bigr)^{\frac{\b\a}{2\a+1}}.
\end{align*}
\item Set $M=M_oN^{\frac{2\a}{2\a+1}}$. Then
\begin{align*}
\E_\partitions[R(\selectcls{k}{M}{L})]-R^* &\le {C_o'{N^{-\frac{\a(\b+1)}{2\a+1}}}}\quad\text{and}\quad\\
\CIS_N(\selectcls{k}{M}{L})&\le {C_o''{N^{-\frac{\a\b}{2\a+1}}}}.
\end{align*}
\end{enumerate}
\end{theorem}

\subsubsection{A Key Technical Lemma}

The analysis of the standard $k$-NN classifier by \citet{Chaudhuri--Dasgupta2014} relies on their key lemma~\citep[Lemma~7]{Chaudhuri--Dasgupta2014}, which proves a sufficient condition for the $k$-NN classifier to agree with the Bayes classifier.
Here, we provide an analogous lemma for the $(k,M,L)$-NN classifier.
\begin{lemma}
\label{lem:sufficient_selective}
For any $x_o\in\Xc$, pick any $p\in(0,1)$ and $\Delta\in(0,\half]$. 
For each $m\in[M]$, 
define $\ball_m\defeq \openball(x_o,r_{k+1}(x_o;\Xb_m))$.
For any $L\le M$, we have
\begin{align*}
\ones(\selectcls{k}{M}{L}(x_o;\partitions)\neq \gbayes(x_o))
&\le \ones(x_o\in\partial_{p,\Delta})\\
&\quad+\ones\Bigl(\max_{j\in [L]} r_{k+1}(x_o;\Xb_{m_j})> r_p(x_o)\Bigr)\\
&\quad+\ones\Bigl(\Bigl|\frac{1}{L }\sum_{j=1}^{L }(\Yh(\ball_{m_j};\dataset_{m_j})-\eta(\ball_{m_j}))\Bigr|\ge \Delta\Bigr),
\numberthis\label{eq:sufficient_classification_selective}
\end{align*}
where $m_1,\ldots,m_L$ are the indices that correspond to the $L$-smallest values among $(r_{k+1}(x_o;\Xb_m))_{m=1}^M$.
\end{lemma}

To prove this, we need the following lemma:
\begin{lemma}[{\citealp[Lemma~26]{Chaudhuri--Dasgupta2014}}]
\label{lem:lem26_cd14}
Suppose that for some $x_o\in\supp(\mu)$ and $r_o>0$ and $q>0$, we have [$r\le r_o$ $\Rightarrow$ $\eta(\closedball(x_o,r))\ge q$].
Then, we also have [$r\le r_o$ $\Rightarrow$ $\eta(\openball(x_o,r))\ge q$].
\end{lemma}

\begin{proof}[Proof of Lemma~\ref{lem:sufficient_selective}] %
Suppose $x_o\notin\partial_{p,\Delta}$.
Without loss of generality, consider $x_o\in\Xc_{p,\Delta}^+$, whereupon $\gbayes(x_o)=1$.
By definition of the effective interior, 
$\eta(\closedball(x_o,r))\ge \half+\Delta$ for all $r\le r_p(x_o)$.
If we further suppose
\[
\max_{j\in [L]} r_{k+1}(x_o;\Xb_{m_j})\le r_p(x_o),
\]
we have
\[
\eta(\ball_{m_j})
=\eta(\openball(x_o,r_{k+1}(x_o;\Xb_{m_j})))
\ge \half+\Delta
\]
for any $j\in[L]$, by Lemma~\ref{lem:lem26_cd14}.

Further, if $|\frac{1}{L}\sum_{j=1}^{L }(\Yh(\ball_{m_j};\dataset_{m_j})-\eta(\ball_{m_j}))|< \Delta$, then 
\[
\selectreg{k}{M}{L}(x_o;\partitions)
=\frac{1}{L}\sum_{j=1}^L \Yh(\ball_{m_j};\dataset_{m_j}) > \half,
\]
where we recall that $\selectreg{k}{M}{L}(\cdot;\partitions)$ denotes the $(k,M,L)$-NN regressor based on the training data splits $\partitions=\{(\Xb_m,\Yb_m)\}_{m=1}^M$.
This implies that $\selectcls{k}{M}{L}(x_o;\partitions)=1=g(x_o)$, which completes the proof.
\end{proof}

\subsubsection{Proof of Theorem~\ref{thm:classification_rate_margin_selective}(a)}

Before we prove Theorem~\ref{thm:classification_rate_margin_selective}(a), we first present a more general upper bound on the misclassification error rate, which is a variant of the main result of \citep{Chaudhuri--Dasgupta2014} (Theorem~5 therein).
Theorem~\ref{thm:classification_rate_margin_selective}(a) will follow as a corollary of this theorem under the smoothness and margin condition.

\begin{theorem}
\label{thm:classification_general_upper_bound_selective}
Let $k\ge 1$ be fixed and
pick any $\d\in(0,1)$.
Let $L=\lceil(1-\tau)(1-q_{k;\g})M\rceil$ for some $\tau\in (0,1)$ and $\g\in(0,1)$. 
Let $p\defeq\frac{1}{1-\g}\frac{k}{n}$ and
let $\Delta\defeq\min(\half,\sqrt{\frac{1}{kL}\log\frac{2}{\d}})$.
Then, for a set of data splits $\partitions=\{\dataset_1,\ldots,\dataset_M\}$, where every split $\dataset_m$ has $n$ data points, with probability at least $1-\d$ over $\partitions$, we have
\begin{align*}
\Pr(\selectcls{k}{M}{L}(X;\partitions)\neq \gbayes(X)|\partitions)
\le \d + \mu(\partial_{p,\Delta}),
\end{align*}
either if $kM\ge \frac{2c}{c-1}\log\frac{2}{\d^2}$ and $k\ge c\frac{2-\tau^2}{\g^2\tau}$ for some $c>1$, or if $k\ge 1$, $M\ge \frac{1}{\phi_{k;\g,\tau}}\log\frac{2}{\d^2}$, and $\tau\in(\tau_{k;\g}^{\inf},1)$; see Lemma~\ref{lem:key_selective} for the definitions of $\phi_{k;\g,\tau}$ and $\tau_{k;\g}^{\inf}$.
\end{theorem}

\begin{proof}
Pick any $x_o\in\Xc$.
Applying Lemma~\ref{lem:sufficient_selective}, we have
$$
\begin{aligned}
\ones(\selectcls{k}{M}{L}(x_o;\partitions)\neq \gbayes(x_o))
&\le \ones(x_o\in\partial_{p,\Delta})+ I_{\bad}(x_o;\partitions),
\end{aligned}$$ where we define the bad event indicator variable
\begin{align*}
I_{\bad}(x_o;\partitions)
&\defeq \ones\Bigl(\max_{j\in [L]} r_{k+1}(x_o;\Xb_{m_j})> r_p(x_o)\Bigr)
\\&\quad+
\ones\Bigl(\Bigl|\frac{1}{L }\sum_{j=1}^{L }(\Yh(\ball_{m_j};\dataset_{m_j})-\eta(\ball_{m_j}))\Bigr|\ge \Delta\Bigr),
\numberthis\label{eq:bad_indicator_selective}
\end{align*}
where $m_1,\ldots,m_L$ are the indices for the $L$ smallest distances among $\{r_{k+1}(x_o;\Xb_m)\}_{m=1}^M$. For any fixed point $x_o\in\Xc$, if we take the expectation over the training data splits $\partitions$, we have
\begin{align*}
\E[I_{\bad}(x_o;\partitions)]
&\defeq \P\Bigl(\max_{j\in [L]} r_{k+1}(x_o;\Xb_{m_j})> r_p(x_o)\Bigr)
\\&\qquad
+\P\Bigl(\Bigl|\frac{1}{L }\sum_{j=1}^{L }(\Yh(\ball_{m_j};\dataset_{m_j})-\eta(\ball_{m_j}))\Bigr|\ge \Delta\Bigr).
\numberthis\label{eq:exp_ind_bad_event_selective}
\end{align*}
To handle the second term, we need the following concentration bound, which is a distributed version of \citep[Lemma~10]{Chaudhuri--Dasgupta2014}. 

\begin{lemma}
\label{lem:label_concentration_new}
\[
\P\Bigl(\Bigl|\frac{1}{L }\sum_{j=1}^{L }(\Yh(\ball_{m_j};\dataset_{m_j})-\eta(\ball_{m_j}))\Bigr|\ge \Delta\Bigr)\le 2e^{-{2\Delta^2}kL}.
\numberthis\label{eq:Hoeffding}
\]
\end{lemma}

\begin{proof}[Proof of Lemma~\ref{lem:label_concentration_new}]
To prove it, first observe that we can draw the training data splits $\partitions=\dataset_{1:M}$, $\dataset_m=\{(X_{mi},Y_{mi})\}_{i=1}^n$, where $N=Mn$, by the following steps. 
\begin{enumerate}
\item[1.] Draw $M$ points $X_{1}^{(1)},\ldots,X_{1}^{(M)}\in\Xc$ independently at random, according to the marginal distribution of the $(k+1)$-th nearest neighbor of the fixed point $x_o$ with respect to $n$ independent sample points.
\item[2.] Sort the $M$ points $\{X_{1}^{(1)},\ldots,X_{1}^{(M)}\}$ based on their distances to $x_o$. 
Let $\tilde{X}_{1}^{(1)},\ldots,\tilde{X}_{1}^{(M)}$ denote the sorted points in the increasing order of the distances, where we break ties at random.
Let $\widetilde{\openball}_j\defeq\openball(x_o,\rho(x_o,\tilde{X}_{1}^{(j)}))$.
\item[3(a).] For each $j\in[L]$, pick $k$ points at random from the distribution $\mu$ restricted to $\widetilde{\openball}_j$.
\item[3(b).] For each $j\in[L]$, pick $n-k-1$ points at random from the distribution $\mu$ restricted to $\Xc\backslash \widetilde{\openball}_j$.
\item[4.] For each $m\in[M]\backslash[L]$, repeat the same steps in 3a and 3b.
\item[5.] For each $m\in[M]$, randomly permute the $n$ points obtained in this way.
\item[6.] For each $m\in[M]$ and for $X_{mi}$ in the permuted order, draw a label $Y_{mi}$ from the conditional distribution $\eta(X_{mi})$.
\end{enumerate}

We now suppose that we are given $\tilde{X}_{1}^{(1)},\ldots,\tilde{X}_{1}^{(M)}$ chosen in Step 1 and Step 2.
Recall that we denote by $m_1,\ldots,m_L$ the indices that correspond to the $L$-smallest values among $(r_{k+1}(x_o;\Xb_m))_{m=1}^M$.
Since the corresponding sample points are $\tilde{X}_{1}^{(1)},\ldots,\tilde{X}_{1}^{(L)}$, we can write $\ball_{m_j}=\widetilde{\openball}_j$.
Hence, in the desired inequality, $\hat{Y}(\ball_{m_j};\dataset_{m_j})$ for each $j\in[L]$ is the average of the $Y$-values which correspond to the $X$'s drawn from Step 3(a).
Since the corresponding $Y$'s have expectation $\E[Y|\Xb\in\widetilde{\openball}_j]=\eta(\widetilde{\openball}_j)$ for each $j\in[L]$ and the total $kL$ of $Y$'s are independent, we can apply Hoeffding's inequality and obtain
\begin{align*}
&\P\Bigl(\Bigl|\frac{1}{L }\sum_{j=1}^{L }(\Yh(\ball_{m_j};\dataset_{m_j})-\eta(\ball_{m_j}))\Bigr|\ge \Delta
~\Big|~ \tilde{X}_{1}^{(1)},\ldots,\tilde{X}_{1}^{(M)}
\Bigr)
\le 2e^{-2\Delta^2 kL}.
\end{align*}
Taking expectations over $\tilde{X}_{1}^{(1)},\ldots,\tilde{X}_{1}^{(M)}$, we prove the desired inequality.
\end{proof}

Now, by applying Lemma~\ref{lem:key_selective} and Lemma~\ref{lem:label_concentration_new} on the first and second terms, respectively,
we have
\[
\E_{\partitions}[I_{\bad}(x_o,\partitions)] \le
P_e(k,M;\g,\tau)
+2e^{-2\Delta^2kL}
\defeq \d_o.
\]
Note that the expectation is over the training data $\partitions$.
Taking expectation over the query point $X_o\sim \mu$, we have $\E_{X_o,\partitions}[I_{\bad}(X_o,\partitions)]\le \d_o^2$, which in turn implies, by Markov's inequality, that 
\begin{align*}
\Pr_{\partitions}\bigl(\E_{X_o}[I_{\bad}(X_o,\partitions)]\ge \d_o\bigr)\le \d_o.
\numberthis\label{eq:markov}
\end{align*}
Note that Lemma~\ref{lem:sufficient} implies
\begin{align*}
&\Pr\bigl(\selectcls{k}{M}{L}(X_o;\partitions)\neq \gbayes(X_o)|\partitions\bigr) 
\le \mu(\partial_{p,\Delta}) + \E_{X_o}[I_{\bad}(X_o,\partitions)].
\end{align*}
To conclude, we first note that $2e^{-2\Delta^2 kL}=\frac{\d^2}{4}<\frac{\d^2}{2}$. 
Finally, to bound $P_e(k,M;\g,\tau)$ by $\frac{\d^2}{2}$, we invoke the latter part of Lemma~\ref{lem:key_selective}. 
This completes the proof of Theorem~\ref{thm:classification_general_upper_bound_selective}.
\end{proof}

Now we prove Theorem~\ref{thm:classification_rate_margin_selective}(a).
Recall that we let $\partial\eta_\Delta\defeq \{x\in\supp(\mu)\suchthat |\eta(x)-\half|\le \Delta\}$ to denote the decision boundary with margin $\Delta\ge 0$.
Under the smoothness of the measure $\mu$, the effective decision boundary $\partial_{p,\Delta}$ is a subset of the decision boundary with a certain margin as stated below:

\begin{lemma}[{\citealp[Lemma~18]{Chaudhuri--Dasgupta2014}}]
\label{lem:smooth_margin}
If $\eta$ is $(\alpha,A)$-smooth in $(\Xc,\rho,\mu)$, then for any $p\in[0,1]$ and $\Delta\in (0,\half]$, we have
$\partial_{p,\Delta}
\subset 
\partial\eta_{\Delta+Ap^\a}$.
\end{lemma}
Set $\Delta=\min(\half,\sqrt{\frac{1}{kL}\log\frac{2}{\d}})$. 
Since we choose $p=\frac{1}{1-\g}\frac{k}{n}=\frac{1}{1-\g}\frac{kM}{N}$ in Theorem~\ref{thm:classification_general_upper_bound_selective},
under the $(\b,B)$-margin condition (Assumption~\ref{assum:margin}), the general upper bound in Theorem~\ref{thm:classification_general_upper_bound} and Lemma~\ref{lem:smooth_margin} implies that, with probability $\ge 1-\d$,
\begin{align*}
\Pr(\selectcls{k}{M}{L}(X)\neq g(X)|\partitions)
&\le \d+
\mu(\partial_{p,\Delta})\\
&\le \d+ B\Bigl(\sqrt{\frac{1}{kL}\log\frac{2}{\d}} + A\Bigl(\frac{1}{1-\g}\frac{kM}{N}\Bigr)^{\a}\Bigr),
\end{align*}
provided that either $kM\ge \frac{2c}{c-1}\log\frac{2}{\d^2}$ and $k\ge c\frac{2-\tau^2}{\g^2\tau}$ for some $c>1$, or $k\ge 1$, $M\ge \frac{1}{\phi_{k;\g,\tau}}\log\frac{2}{\d^2}$, and $\tau\in(\tau_{k;\g}^{\inf},1)$.
This bound is optimized if we choose $M=\Th(N^{\frac{2\a}{2\a+1}}(\log\frac{1}{\d})^{\frac{1}{2\a+1}})$, which yields the rate $\mu(\partial_{p,\Delta})=O((\frac{1}{N}\log\frac{1}{\d})^{\frac{1}{2\a+1}})$.
\qed

\subsubsection{Proof of Theorem~\ref{thm:classification_rate_margin_selective}(b) Expected Risk Bound}
This proof modifies that of \citep[Theorem~4]{Chaudhuri--Dasgupta2014} in accordance with Lemma~\ref{lem:sufficient_selective} in place of \citep[Lemma~7]{Chaudhuri--Dasgupta2014}.
Set $L=\lceil(1-\tau)(1-e^{-\frac{\g^2}{2}k}) M\rceil$, $p=\frac{1}{1-\g}\frac{M}{N}$ as in Lemma~\ref{lem:sufficient_selective} and Theorem~\ref{thm:classification_general_upper_bound_selective}, respectively, and define $\Delta_o=Ap^\a$.

We first state and prove the following lemma,
which is as a distributed counterpart to \citep[Lemma~20]{Chaudhuri--Dasgupta2014}.
\begin{lemma}
\label{lem:cd14_lemma20_selective}
For any $x_o\in \supp(\mu)$ with $\Delta(x_o)> \Delta_o$. Under the $(\a,A)$-smoothness (Assumption~\ref{assum:smooth}) condition, we have
\begin{align*}
\E_{\partitions}[R(x_o;\selectcls{k}{M}{L})]-R^*(x_o) 
&\le P_e(k,M;\g,\tau) + 4\Delta(x_o)e^{-2(\Delta(x_o)-\Delta_o)^2kL}.
\end{align*}
\end{lemma}
\begin{proof}[Proof of Lemma~\ref{lem:cd14_lemma20_selective}]
Without loss of generality, assume that $\eta(x_o)>\half$. 
By the smoothness condition, for any $0\le r\le r_p(x_o)$, we have 
\[
\eta(\closedball(x_o,r))\ge \eta(x_o)-Ap^\a=\eta(x_o)-\Delta_o=\half+(\Delta(x_o)-\Delta_o),
\]
which implies $x_o\in \Xc_{p,\Delta(x_o)-\Delta_o}^+$ and thus $x_o\notin \partial_{p,\Delta(x_o)-\Delta_o}$.

Recall that for any classifier $\gh$, we can write $R(x_o;\gh)-R^*(x_o)=2\Delta(x_o)1(\gh(x_o)\neq \gbayes(x_o))$, where $R^*(x_o)$ is the Bayes risk.
We can apply Lemma~\ref{lem:sufficient_selective} with $\Delta\gets \Delta(x_o)-\Delta_o$ and have 
\begin{align}
R(x_o,\selectcls{k}{M}{L})-R^*(x_o) 
&= 2\Delta(x_o) 1(\selectcls{k}{M}{L}(x_o)\neq \gbayes(x_o)) \\
&\le 2\Delta(x_o) I_{\bad}(x_o;\partitions),
\label{eq:star_selective}
\end{align} where we define the bad-event indicator variable as
\begin{align*}
I_{\bad}(x_o;\partitions)
&\defeq \ones\Bigl(\max_{j\in [L]} r_{k+1}(x_o;\Xb_{m_j})> r_p(x_o)\Bigr)
\\&\quad
+\ones\Bigl(\Bigl|\frac{1}{L }\sum_{j=1}^{L }(\Yh(\ball_{m_j};\dataset_{m_j})-\eta(\ball_{m_j}))\Bigr|\ge {\Delta(x_o)-\Delta_o}\Bigr).
\end{align*}
By taking the expectations over the random splits $\partitions$ in \eqref{eq:star_selective}, we have $$ \E_{\partitions}R(x_o;\selectcls{k}{M}{L}) - R^*(x_o) \le 2\Delta(x_o) \E[I_{\bad}(x_o;\partitions)].$$ 
Now, by applying Lemma~\ref{lem:key_selective} and Lemma~\ref{lem:label_concentration_new} as in the proof of Theorem~\ref{thm:classification_general_upper_bound_selective}, we can bound the right hand side as 
\begin{align*}
\E_{\partitions} R(x_o;\selectcls{k}{M}{L})-R^*(x_o)
&\le 2\Delta(x_o)\bigl(P_e(k,M;\g,\tau) + 2e^{-2(\Delta(x_o)-\Delta_o)^2kL}\bigr)\\
&\le P_e(k,M;\g,\tau) + 4\Delta(x_o)e^{-2(\Delta(x_o)-\Delta_o)^2kL},
\end{align*} where the last inequality follows from the assumption that $\Delta(x_o)\le \half$. 
\end{proof}

We then prove the following statement under the smoothness and margin conditions, which is as a counterpart of \citep[Lemma~21]{Chaudhuri--Dasgupta2014} for the $(k,M,L)$-NN classifier, and Theorem~\ref{thm:classification_rate_margin_selective}(b) immediately follows as its corollary.
\begin{lemma}
\label{lem:cd14_lemma21_selective}
Under the $(\a,A)$-smoothness (Assumption~\ref{assum:smooth}) and the $(\b,B)$-margin (Assumption~\ref{assum:margin}) conditions, we have
\begin{align*}
\E_{\partitions}R(\selectcls{k}{M}{L}) - R^* 
&\le 
P_e(k,M;\g,\tau) 
+ 6B
\max\biggl(
2A\Bigl(\frac{1}{1-\g}\frac{kM}{N}\Bigr)^\a,
\sqrt{\frac{8(\b+2)}{kL}}
\biggr)^{\b+1}.
\numberthis\label{eq:lem:cd14_lemma21_selective}
\end{align*}
\end{lemma}

\begin{proof}[Proof of Lemma~\ref{lem:cd14_lemma21_selective}]
For each integer $i\ge 1$, define $\Delta_i=2^i\Delta_o$. 
Fix any $i_o\ge 1$.
To bound the expected risk,
we apply Lemma~\ref{lem:cd14_lemma20_selective} for any $x_o$ with $\Delta(x_o)>\Delta_{i_o}$ and use $\E_{\partitions}R(x_o;\selectcls{k}{M}{L})-R^*(x_o)\le 2\Delta_{i_o}$ for all remaining $x_o$.
Taking expectations over $X_o$, we have
\begin{align*}
\E_{\partitions}R(\selectcls{k}{M}{L}) - R^*
&\le \E_{X_o}\Bigl[2\Delta_{i_o}1(\Delta(X_o)\le\Delta_{i_o}) 
\\&\qquad\qquad 
+ P_e(k,M;\g,\tau)
\\&\qquad\qquad
+ 4\Delta(X_o)e^{-2(\Delta(x_o)-\Delta_o)^2kL}1(\Delta(X_o)>\Delta_{i_o})\Bigr]\\
&\le 2B\Delta_{i_o}^{\b+1} \numberthis\label{eq:interim_bound_selective}
\\&\quad + P_e(k,M;\g,\tau) 
\\&\quad + 4\E_{X_o}\Bigl[\Delta(X_o)e^{-2(\Delta(x_o)-\Delta_o)^2kL}1(\Delta(X_o)>\Delta_{i_o})\Bigr].
\end{align*}
Here, we invoke the $(\b,B)$-margin condition in the second inequality to bound the first term.
It only remains to bound the last term.
First, by another application of the $(\b,B)$-margin condition, we have 
\begin{align*}
&
\E_{X_o}\Bigl[\Delta(X)e^{-2(\Delta(x_o)-\Delta_o)^2kL}1(\Delta_{i}< \Delta(X) \le \Delta_{i+1})\Bigr]\\
&\le 
\E_{X_o}\Bigl[\Delta_{i+1}e^{-2(\Delta_i-\Delta_o)^2kL}1(\Delta(X) \le \Delta_{i+1})\Bigr]\\
&\le 
B\Delta_{i+1}^{\b+1}e^{-2(\Delta_i-\Delta_o)^2kL}.\numberthis\label{eq:bound_interval_selective}
\end{align*}
Now, we set
\[
i_o=\max\Bigl(1, \Bigl\lceil \log_2\sqrt{\frac{2(\b+2)}{kL\Delta_o^2}}\Bigr\rceil\Bigr),
\]
so that the terms~\eqref{eq:bound_interval_selective} are upper-bounded by a geometric series with ratio $\half$. Indeed, for $i\ge i_o$, we have
\begin{align*}
&
\frac{\Delta_{i+1}^{\b+1}\exp(-2(\Delta_i-\Delta_o)^2kL)}{\Delta_{i}^{\b+1}\exp(-2(\Delta_{i-1}-\Delta_o)^2kL)}
\\
&= 2^{\b+1}
    \exp\Bigl(-2\bigl\{(\Delta_i-\Delta_o)^2 - (\Delta_{i-1}-\Delta_o)^2\bigr\}kL\Bigr)\\
&\le 2^{\b+1}
    \exp\Bigl(-2\bigl\{((2^i-1)\Delta_o+8\tau)^2 
    - ((2^{i-1}-1)\Delta_o +8\tau)^2\bigr\}kL\Bigr)\\
&= 2^{\b+1}
    \exp\Bigl(-2\bigl\{\Delta_o^2((2^i-1)^2-(2^{i-1}-1)^2)
    +16\Delta_o\tau(2^i-2^{i-1})\bigr\}kL\Bigr)\\
&\le 2^{\b+1}
    \exp(-2^{2i-1}\Delta_o^2 kL)\\
&\le 2^{\b+1}
    \exp(-(\b+2))
\le \half.
\end{align*}
Therefore, we can bound the last term in \eqref{eq:interim_bound_selective} as
\begin{align*}
&
\E[\Delta(X_o)e^{-2(\Delta(X_o)-\Delta_o)^2kL}1(\Delta(X_o)>\Delta_{i_o})]\\
&= \sum_{i=i_o}^{\infty} \E[\Delta(X_o)e^{-2(\Delta(X_o)-\Delta_o)^2kL}\ones(\Delta(X_o)\in(\Delta_{i},\Delta_{i+1}]))]\\
&\le B\sum_{i=i_o}^\infty \Delta_{i+1}^{\b+1}e^{-2(\Delta_i-\Delta_o)^2kL}
\le B\Delta_{i_o}^{\b+1}.
\end{align*}
Plugging this back into \eqref{eq:interim_bound_selective}, we have 
$\E_{\partitions}R(\selectcls{k}{M}{L}) - R^* 
\le P_e(k,M;\g,\tau) + 6B\Delta_{i_o}^{\b+1}$.
The desired inequality follows by substituting $\Delta_{i_o}=2^{i_o}\Delta_o = \max\{2Ap^\a, \sqrt{\frac{2(\b+2)}{kL}}\}$.
\end{proof}

\subsubsection{Proof of Theorem~\ref{thm:classification_rate_margin_selective}(b) CIS bound}
The proof is an easy modification of the previous proof of the expected risk bound.
Observe that the classification instability is upper-bounded as
\[
\CIS_N(\gh)
\le 2\E_{\dataset}[\Pr_X(\gh(X;\dataset)\neq \gbayes(X))]
\]
for any classification procedure $\gh(\cdot;\dataset)$.
Hence, following the exact same line of the proof of Lemma~\ref{lem:cd14_lemma20_selective}, we have 
\begin{lemma}\label{lem:cd14_lemma20_var_selective}
For any $x_o\in \supp(\mu)$ with $\Delta(x_o)> \Delta_o+8\tau$. Under the $(\alpha,A)$-smoothness condition, we have
\begin{align*}
\E_{\dataset}[1(\selectcls{k}{M}{L}(x_o;\dataset)\neq \gbayes(x_o))] 
\le P_e(k,M;\g,\tau) + 4e^{-2(\Delta(x_o)-\Delta_o)^2kL}.
\end{align*}
\end{lemma}
We then follow the same line of the proof of Lemma~\ref{lem:cd14_lemma21_selective}. 
For each integer $i\ge 1$, define $\Delta_i=2^i\Delta_o$. 
Fix any $i_o\ge 1$.
To bound the expected probability of the mismatch $\E_{\dataset}[\partitions_X(\selectcls{k}{M}{L}(X_o;\dataset)\neq \gbayes(X_o))]$,
we will apply Lemma~\ref{lem:cd14_lemma20_var_selective} for any $x_o$ with $\Delta(x_o)>\Delta_{i_o}$ and use a trivial bound $\E_{\dataset}[1(\selectcls{k}{M}{L}(x_o;\dataset)
\neq \gbayes(x_o))]\le 1$ for all remaining $x_o$.
Taking expectations over $X_o$ and invoking the $(\b,B)$-margin condition, we have
\begin{align*}
&
\E_{\dataset}[\partitions_{X_o}(\selectcls{k}{M}{L}(X_o;\dataset)\neq \gbayes(X_o))]\\
&\le \E_{X_o}\Bigl[1(\Delta(X_o)\le\Delta_{i_o}) 
+ P_e(k,M;\g,\tau)
+ 4e^{-2(\Delta(x_o)-\Delta_o)^2kL} 1(\Delta(X_o)>\Delta_{i_o})\Bigr]\\
&\le 
B\Delta_{i_o}^{\b}
+ P_e(k,M;\g,\tau)
+ 4\E_{X_o}\Bigl[e^{-2(\Delta(x_o)-\Delta_o)^2kL} 1(\Delta(X_o)>\Delta_{i_o})\Bigr].
\numberthis\label{eq:interim_bound_cis_selective}
\end{align*}
By the same logic in the proof of Lemma~\ref{lem:cd14_lemma21_selective}, the last term can be bounded by $B\Delta_{i_o}^{\beta}$ with the same choice of $i_o$. 
Plugging this back into \eqref{eq:interim_bound_cis_selective}, we have
\begin{align*}
\E_{\dataset}[\partitions_{X_o}(\selectcls{k}{M}{L}(X_o;\dataset)\neq \gbayes(X_o))]
&\le P_e(k,M;\g,\tau) + 5B\Delta_{i_o}^{\b}.
\end{align*}
By substituting $\Delta_{i_o}=2^{i_o}\Delta_o$, we have
\begin{align*}
\CIS_N(\selectcls{k}{M}{L})
&\le 2\E_{\dataset}[\Pr_X(\selectcls{k}{M}{L}(X_o;\dataset)\neq \gbayes(X_o))]\\
&\le 2P_e(k,M;\g,\tau) 
+ 10B\max\biggl(
2A\Bigl(\frac{1}{1-\g}\frac{kM}{N}\Bigr)^\a,
\sqrt{\frac{8(\b+2)}{kL}}
\biggr)^{\b}
\end{align*}
and setting $\tau=\sqrt{\frac{(\log M)^{1+\eps}}{M}}$ for some $\eps>0$ concludes the proof of the CIS bound in Theorem~\ref{thm:classification_rate_margin_selective}(b).
\qed

\subsection{Analysis of the \texorpdfstring{$(k,M)$}{(k,M)}-NN Classification Rule}

\subsubsection{A Key Technical Lemma}
The following statement is a variant of Lemma~\ref{lem:sufficient_selective} for the $(k,M)$-NN rule.
Its proof almost exactly follows that of Lemma~\ref{lem:sufficient_selective}. The high-level idea is that given the margin parameter $\Delta$, if $\tau$ is chosen sufficiently small (so that $L$ is close to $M$), then the $(k,M)$-NN rule can be analyzed via the $(k,M,L)$-NN rule. 
Note that the margin condition, which is in the last term, becomes tighter to ensure that the two rules result in the same decision.

\begin{lemma}
\label{lem:sufficient}
For any $x_o\in\Xc$, pick any $p\in(0,1)$ and $\Delta\in(0,\half]$. 
For each $m\in[M]$, 
define $\ball_m\defeq \openball(x_o,r_{k+1}(x_o;\Xb_m))$.
Pick $\tau\in(0,\frac{\Delta}{8}]$ and
let $L\defeq \lceil (1-\tau)^2 M\rceil$.
Then, we have
\begin{align*}
\ones(\splitcls{k}{M}(x_o;\partitions)\neq \gbayes(x_o))
&\le \ones(x_o\in\partial_{p,\Delta})\\
&\quad+\ones\Bigl(\max_{j\in [L]} r_{k+1}(x_o;\Xb_{m_j})> r_p(x_o)\Bigr)\\
&\quad+\ones\Bigl(\Bigl|\frac{1}{L }\sum_{j=1}^{L }(\Yh(\ball_{m_j};\dataset_{m_j})-\eta(\ball_{m_j}))\Bigr|\ge \frac{\Delta}{2}\Bigr),
\numberthis\label{eq:sufficient_classification}
\end{align*}
where $m_1,\ldots,m_L$ are the indices that correspond to the $L$-smallest values among $(r_{k+1}(x_o;\Xb_m))_{m=1}^M$.
\end{lemma}

\begin{proof}%
Suppose $x_o\notin\partial_{p,\Delta}$.
Without loss of generality, consider $x_o\in\Xc_{p,\Delta}^+$, whereupon $\gbayes(x_o)=1$.
By definition of the effective interior, 
$\eta(\closedball(x_o,r))\ge \half+\Delta$ for all $r\le r_p(x_o)$.
If we further suppose
\[
\max_{j\in [L]} r_{k+1}(x_o;\Xb_{m_j})\le r_p(x_o),
\]
we have
\[
\eta(\ball_{m_j})
=\eta(\openball(x_o,r_{k+1}(x_o;\Xb_{m_j})))
\ge \half+\Delta
\]
for any $j\in[L]$, by Lemma~\ref{lem:lem26_cd14}.

Further, if $|\frac{1}{L}\sum_{j=1}^{L }(\Yh(\ball_{m_j};\dataset_{m_j})-\eta(\ball_{m_j}))|< \frac{\Delta}{2}$, then 
\[
\selectreg{k}{M}{L}(x_o;\partitions)
=\frac{1}{L}\sum_{j=1}^L \Yh(\ball_{m_j};\dataset_{m_j}) > \half+\frac{\Delta}{2}.
\]

{Finally, since $|\splitreg{k}{M}(x_o;\partitions)-\selectreg{k}{M}{L}(x_o;\partitions)|\le 2(1-\frac{L}{M}) \le 2(2\tau-\tau^2) <4\tau\le \frac{\Delta}{2}$ by the choice of $\tau\le \frac{\Delta}{8}$, we have $\splitreg{k}{M}(x_o;\partitions)>\half$,
which concludes $\splitcls{k}{M}(x_o;\partitions)=\selectcls{k}{M}{L}(x_o;\partitions)=1=\gbayes(x_o)$.}
\end{proof}

\subsubsection{Proof of Theorem~\ref{thm:classification_rate_margin}(a)}
We first present a more general upper bound on the misclassification error rate, as was done for analyzing the $(k,M,L)$-NN rule.
Theorem~\ref{thm:classification_rate_margin}(a) will follow as a corollary.

\begin{theorem}
\label{thm:classification_general_upper_bound}
Let $k\ge 1$ be fixed and
pick any $\d\in(0,1)$.
Pick any integer $M\ge \frac{2^{14}}{15}\log\frac{2}{\d}$, set $\Delta\defeq\sqrt{\frac{2^{12}}{15M}\log\frac{2}{\d}}\in (0,\half]$.
Pick any integer $n\ge k+\log\frac{8}{\Delta}+\sqrt{2k\log\frac{8}{\Delta}+(\log\frac{8}{\Delta})^2}$ and set 
$p\defeq \frac{1}{n}(k+\log\frac{8}{\Delta}+\sqrt{2k\log\frac{8}{\Delta}+(\log\frac{8}{\Delta})^2})\in(0,1]$.
Then, for a set of data splits $\partitions=\{\dataset_1,\ldots,\dataset_M\}$, where every split $\dataset_m$ has $n$ data points, with probability at least $1-\d$ over $\partitions$, we have
\begin{align*}
\Pr(\splitcls{k}{M}(X;\partitions)\neq \gbayes(X)|\partitions)
\le \d + \mu(\partial_{p,\Delta}).
\end{align*}
\end{theorem}

\begin{proof}
Given $k\ge 1$, $\delta\in(0,1)$, and $\Delta\in(0,\half]$, we set $\tau=\frac{\Delta}{8}$ and define $L=\lceil (1-\tau)^2M\rceil$ as stated in Lemma~\ref{lem:sufficient}.
Pick any $x_o\in\Xc$. 
Applying Lemma~\ref{lem:sufficient}, we have
$$
\begin{aligned}
\ones(\splitcls{k}{M}(x_o;\partitions)\neq \gbayes(x_o))
&\le \ones(x_o\in\partial_{p,\Delta})+ I_{\bad}(x_o;\partitions),
\end{aligned}$$ where we define the bad event indicator variable
\begin{align*}
I_{\bad}(x_o;\partitions)
&\defeq \ones\Bigl(\max_{j\in [L]} r_{k+1}(x_o;\Xb_{m_j})> r_p(x_o)\Bigr)
\\&\quad+
\ones\Bigl(\Bigl|\frac{1}{L }\sum_{j=1}^{L }(\Yh(\ball_{m_j};\dataset_{m_j})-\eta(\ball_{m_j}))\Bigr|\ge \frac{\Delta}{2}\Bigr),
\numberthis\label{eq:bad_indicator}
\end{align*}
where $m_1,\ldots,m_L$ are the indices for the $L$ smallest distances among $\{r_{k+1}(x_o;\Xb_m)\}_{m=1}^M$. For any fixed point $x_o\in\Xc$, if we take the expectation over the training data splits $\partitions$, we have
\begin{align*}
\E[I_{\bad}(x_o;\partitions)]
&\defeq \P\Bigl(\max_{j\in [L]} r_{k+1}(x_o;\Xb_{m_j})> r_p(x_o)\Bigr)
\\&\quad+
\P\Bigl(\Bigl|\frac{1}{L }\sum_{j=1}^{L }(\Yh(\ball_{m_j};\dataset_{m_j})-\eta(\ball_{m_j}))\Bigr|\ge \frac{\Delta}{2}\Bigr).
\numberthis\label{eq:exp_ind_bad_event}
\end{align*}
Now, by applying Lemma~\ref{lem:key_original} and Lemma~\ref{lem:label_concentration_new} on the first and second terms, respectively, we have
\[
\E[I_{\bad}(x_o,\partitions)] \le
e^{-\frac{(1-\tau)\tau^2}{2}M}
+2e^{-\frac{\Delta^2}{2}kL}.
\]
Since $\Delta=\sqrt{\frac{2^{12}}{15M}\log\frac{2}{\delta}}\le\half$, 
we have $\tau=\frac{\Delta}{8}\le \frac{1}{16}$, which implies that
\[
\frac{(1-\tau)\tau^2}{2}M \ge \Bigl(1-\frac{\Delta}{8}\Bigr)\frac{\Delta^2M}{2^7}
\ge \frac{15}{2^{11}}\Delta^2M
=2\log\frac{2}{\d}.
\]
Moreover, 
\[
\frac{\Delta^2}{2}kL
=\frac{2^{11}}{15}\frac{L}{M}\log\frac{2}{\d}
>\frac{2^{11}}{15}\frac{3}{4}\log\frac{2}{\d}
> 2\log\frac{2}{\d},
\]
since $\frac{L}{M}\ge (1-\frac{\Delta}{8})^2-\frac{1}{M} \ge 1-\frac{\Delta}{4}-\frac{1}{M}\ge \frac{7}{8}-\frac{1}{M}> \frac{3}{4}$ for $M\ge \frac{2^{14}}{15}\log\frac{2}{\d}>2$.
Therefore, we can further upper bound the expectation as
\[
\E_{\partitions}[I_{\bad}(x_o,\partitions)] \le \frac{\d^2}{4} +\frac{\d^2}{2} < \delta^2.
\]
The conclusion follows from the same argument using Markov's inequality in \eqref{eq:markov} in the proof of Theorem~\ref{thm:classification_general_upper_bound_selective}.
\end{proof}

We now prove Theorem~\ref{thm:classification_rate_margin}(a).
Set $\Delta=\min(\half,\sqrt{\frac{2^{12}}{15M}\log\frac{2}{\d}})$. 
To apply Theorem~\ref{thm:classification_general_upper_bound}, let
\[
p\defeq \frac{1}{n}\Bigl(k+\log\frac{8}{\Delta}+\sqrt{2k\log\frac{8}{\Delta}+(\log\frac{8}{\Delta})^2}\Bigr).
\]
Here, as we assume that the sample size $n$ of each split satisfies
$n=\frac{N}{M}\ge 2k + \log (\frac{15}{2^6}M\log\frac{2}{\d}) = 2(k+\log\frac{8}{\Delta}) \ge k+\log\frac{8}{\Delta}+\sqrt{2k\log\frac{8}{\Delta}+(\log\frac{8}{\Delta})^2}$,
we have
\[
p=O\Bigl(\frac{1}{n}\bigl(k+\log\frac{M}{\log\frac{2}{\d}}\bigr)\Bigr).
\]
Under the $(\b,B)$-margin condition, Theorem~\ref{thm:classification_general_upper_bound} and Lemma~\ref{lem:smooth_margin} imply that 
\begin{align*}
\Pr(\splitcls{k}{M}(x)\neq g(X)|\partitions)
&\le \d + \mu(\partial_{p,\Delta})\\
&\le \d + B\Bigl(\sqrt{\frac{2^{12}}{15M}\log\frac{2}{\d}} + Ap^{\a}\Bigr)^\b\\
&= \d + O\Bigl(\sqrt{\frac{1}{M}\log\frac{2}{\d}}+ \bigl(\frac{M}{N}\log\frac{M}{\log\frac{2}{\d}}\bigr)^\a\Bigr)^\b,
\end{align*}
where we assume $k=O(1)$ in the last bound.
This bound is optimized if we choose $M=\Th(N^{\frac{2\a}{2\a+1}}(\log\frac{2}{\d})^{\frac{1}{2\a+1}})$, which yields the rate $\mu(\partial_{p,\Delta})=O((\frac{\log N}{N}\log\frac{2}{\d})^{\frac{\a\b}{2\a+1}})$.
\qed 

\subsubsection{Proof of Theorem~\ref{thm:classification_rate_margin}(b)}
The part (b) has two bounds on the expected risk and CIS. 
The proof for the CIS bound is similar to the analysis of the expected risk.

\paragraph{Proof of Expected Risk Bound}
This proof modifies that of \citep[Theorem~4]{Chaudhuri--Dasgupta2014} in accordance with Lemma~\ref{lem:sufficient} instead of \citep[Lemma~7]{Chaudhuri--Dasgupta2014}. 
We choose an arbitrary $\tau\in(0,1)$ for now, and will set it to a specific value at the end of the analysis.
Set $L=\lceil(1-\tau)^2 M\rceil$ and  $p=\frac{M}{N}(k+\log\frac{1}{\tau}+\sqrt{2k\log\frac{1}{\tau}+(\log\frac{1}{\tau})^2})\le 2\frac{M}{N}(k+\log\frac{1}{\tau})$ for $\tau=\frac{\Delta}{8}$ as in Lemma~\ref{lem:key_original} and Lemma~\ref{lem:sufficient}, respectively, and define $\Delta_o=Ap^\a$.
We further assume that the sample size $n$ of each split satisfies
$n=\frac{N}{M}\ge 2(k+\log\frac{1}{\tau}) \ge k+\log\frac{1}{\tau}+\sqrt{2k\log\frac{1}{\tau}+(\log\frac{1}{\tau})^2}$.

We first state and prove the following lemma,
which is as a distributed counterpart to \citep[Lemma~20]{Chaudhuri--Dasgupta2014} and a variant of Lemma~\ref{lem:cd14_lemma20_selective}.
Note that, compared to Lemma~\ref{lem:cd14_lemma20_selective}, 
we consider points $x_o$ which have an additional margin of $8\tau$ in $\Delta(x_o)$.
The rest of the proof needs to be modified to handle this slackness.
\begin{lemma}
\label{lem:cd14_lemma20}
For any $x_o\in \supp(\mu)$ with $\Delta(x_o)\ge \Delta_o+8\tau$. Under the $(\a,A)$-smoothness (Assumption~\ref{assum:smooth}) condition, 
if $n\ge 2(k+\log\frac{1}{\tau})$,
we have
\begin{align*}
\E_{\partitions}[R(x_o;\splitcls{k}{M})]-R^*(x_o) 
&\le e^{-\frac{(1-\tau)\tau^2}{2}M} + 4\Delta(x_o)e^{-\frac{(\Delta(x_o)-\Delta_o)^2}{8}M}.
\end{align*}
\end{lemma}
\begin{proof}[Proof of Lemma~\ref{lem:cd14_lemma20}]
Without loss of generality, assume that $\eta(x_o)>\half$. 
By the smoothness condition, for any $0\le r\le r_p(x_o)$, we have 
\[
\eta(\closedball(x_o,r))\ge \eta(x_o)-Ap^\a=\eta(x_o)-\Delta_o=\half+(\Delta(x_o)-\Delta_o),
\]
which implies $x_o\in \Xc_{p,\Delta(x_o)-\Delta_o}^+$ and thus $x_o\notin \partial_{p,\Delta(x_o)-\Delta_o}$.

Recall that for any classifier $\gh$, we can write $R(x_o;\gh)-R^*(x_o)=2\Delta(x_o)1(\gh(x_o)\neq \gbayes(x_o))$, where $R^*(x_o)$ is the Bayes risk.
Since we assume that $\tau\le \frac{\Delta(x_o)-\Delta_o}{8}$, we can apply Lemma~\ref{lem:key_original} with $\Delta\gets \Delta(x_o)-\Delta_o$ and have 
\begin{align*}
R(x_o,\splitcls{k}{M})-R^*(x_o) 
&= 2\Delta(x_o) 1(\splitcls{k}{M}(x_o)\neq \gbayes(x_o)) \\
&\le 2\Delta(x_o) I_{\bad}(x_o;\partitions),
\numberthis\label{eq:star}
\end{align*} where we define the bad-event indicator variable as
\begin{align*}
I_{\bad}(x_o;\partitions)
&\defeq \ones\Bigl(\max_{j\in [L]} r_{k+1}(x_o;\Xb_{m_j})> r_p(x_o)\Bigr)
\\&\quad
+\ones\Bigl(\Bigl|\frac{1}{L }\sum_{j=1}^{L }(\Yh(\ball_{m_j};\dataset_{m_j})-\eta(\ball_{m_j}))\Bigr|\ge \frac{\Delta(x_o)-\Delta_o}{2}\Bigr),
\end{align*}
as in \eqref{eq:bad_indicator} in Lemma~\ref{lem:sufficient} with $\Delta\gets \Delta(x_o)-\Delta_o$.
By taking an expectation over the random splits $\partitions$ in \eqref{eq:star}, we have 
\begin{align*} &\E_{\partitions}R(x_o;\splitcls{k}{M}) - R^*(x_o) \le 2\Delta(x_o) \E[I_{\bad}(x_o;\partitions)].
\end{align*} 
Now, by applying Lemma~\ref{lem:key_original} and Lemma~\ref{lem:label_concentration_new}, we can bound the right hand side as 
\begin{align*}
\E_{\partitions} R(x_o;\splitcls{k}{M})-R^*(x_o)
&\le 2\Delta(x_o)(e^{-\frac{(1-\tau)\tau^2}{2}M} + 2e^{-\frac{(\Delta(x_o)-\Delta_o)^2}{8}M})\\
&\le e^{-\frac{(1-\tau)\tau^2}{2}M} + 4\Delta(x_o)e^{-\frac{(\Delta(x_o)-\Delta_o)^2}{8}M},
\end{align*} where the last inequality follows from the assumption that $\Delta(x_o)\le \half$. 
\qed

Under the smoothness and margin conditions
we can prove the following statement, which is a counterpart of \citep[Lemma~21]{Chaudhuri--Dasgupta2014} for the $(k,M)$-NN classifier.
\begin{lemma}
\label{lem:cd14_lemma21}
Under the $(\a,A)$-smoothness (Assumption~\ref{assum:smooth}) and the $(\b,B)$-margin (Assumption~\ref{assum:margin}) conditions, we have
\begin{align*}
&
\E_{\partitions}R(\splitcls{k}{M}) - R^* \\
&\le 
e^{-\frac{(1-\tau)\tau^2}{2}M} 
\numberthis\label{eq:lem:cd14_lemma21}
+ 6B\Bigl(
\max\Bigl\{
2A\Bigl(2\frac{M}{N}\bigl(k+\log\frac{1}{\tau}\bigr)\Bigr)^\a,
\sqrt{\frac{32(\b+2)}{M}}
\Bigr\}
+8\tau
\Bigr)^{\b+1}.
\end{align*}
\end{lemma}

\begin{proof}[Proof of Lemma~\ref{lem:cd14_lemma21}]
For each integer $i\ge 1$, define $\Delta_i=2^i\Delta_o + 8\tau$. 
Fix any $i_o\ge 1$.
To bound the expected risk,
we apply Lemma~\ref{lem:cd14_lemma20} for any $x_o$ with $\Delta(x_o)>\Delta_{i_o}$ and use $\E_{\partitions}R(x_o;\splitcls{k}{M})-R^*(x_o)\le 2\Delta_{i_o}$ for all remaining $x_o$.
Taking expectations over $X_o$, we have
\begin{align*}
\E_{\partitions}R(\splitcls{k}{M}) - R^*
&\le \E_{X_o}\Bigl[2\Delta_{i_o}1(\Delta(X_o)\le\Delta_{i_o}) 
+ e^{-\frac{(1-\tau)\tau^2}{2}M}
\\&\qquad\qquad+ 4\Delta(X_o)e^{-\frac{(\Delta(X_o)-\Delta_o)^2}{8}M}1(\Delta(X_o)>\Delta_{i_o})\Bigr]\\
&\le 2B\Delta_{i_o}^{\b+1} 
+ e^{-\frac{(1-\tau)\tau^2}{2}M} 
\numberthis\label{eq:interim_bound}
\\&\qquad
+ 4\E_{X_o}[\Delta(X_o)e^{-\frac{(\Delta(X_o)-\Delta_o)^2}{8}M}1(\Delta(X_o)>\Delta_{i_o})].
\end{align*}
Here, we invoke the $(\b,B)$-margin condition in the second inequality to bound the first term.
It only remains to bound the last term.
First, by another application of the $(\b,B)$-margin condition, we have 
\begin{align*}
\E_{X_o}\Bigl[\Delta(X)e^{-\frac{(\Delta(X)-\Delta_o)^2}{8}M}1(\Delta_{i}< \Delta(X) \le \Delta_{i+1})\Bigr]
&\le 
\E_{X_o}\Bigl[\Delta_{i+1}e^{-\frac{(\Delta_i-\Delta_o)^2}{8}M}1(\Delta(X) \le \Delta_{i+1})\Bigr]\\
&\le 
B\Delta_{i+1}^{\b+1}e^{-\frac{(\Delta_i-\Delta_o)^2}{8}M}.\numberthis\label{eq:bound_interval}
\end{align*}
Now, we set
\[
i_o=\max\Bigl(1, \Bigl\lceil \log_2\sqrt{\frac{32(\b+2)}{M\Delta_o^2}}\Bigr\rceil\Bigr),
\]
so that the terms~\eqref{eq:bound_interval} are upper-bounded by a geometric series with ratio $\half$. Indeed, for $i\ge i_o$, we have
\begin{align*}
&\frac{\Delta_{i+1}^{\b+1}\exp(-\frac{M}{8}(\Delta_i-\Delta_o)^2)}{\Delta_{i}^{\b+1}\exp(-\frac{M}{8}(\Delta_{i-1}-\Delta_o)^2)}\\
&= \Bigl(\frac{2^{i+1}\Delta_o}{2^i\Delta_o}\Bigr)^{\b+1}
    \exp\Bigl(-\frac{M}{8}\bigl\{(\Delta_i-\Delta_o)^2 - (\Delta_{i-1}-\Delta_o)^2\bigr\}\Bigr)\\
&\le 2^{\b+1}
    \exp\Bigl(-\frac{M}{8}\bigl\{((2^i-1)\Delta_o)^2 - ((2^{i-1}-1)\Delta_o)^2\bigr\}\Bigr)\\
&= 2^{\b+1}
    \exp\Bigl(-\frac{M}{8}\bigl\{\Delta_o^2((2^i-1)^2-(2^{i-1}-1)^2)
    \\&\qquad\qquad\qquad\qquad
    +16\Delta_o\tau(2^i-2^{i-1})\bigr\}\Bigr)\\
&\le 2^{\b+1}
    \exp(-M\Delta_o^2 2^{2i-5})\\
&\le 2^{\b+1}
    \exp(-(\b+2))
\le \half.
\end{align*}
Therefore, we can bound the last term in \eqref{eq:interim_bound} as
\begin{align*}
&\E_{X_o}[\Delta(X_o)e^{-\frac{(\Delta(X_o)-\Delta_o)^2}{8}M}1(\Delta(X_o)>\Delta_{i_o})]\\
&= \sum_{i=i_o}^{\infty} \E_{X_o}[\Delta(X_o)e^{-\frac{(\Delta(X_o)-\Delta_o)^2}{8}M}1(\Delta_{i}< \Delta(X_o) \le \Delta_{i+1})]\\
&\le B\sum_{i=i_o}^\infty \Delta_{i+1}^{\b+1}e^{-\frac{(\Delta_i-\Delta_o)^2}{8}M}
\le B\Delta_{i_o}^{\b+1}.
\end{align*}
Plugging this back into \eqref{eq:interim_bound}, we have 
$\E_{\partitions}R(\splitcls{k}{M}) - R^* 
\le e^{-\frac{(1-\tau)\tau^2}{2}M} + 6B\Delta_{i_o}^{\b+1}$.
The desired inequality follows by substituting $\Delta_{i_o}=2^{i_o}\Delta_o+8\tau= \max\{2Ap^\a, \sqrt{\frac{32(\b+2)}{M}}\}+ 8\tau$.
\end{proof}

Finally, to prove Theorem~\ref{thm:classification_rate_margin}(b), we can set $\tau=\sqrt{\frac{(\log M)^{1+\eps}}{M}}$ for some $\eps>0$ so that the term $e^{-\frac{(1-\tau)\tau^2}{2}M}$ in \eqref{eq:lem:cd14_lemma21} decays faster than any polynomial rate.
Then Lemma~\ref{lem:cd14_lemma21} reduces to
\[
\E_{\partitions}R(\splitcls{k}{M}) - R^* 
=O\Bigl(
\Bigl(\frac{M}{N}\log M\Bigr)^\a
+\sqrt{\frac{(\log M)^{1+\eps}}{M}}
\Bigr)^{\b+1}.
\]
Setting $M\propto N^{\frac{2\a}{2\a+1}}$ leads to the final rate $O(N^{-\frac{\a(\b+1)}{2\a+1}} (\log N)^{\half(1+\eps)(\b+1)})$.
\end{proof}

\paragraph{Proof of CIS Bound}
Since the proof is an easy modification of the previous proof of the expected risk bound, we only outline the critical steps that differ from the proof of Theorem~\ref{thm:classification_rate_margin}(b) regret bound.
Observe that the classification instability is upper-bounded as
\[
\CIS_N(\gh)
\le 2\E_{\dataset}[\Pr_X(\gh(X;\dataset)\neq \gbayes(X))]
\]
for any classification procedure $\gh(\cdot;\dataset)$.
Hence, following the exact same line of the proof of Lemma~\ref{lem:cd14_lemma20}, we have 
\begin{lemma}\label{lem:cd14_lemma20_var}
For any $x_o\in \supp(\mu)$ with $\Delta(x_o)\ge \Delta_o+8\tau$. Under the $(\alpha,A)$-smoothness (Assumption~\ref{assum:smooth}) condition, we have
\[
\E_{\dataset}[1(\splitcls{k}{M}(x_o;\dataset)\neq \gbayes(x_o))] \le e^{-\frac{(1-\tau)\tau^2}{2}M} + 4e^{-\frac{(\Delta(x_o)-\Delta_o)^2}{8}M}.
\]
\end{lemma}
We then follow the same line of the proof of Lemma~\ref{lem:cd14_lemma21}. 
For each integer $i\ge 1$, define $\Delta_i=2^i\Delta_o + 8\tau$. 
Fix any $i_o\ge 1$.
To bound the expected probability of the mismatch $\E_{\dataset}[\partitions_X(\splitcls{k}{M}(X_o;\dataset)\neq \gbayes(X_o))]$,
we will apply Lemma~\ref{lem:cd14_lemma20_var} for any $x_o$ with $\Delta(x_o)>\Delta_{i_o}$ and use a trivial bound $\E_{\dataset}[1(\splitcls{k}{M}(x_o;\dataset)\neq \gbayes(x_o))]\le 1$ for all remaining $x_o$.
Taking expectations over $X_o$ and invoking the $(\b,B)$-margin condition, we have
\begin{align*}
\E_{\dataset}[\partitions_{X_o}(\splitcls{k}{M}(X_o;\dataset)\neq \gbayes(X_o))]
&\le \E_{X_o}\Bigl[1(\Delta(X_o)\le\Delta_{i_o}) 
+ e^{-\frac{(1-\tau)\tau^2}{2}M}
\\&\qquad\qquad+ 4e^{-\frac{(\Delta(X_o)-\Delta_o)^2}{8}M}1(\Delta(X_o)>\Delta_{i_o})\Bigr]\\
&\le e^{-\frac{(1-\tau)\tau^2}{2}M} + B\Delta_{i_o}^{\b}
\\&\qquad+ 4\E_{X_o}\Bigl[e^{-\frac{(\Delta(X_o)-\Delta_o)^2}{8}M}1(\Delta(X_o)>\Delta_{i_o})\Bigr].
\numberthis\label{eq:interim_bound_cis}
\end{align*}
By the same logic in the proof of Lemma~\ref{lem:cd14_lemma21}, the last term can be bounded by $B\Delta_{i_o}^{\beta}$ with the same $i_o$. 
Plugging this back into \eqref{eq:interim_bound_cis}, we have
\begin{align*}
\E_{\dataset}[\partitions_{X_o}(\splitcls{k}{M}(X_o;\dataset)\neq \gbayes(X_o))]
\le e^{-\frac{(1-\tau)\tau^2}{2}M} + 5B\Delta_{i_o}^{\b}.
\end{align*}
By substituting $\Delta_{i_o}=2^{i_o}\Delta_o+8\tau$, we have
\begin{align*}
\CIS_N(\splitcls{k}{M})
&\le 2\E_{\dataset}[\Pr_X(\splitcls{k}{M}(X_o;\dataset)\neq \gbayes(X_o))]\\
&\le 2e^{-\frac{(1-\tau)\tau^2}{2}M} 
\\&\quad+ 10B\Bigl(
\max\Bigl\{
2A\Bigl(2\frac{M}{N}\bigl(k+\log\frac{1}{\tau}\bigr)\Bigr)^\a,
\sqrt{\frac{32(\b+2)}{M}}
\Bigr\}
+8\tau
\Bigr)^{\b}
\end{align*}
and setting $\tau=\sqrt{\frac{(\log M)^{1+\eps}}{M}}$ for some $\eps>0$ concludes the proof.
\qed

\section{Analyzing Density Estimation Rules}
\label{app:density}
\subsection{On the Weak Consistency of the Standard \texorpdfstring{$k$}{k}-NN Density Estimator}

For completeness, here we provide a short proof to Theorem~\ref{thm:knn_density_estimate_weak_consistent}.
\ThmKnnDensityEstimateWeakConsistency*
\begin{proof}%
We rewrite the density estimate~\eqref{eq:base_knn_density_estimate} as
\begin{align*}
\nndensity{k}(x)
= \frac{\mu(\openball(x,r_{k}(x)))}{\Leb(\openball(x,r_{k}(x)))} \Bigl(\frac{\mu(\openball(x,r_{k}(x)))}{k/(n+1)}\Bigr)^{-1} \frac{n+1}{n}\frac{k-1}{k}.
\end{align*}
The following is the key observation:
\begin{lemma}\label{lem:cdf_knn_ball_prob}
For any underlying distribution, 
$\mu(\openball(x,r_{k}(x;X_{1:N})))\sim \BetaDist(k,n-k+1)$.
\end{lemma}
Hence, we have 
\[
\E[\mu(\openball(x,r_{k}(x)))]=\frac{k}{n+1},
\]
and by Chebyshev's inequality, the second term converges to 1 in probability, since
\begin{align*}
\Var\Bigl(\frac{\mu(\openball(x,r_{k}(x)))}{k/(n+1)}\Bigr)
&=\frac{(n+1)^2}{k^2} \frac{k(n-k+1)}{(n+1)^2(n+2)}
=\frac{n-k+1}{k(n+2)}\to 0
\end{align*}
as $n\to\infty$, provided that $k\to\infty$ as $n\to\infty$.

It is then enough to handle the convergence of the first term to $p(x)$ in probability.
Note that $\mu(\openball(x,r_{k}(x)))\to_p 0$ from Chebyshev's inequality since $k/n\to0$ as $n\to\infty$.
This can happen only if $\Leb(\openball(x,r_{k}(x)))\to_p 0$ as $n\to\infty$, or equivalently $r_{k}(x)\to_p 0$, since the density $p$ is continuous and positive at $x$.
Then, the desired convergence holds in probability, as  $\frac{\mu(\openball(x,r))}{\Leb(\openball(x,r))}\to p(x)$ as $r\to 0$ by the Lebesgue differentiation theorem (see, e.g., \cite{Rudin1987}). This concludes the proof.
\end{proof}

\subsection{Proof of Theorem~\ref{thm:function_of_density_knn_estimator_mse}}
In this section, we prove the convergence rate guarantee of the truncated $(k,M)$-NN estimator for a function of density.

\ThmKnnFunctionOfDensityMSE*

\begin{proof}%
Recall that
\begin{align*}
\splitftr{k}{M}(x;\Xb_{1:M})
&= \frac{1}{M}\sum_{m=1}^M \nnftr{k}(x;\Xb_m),
\end{align*}
where we denote $\nnftr{k}(x;\Xb)\defeq \phi_k^{(f)}(U_k(x;\Xb))$.
Note that we can decompose the expected squared error into bias and variance:
\begin{align*}
&\E_{\Xb_{1:M}}[(\splitftr{k}{M}(x;\Xb_{1:M})-f(p(x)))^2]\\
&= (\E_{\Xb_{1:M}}[\splitftr{k}{M}(x;\Xb_{1:M})]-f(p(x))])^2
+\Var_{\Xb_{1:M}}(\splitftr{k}{M}(x;\Xb_{1:M}))\\
&= (\E_{\Xb_1}[\nnftr{k}(x;\Xb_{1})]-f(p(x)))^2
+\frac{1}{M}\Var_{\Xb_1}(\nnftr{k}(x;\Xb_{1})). 
\end{align*}
Hence, it suffices to bound the bias and variance of the $k$-NN density estimator, as stated below;
we their proofs to the next sections.

\begin{restatable}[Bias of the truncated $k$-NN density estimator]{theorem}{ThmKnnFunctionOfDensityBias}
\label{thm:function_of_density_knn_estimator_bias}
For $x\in\Xc$, assume that $p$ is locally $(\sigma,S)$-H\"older continuous for some $\sigma\in(0,2]$ at $x$.
Let $\bar{\sigma}_d\defeq \frac{\sigma}{d}$ denote the normalized order of smoothness.
For $\a\in\Real$, consider the function $f(p)=f_\a(p)$ defined in \eqref{eq:poly_log_alpha_ftn}.
For $k>\a$ fixed, set $\nu_n=\Theta((\log n)^{1+\eps})$ for an arbitrary $\eps>0$.
\begin{enumerate}
\item If $\bar{\sigma}_d\ge \a-1$, for $\tau_n=0$, 
\begin{align*}
\abs{\E_{\Xb}[\nnftr{k}(x;\Xb)]-f(p(x))}
&= \bigOtilde\Bigl(p(x)^{k+2}n^{-1} + n^{-\bar{\sigma}_d}\Bigr).
\end{align*}

\item If $\bar{\sigma}_d< \a-1$, for $\tau_n=\Th(n^{-\frac{\bar{\sigma}_d}{k-\bar{\sigma}_d-1}})$, 
\begin{align*}
\abs{\E_{\Xb}[\nnftr{k}(x;\Xb)]-f(p(x))}
&= \bigOtilde\Bigl((p(x)^{k}\vee 1)n^{-\bar{\sigma}_d\frac{k-\a}{k-\bar{\sigma}_d-1}} 
+ p(x)^{k+2} n^{-1}\Bigr).
\end{align*}

\end{enumerate}
\end{restatable}

\begin{restatable}[Variance of the truncated $k$-NN density function estimator]{theorem}{ThmKnnFunctionOfDensityVariance}
\label{thm:function_of_density_knn_estimator_variance}
For $\a\in\Real$, the function $f(p)=f_\a(p)$ defined in \eqref{eq:poly_log_alpha_ftn}.
Let $k>2\a$.
If the underlying density $p$ is locally bounded from above by $\Cab_p$ around $x$, then
\begin{align*}
\Var_{\Xb}(\nnftr{k}(x;\Xb))
&= \bigOtilde\bigl(C_p^k\nu_n^{k-2\a}\bigr).
\end{align*}
\end{restatable}

Plugging in the respective bounds with $\tau_n=0$ and $\nu_n=\Th((\log n)^{1+\eps})$ for some $\eps>0$, we conclude the proof.
\end{proof}

\subsection{Technical Lemmas}

To prove Theorems~\ref{thm:function_of_density_knn_estimator_bias} and \ref{thm:function_of_density_knn_estimator_variance}, we invoke the following lemmas from \citep{Ryu--Ganguly--Kim--Noh--Lee2022}, which analyzes a family of $L_2$-consistent fixed-$k$-NN-based density functional estimators. 
The key technique is a bound on the gap between the densities of $U_{kn}(x)$ and $U_{k\infty}(x)\sim \GammaDist(k,p(x))$ to analyze the bias; see Lemma~\ref{lem:GOVLemma2_pdf_gap_bound} below.

\begin{lemma}[{\citealp[Lemma~B.2]{Ryu--Ganguly--Kim--Noh--Lee2022}}]
\label{supp:lem:incomplete_gamma}
For the \emph{lower incomplete gamma function}
$\gamma(s,x)\defeq \int_0^x t^{s-1}e^{-t}\diff t$
and the \emph{upper incomplete gamma function} $\Gamma(s,x)\defeq \int_x^\infty t^{s-1}e^{-t}\diff t$,
we have
\begin{align}
\gamma(s,x)
&\le \Gamma(s)\wedge\frac{x^s}{s}, ~~~~\quad\quad\forall s>0, x> 0,
\label{supp:eq:incomplete_gamma_lower}\\
\Gamma(s,x)
&\le \Gamma(s)x^{s-1}e^{-x+1}, \quad\forall s\ge 1, x\ge 1.
\label{supp:eq:incomplete_gamma_upper}
\end{align}
\end{lemma}

\begin{lemma}[{\citealp[Lemma~B.6]{Ryu--Ganguly--Kim--Noh--Lee2022}}]
\label{lem:GOVLemma4_smoothing} %
For a Lebesgue measure $\lambda$,
consider a $\lambda$-absolutely continuous probability measure $\mu$ with density $p$.
If the density $p$ is $(\sigma,S)$-H\"older continuous with constant over $\Bb(x,R)$ for $x\in\Real^d$ and some $\sigma\in[0,2]$,
then we have for any $r< R$,
\begin{align*}
\Bigl|\frac{\mu(\Bb(x,r))}{\lambda(\Bb(x,r))}-p(x)\Bigr|
&\le \frac{d}{\sigma+d}Sr^{\sigma},\text{ and}\\
\Bigl|\frac{\diff\mu(\Bb(x,r))}{\diff\lambda(\Bb(x,r))}-p(x)\Bigr| %
&\le Sr^{\sigma}.
\end{align*}
\end{lemma}

In what follows, let $\density_{U}(u)$ denote the density of a random variable $U$.
Further, we define a shorthand notation $\rvol_d(v)\defeq (v/\ups_d)^{1/d}$ to denote the radius of a $d$-dimensional ball of a volume $v$, where $\ups_d$ denotes the volume of a unit ball.
\begin{lemma}[{\citealp[Lemma~B.4]{Ryu--Ganguly--Kim--Noh--Lee2022}}]
\label{lem:GOVLemma2_pdf_gap_bound}
Suppose that $\nu_n=o(\sqrt{n})$ and $k=k_n=o(\sqrt{n})$.
Then there exists an absolute constant $C_0>0$ such that, for $n$ sufficiently large, we have
\begin{align*}
\abs{\rho_{U_{kn}(x)}(u)-\rho_{U_{k\infty}(x)}(u)}
&\le 
\Bigl|\frac{\diff\P(\Bb(x,\rvol(\frac{u}{n})))}{\diff\Leb(\Bb(x,\rvol(\frac{u}{n})))}-p(x)\Bigr|\\
&\qquad+2up(x)\Bigl|\frac{\P(\Bb(x,\rvol(\frac{u}{n})))}{\Leb(\Bb(x,\rvol(\frac{u}{n})))}-p(x)\Bigr|\\
&\qquad+\frac{C_0}{\Gamma(k)}\frac{(k^2+(up(x))^2)}{nu} (up(x))^{k}e^{-up(x)}.
\end{align*}
In particular, if $p$ is bounded such that  $p(\cdot)\le C_p<\infty$ and $p$ is $(\sigma,S)$-H\"older continuous over $\Bb(x,\rvol(\frac{u}{n}))$ for some $\sigma\in[0,2]$ and $S>0$,
we have
\begin{align*}
\abs{\rho_{U_{kn}(x)}(u)-\rho_{U_{k\infty}(x)}(u)}
&\le 
S\Bigl(1+2C_p \frac{d}{\sigma+d}u\Bigr)\Bigl(\frac{u}{n\ups_d}\Bigr)^{\bar{\sigma}_d}
\\ &\qquad
+ \frac{C_0}{\Gamma(k)}\frac{(k^2+(up(x))^2)}{nu} (up(x))^{k}e^{-up(x)}.
\end{align*}
\end{lemma}

\begin{lemma}[{\citealp[Lemma~B.11]{Ryu--Ganguly--Kim--Noh--Lee2022}}]
\label{lem:bound_kNN_ball_density}
If $p(\cdot)\le \Cab_p$ over $\overline{\Bb}(x,\rvol(\frac{u}{n}))$, we have
\[
\density_{kn}(u)
\le \frac{\Cab_p^k u^{k-1}}{\Gamma(k)}.
\]
\end{lemma}

\subsection{Proofs of Theorems~\ref{thm:function_of_density_knn_estimator_bias} and \ref{thm:function_of_density_knn_estimator_variance}}
We are now ready to prove the convergence rates of bias and variance of the truncated estimator.

\ThmKnnFunctionOfDensityBias*
\begin{proof}
Let $\density_{kn}(u)\defeq\rho_{U_{kn}(x)}(u)$, $\density_{k\infty}(u)\defeq\rho_{U_{k\infty}(x)}(u)=\frac{p^k}{\Gamma(k)}u^{k-1}e^{-up(x)}$, and 
let $\delta_{kn}(u)\defeq\rho_{U_{kn}(x)}(u)-\rho_{U_{k\infty}(x)}(u)$.
Recall that for $f(p)=f_\a(p)=\frac{p^\a}{\a}$, the estimator function is $\phi_k^{(f)}(u)=\frac{\Gamma(k)}{\Gamma(k-\a)}\frac{u^{-\a}}{\a}$ for $\a\neq 0$.
For $\a=0$, we can define the function as the limit of $f_\a(p)$ when $\a\to 0$, which is $f_0(p)=\log p$ and $\phi_k^{(f)}(u)=-\log u + \digamma(k)$.
\begin{align*}
&
\a \E_{\Xb}[\nnftr{k}(x;\Xb)]-f(p(x))
\\
&= \frac{\Gamma(k)}{\Gamma(k-\a)}\Bigl|\int_{\tau_n}^{\nu_n}\frac{ \density_{kn}(u)}{u^\a}\diff u-\int_0^{\infty}\frac{ \density_{k\infty}(u)}{u^\a}\diff u\Bigr|\\
&\le \frac{\Gamma(k)}{\Gamma(k-\a)}\Bigl(
\int_{0}^{\tau_n} \frac{\density_{k\infty}(u)}{u^\a}\diff u
+ \int_{\tau_n}^{\nu_n}\frac{\abs{\delta_{kn}(u)}}{u^\a}\diff u
+ \int_{\nu_n}^{\infty} \frac{\density_{k\infty}(u)}{u^\a}\diff u\Bigr).
\end{align*}
The first and the last terms integrating $\density_{k\infty}(u)/u$ over $(0,\tau_n)\cup (\nu_n,\infty)$ are bounded by incomplete gamma functions.
\begin{itemize}
\item The first term can be handled as follows:
\begin{align*}
\frac{\Gamma(k)}{\Gamma(k-\a)}\int_{0}^{\tau_n} \frac{\density_{k\infty}(u)}{u^\a}\diff u
= \frac{\gamma(k-\a,\tau_np(x))}{\Gamma(k-\a)} p(x)^{\a}
\le \frac{(\tau_n p(x))^{k-\a}}{\Gamma(k-\a+1)} p(x)^\a.
\numberthis\label{eq:first_term}
\end{align*}
Here, the inequality follows from Lemma~\ref{supp:lem:incomplete_gamma}, provided that $k> \a$.

\item For the last term that corresponds to $(\nu_n,\infty)$, we have
\begin{align*}
\frac{\Gamma(k)}{\Gamma(k-\a)}
\int_{\nu_n}^\infty \frac{\density_{k\infty}(u)}{u^\a}\diff u
&=\frac{\Gamma(k-\a,\nu_np(x))}{\Gamma(k-\a)}p(x)^\a
\\&
\le p(x)^\a (\nu_n p(x))^{k-\a-1} e^{-\nu_np(x)+1},
\numberthis\label{eq:last_term}
\end{align*}
where the inequality follows from Lemma~\ref{supp:lem:incomplete_gamma}, provided that $k\ge 2$ and $\nu_n p(x)>1$. 

\item We now bound the second term, integrating $\abs{\delta_{kn}(u)}/u^\a$ over $(\tau_n,\nu_n)$. 
Recall that from Lemma~\ref{lem:GOVLemma2_pdf_gap_bound} we have, for $k=o(\sqrt{n})$ and $u\in(0,\nu_n)$ with $\nu_n=o(\sqrt{n})$,
\begin{align*}
\frac{|\d_{kn}(u)|}{u^\a}
&=\frac{\abs{\rho_{U_{kn}(x)}(u)-\rho_{U_{k\infty}(x)}(u)}}{u^\a}\\
&\le 
\frac{S}{u^\a}\Bigl(1+2C_p \frac{d}{\sigma+d}u\Bigr)\Bigl(\frac{u}{n\ups_d}\Bigr)^{\bar{\sigma}_d}
+\frac{C_0}{\Gamma(k)}\frac{(k^2+(up(x))^2)}{nu^{\a+1}} (up(x))^{k}e^{-up(x)}.
\end{align*}
for $n$ sufficiently large.
To integrate the first term, let $w=\bar{\sigma}_d-\a$ and $D=\frac{2C_pd}{\sigma+d}$ temporarily.
Then, we have
\begin{align*}
&\frac{\Gamma(k)}{\Gamma(k-\a)}\int_{\tau_n}^{\nu_n} \frac{1+up(x)}{u^\a}\Bigl(\frac{u}{n}\Bigr)^{\bar{\sigma}_d}\diff u \\
&= \frac{S}{(n\ups_d)^{\bar{\sigma}_d}} \frac{\Gamma(k)}{\Gamma(k-\a)} \int_{\tau_n}^{\nu_n} \Bigl(1+\frac{2C_pd}{\sigma+d}u\Bigr)u^{\bar{\sigma}_d-\a}\diff u\\
&= \frac{S}{(n\ups_d)^{\bar{\sigma}_d}} \frac{\Gamma(k)}{\Gamma(k-\a)} \Bigl(\frac{\nu_n^{w+1}-\tau_n^{w+1}}{w+1}+D\frac{\nu_n^{w+2}-\tau_n^{w+2}}{w+2}\Bigr).
\numberthis\label{eq:second_term_first}
\end{align*}
Integrating the second term yields
\begin{align*}
&\frac{\Gamma(k)}{\Gamma(k-\a)}\frac{C_0}{\Gamma(k)} \int_{\tau_n}^{\nu_n}\frac{(k^2+(up(x))^2)}{nu^{\a+1}} (up(x))^{k}e^{-up(x)}\diff u\\
&\le \frac{C_0 p(x)^\a}{n\Gamma(k-\a)} \int_{0}^{\nu_n p(x)} t^{k-\a-1}e^{-t}(k^2+t^2)\diff t\\
&= \frac{C_0 p(x)^\a}{n\Gamma(k-\a)} \bigl\{k^2\gamma(k-\a,\nu_np(x))
+\gamma(k-\a+2,\nu_n,p(x))\bigr\}\\
&\le \frac{C_0 p(x)^\a}{n\Gamma(k-\a)} \Bigl(k^2 \frac{(\nu_n p(x))^{k-\a}}{k-\a}+\frac{(\nu_n p(x))^{k-\a+2}}{k-\a+2}\Bigr)\\
&= O\Bigl(\frac{p(x)^\a(\nu_np(x))^{k-\a+2}}{n}\Bigr),
\numberthis\label{eq:second_term_second}
\end{align*}
where the last inequality follows from Lemma~\ref{supp:lem:incomplete_gamma}.
\end{itemize}
Combining \eqref{eq:first_term}, \eqref{eq:last_term}, \eqref{eq:second_term_first}, and \eqref{eq:second_term_second}, 
we have
\begin{align*}
&\abs{\E_{\Xb}[\nnftr{k}(x;\Xb)]-f(p(x))}\\
&\le \frac{p(x)^\a}{\a} \Bigl\{
\frac{(\tau_n p(x))^{k-\a}}{\Gamma(k-\a+1)}
+ (\nu_n p(x))^{k-\a-1}e^{-\nu_np(x)+1}\\
&\qquad\qquad+\frac{C_0}{\Gamma(k-\a)} \frac{1}{n}\Bigl(
k^2 \frac{(\nu_np(x))^{k-\a}}{k-\a} + \frac{(\nu_np(x))^{k-\a+2}}{k-\a+2}
\Bigr)
\Bigr\}\\ 
&\quad+ \frac{1}{\a}\frac{S}{(n\ups_d)^{\bar{\sigma}_d}}\frac{\Gamma(k)}{\Gamma(k-\a)}\Bigl(
\frac{\nu_n^{w+1}-\tau_n^{w+1}}{w+1}
+D\frac{\nu_n^{w+2}-\tau_n^{w+2}}{w+2}
\Bigr)
\numberthis\label{eq:bias_bound_raw}\\
&= O\Bigl(
{p(x)^\a}\Bigl\{
(\tau_n p(x))^{k-\a}
+ (\nu_n p(x))^{k-\a-1}e^{-\nu_np(x)}
+ \frac{1}{n}(\nu_n p(x))^{k-\a+2}\Bigr\}\\
&\qquad\quad+\frac{1}{n^{\bar{\sigma}_d}} \Bigl(
\tau_n^{\bar{\sigma}_d-\a+1}\ones\bigl\{\bar{\sigma}_d<\a-1\bigr\} + \nu_n^{\bar{\sigma}_d-\a+2}\ones\bigl\{\bar{\sigma}_d\ge \a-2\bigr\}\Bigr)
\Bigr),
\end{align*}
where we hide multiplicative factors depending on $d$, $S$, $\sigma$, $k$, and $\a$ in the last line, assuming $k=O(1)$. For now, we assume $\a\neq 0$.

To simplify the rate further, we first note that it can be shown that {$\nu_n=\Omega(k\log k)$} guarantees the bound to vanish.\footnote{It can be shown by the following fact: Let $a\ge 1$ and $b>0$. Then, $x\ge 4a\log(2a)+2b$ implies $x\ge a\log x+b$.} 
In particular, for $k=O(1)$, if $\nu_n=\Omega((\log n)^{1+\eps})$ for some $\eps>0$, the integration over $(\nu_n,\infty)$ term decays faster than any polynomial rates.
Then, the rate simplifies to
\begin{align*}
&O\Bigl( p(x)^k\tau_n^{k-\a}+p(x)^{k+2}\frac{(\log n)^{(1+\eps)(k-\a+2)}}{n} 
\\&\qquad
+ \frac{\tau_n^{\bar{\sigma}_d-\a+1}}{n^{\bar{\sigma}_d}}
\ones\bigl\{\bar{\sigma}_d<\a-1\bigr\}
+\frac{(\log n)^{\bar{\sigma}_d-\a+2}}{n^{\bar{\sigma}_d}}
\ones\bigl\{\bar{\sigma}_d\ge \a-2\bigr\}\Bigr).
\end{align*}

Next, we control the last two terms, considering the following two cases separately.
\begin{itemize}
\item If $\bar{\sigma}_d\ge \a-1$, we can set $\tau_n=0$ (no lower truncation), and the rate simplifies to
\begin{align*}
O\Bigl( p(x)^{k+2}\frac{(\log n)^{(1+\eps)(k-\a+2)}}{n} +\frac{(\log n)^{\bar{\sigma}_d-\a+2}}{n^{\bar{\sigma}_d}}\Bigr).
\end{align*}

\item If $\bar{\sigma}_d<\a-1$, then we can choose an optimal $\tau_n$ by equating the first and second terms, which lead to the choice 
\[
\tau_n=\Th(n^{-\bar{\sigma}_d/(k-\bar{\sigma}_d-1)}).
\]
Since $k>\a$ is required, the denominator of the exponent $k-\bar{\sigma}_d-1$ is always positive when $\bar{\sigma}_d< \a -1$, and thus this choice is valid in the sense that $\tau_n= o(1)$.
With this choice, the rate becomes
\begin{align*}
O\Bigl( (p(x)^k\vee 1) n^{-\bar{\sigma}_d \frac{k-\a}{k-\bar{\sigma}_d-1}} 
+p(x)^{k+2}\frac{(\log n)^{(1+\eps)(k-\a+2)}}{n} 
+\frac{(\log n)^{\bar{\sigma}_d-\a+2}}{n^{\bar{\sigma}_d}}
\ones\bigl\{\bar{\sigma}_d\ge \a-2\bigr\}\Bigr).
\end{align*}
\end{itemize}
This proves the desired bound for $\a\neq 0$.
We can obtain the result for $f_0(p)=\log p$, by considering the limit when $\a\to 0$ in \eqref{eq:bias_bound_raw}.
For example, when 
$\bar{\sigma}_d\ge \a-1$, 
, the rate will become
\begin{align*}
O\Bigl( p(x)^{k+2}\frac{(\log n)^{(1+\eps)(k+2)}(\log \log n)}{n} 
+\frac{(\log n)^{\bar{\sigma}_d+2} (\log \log n)}{n^{\bar{\sigma}_d}}\Bigr).
\end{align*}
Note, therefore, that the rates in the statement hiding the multiplicative factors also hold for $\a=0$.
\end{proof}

\ThmKnnFunctionOfDensityVariance*

\begin{proof}
Let $\Cab_p(x,r)\defeq\sup\{p(x)\suchthat x\in \overline{\Bb}(x,r)\}$.
We have
\begin{align*}
\Var_\Xb(\nnftr{k}(x;\Xb))
&\le \E_{\Xb}[\nnftr{k}(x;\Xb)^2]\\
&= \E_{\Xb}\bigl[
\phi_k^{(f)}(U_k(x;\Xb))^2
\ones_{(\tau_n,\nu_n)}(U_{k}(x;\Xb))\bigr]\\
&= \E_{\Xb}\Bigl[
\Bigl(\frac{\Gamma(k)}{\Gamma(k-\a)}\frac{1}{\a U_k(x;\Xb)^\a}\Bigr)^2
\ones_{(\tau_n,\nu_n)}(U_{k}(x;\Xb))\Bigr]\\
&= \Bigl(\frac{\Gamma(k)}{\a\Gamma(k-\a)}\Bigr)^2
\int_{\tau_n}^{\nu_n}\frac{ \density_{kn}(u)}{u^{2\a}}\diff u\\
&\le \Bigl(\frac{\Gamma(k)}{\a\Gamma(k-\a)}\Bigr)^2
\int_{\tau_n}^{\nu_n}\frac{\Cab_p(x,\rvol(\frac{u}{n}))^k u^{k-1}}{\Gamma(k)} \frac{1}{u^{2\a}}\diff u \\
&\le \frac{\Gamma(k)}{\a^2\Gamma(k-\a)^2} 
\Cab_p\Bigl(x,\rvol\bigl(\frac{\nu_n}{n}\bigr)\Bigr)^k \int_{\tau_n}^{\nu_n} u^{k-2\a-1}\diff u\\
&\le \frac{\Gamma(k)}{\a^2\Gamma(k-\a)^2} 
\Cab_p^k \frac{\nu_n^{k-2\a}}{k-2\a}.
\end{align*}
Note that the second inequality follows from Lemma~\ref{lem:bound_kNN_ball_density}, and 
the third inequality holds for $n$ sufficiently large, so that $C_p(x,\rvol(\frac{\nu_n}{n}))\le C_p$.
Since we know that $\lim_{\a\to0}\frac{\Gamma(k)}{\Gamma(k-\a)}\frac{1}{\a u^\a} = -\log u + \digamma(k)$, the bound becomes 
\begin{align*}
\Var_\Xb(\nnftr{k}(x;\Xb))
&\le \E_{\Xb}[\nnftr{k}(x;\Xb)^2]\\
&\le (-\log \nu_n +\digamma(k))^2 \frac{(C_p\nu_n)^k}{\Gamma(k+1)}.
\end{align*}
This concludes the proof.
\end{proof}

\vskip 0.2in
\bibliography{ref}

\newcommand{\noopsort}[1]{}
\begin{thebibliography}{72}
\providecommand{\natexlab}[1]{#1}
\providecommand{\url}[1]{\texttt{#1}}
\expandafter\ifx\csname urlstyle\endcsname\relax
  \providecommand{\doi}[1]{doi: #1}\else
  \providecommand{\doi}{doi: \begingroup \urlstyle{rm}\Url}\fi

\bibitem[Alabduljalil et~al.(2013)Alabduljalil, Tang, and Yang]{Alabduljalil--Tang--Yang2013}
Maha~Ahmed Alabduljalil, Xun Tang, and Tao Yang.
\newblock Optimizing parallel algorithms for all pairs similarity search.
\newblock In \emph{Proc. Int. Conf. Web Search Data Mining}, pages 203--212, 2013.

\bibitem[Anastasiu and Karypis(2019)]{Anastasiu--Karypis2019}
David~C Anastasiu and George Karypis.
\newblock Parallel cosine nearest neighbor graph construction.
\newblock \emph{J. Parallel. Distrib. Comput.}, 129:\penalty0 61--82, 2019.

\bibitem[Angiulli and Pizzuti(2002)]{Angiulli--Pizzuti2002}
Fabrizio Angiulli and Clara Pizzuti.
\newblock Fast outlier detection in high dimensional spaces.
\newblock In \emph{Euro. Conf. Princ. Data Mining Knowledge Discov.}, pages 15--27. Springer, 2002.

\bibitem[Audibert et~al.(2007)Audibert, Tsybakov, et~al.]{Audibert--Tsybakov2007}
Jean-Yves Audibert, Alexandre~B Tsybakov, et~al.
\newblock Fast learning rates for plug-in classifiers.
\newblock \emph{Ann. Statist.}, 35\penalty0 (2):\penalty0 608--633, 2007.

\bibitem[Baldi et~al.(2014)Baldi, Sadowski, and Whiteson]{Baldi--Sadowski--Whiteson2014}
Pierre Baldi, Peter Sadowski, and Daniel Whiteson.
\newblock Searching for exotic particles in high-energy physics with deep learning.
\newblock \emph{{Nat. Commun}}, 5\penalty0 (1):\penalty0 1--9, 2014.

\bibitem[Balsubramani et~al.(2019)Balsubramani, Dasgupta, Freund, and Moran]{Balsubramani--Dasgupta--Freund--Moran2019}
Akshay Balsubramani, Sanjoy Dasgupta, Yoav Freund, and Shay Moran.
\newblock An adaptive nearest neighbor rule for classification.
\newblock In \emph{Adv. Neural Inf. Process. Syst.}, volume~32, pages 7579--7588, 2019.

\bibitem[Bentley(1975)]{Bentley1975}
Jon~Louis Bentley.
\newblock Multidimensional binary search trees used for associative searching.
\newblock \emph{Commun. ACM}, 18\penalty0 (9):\penalty0 509--517, 1975.

\bibitem[Beygelzimer et~al.(2006)Beygelzimer, Kakade, and Langford]{Beygelzimer--Kakade--Langford2006}
Alina Beygelzimer, Sham Kakade, and John Langford.
\newblock Cover trees for nearest neighbor.
\newblock In \emph{Proc. Int. Conf. Mach. Learn.}, pages 97--104, 2006.

\bibitem[Bhatt et~al.(2018)Bhatt, Huang, Kim, Ryu, and Sen]{Bhatt--Huang--Kim--Ryu--Sen2018b}
Alankrita Bhatt, Jiun-Ting Huang, Young-Han Kim, J.~Jon Ryu, and Pinar Sen.
\newblock Variations on a theme by {L}iu, {C}uff, and {V}erd\'{u}: The power of posterior sampling.
\newblock In \emph{Proc. {IEEE} Inf. Theory Workshop}, 2018.

\bibitem[Bhattacharjee and Chaudhuri(2020)]{Bhattacharjee--Chaudhuri2020}
Robi Bhattacharjee and Kamalika Chaudhuri.
\newblock When are non-parametric methods robust?
\newblock In \emph{Proc. Int. Conf. Mach. Learn.}, pages 832--841, July 2020.

\bibitem[Biau and Devroye(2015)]{Biau--Devroye2015}
G{\'{e}}rard Biau and Luc Devroye.
\newblock \emph{{Lectures on the Nearest Neighbor Method}}.
\newblock Springer International Publishing, 2015.

\bibitem[Biau et~al.(2010)Biau, C{\'e}rou, and Guyader]{Biau--Cerou--Guyader2010}
G{\'e}rard Biau, Fr{\'e}d{\'e}ric C{\'e}rou, and Arnaud Guyader.
\newblock On the rate of convergence of the bagged nearest neighbor estimate.
\newblock \emph{J. Mach. Learn. Res.}, 11\penalty0 (Feb):\penalty0 687--712, 2010.

\bibitem[Biau et~al.(2011)Biau, Chazal, Cohen-Steiner, Devroye, and Rodr{\'i}guez]{Biau--Chazal--Cohen-Steiner--Devroye--Rodriguez2011}
G{\'e}rard Biau, Fr{\'e}d{\'e}ric Chazal, David Cohen-Steiner, Luc Devroye, and Carlos Rodr{\'i}guez.
\newblock {A weighted k-nearest neighbor density estimate for geometric inference}.
\newblock \emph{Electron. J. Stat.}, 5\penalty0 (none):\penalty0 204 -- 237, 2011.
\newblock \doi{10.1214/11-EJS606}.
\newblock URL \url{https://doi.org/10.1214/11-EJS606}.

\bibitem[Breiman(1996)]{Breiman1996}
Leo Breiman.
\newblock Bagging predictors.
\newblock \emph{Mach. Learn.}, 24\penalty0 (2):\penalty0 123--140, 1996.

\bibitem[Breiman(1999)]{Breiman1999}
Leo Breiman.
\newblock Pasting small votes for classification in large databases and on-line.
\newblock \emph{Mach. Learn.}, 36\penalty0 (1):\penalty0 85--103, 1999.

\bibitem[Chaudhuri and Dasgupta(2014)]{Chaudhuri--Dasgupta2014}
Kamalika Chaudhuri and Sanjoy Dasgupta.
\newblock Rates of convergence for nearest neighbor classification.
\newblock In \emph{Adv. Neural Inf. Process. Syst.}, volume~27, pages 3437--3445. Curran Associates, Inc., 2014.

\bibitem[Chen et~al.(2018)Chen, Shah, et~al.]{Chen--Shah2018}
George~H Chen, Devavrat Shah, et~al.
\newblock \emph{Explaining the success of nearest neighbor methods in prediction}.
\newblock Now Publishers, 2018.

\bibitem[Cover(1968{\natexlab{a}})]{Cover1968a}
Thomas~M Cover.
\newblock Estimation by the nearest neighbor rule.
\newblock \emph{{IEEE} Trans. Inf. Theory}, 14\penalty0 (1):\penalty0 50--55, 1968{\natexlab{a}}.

\bibitem[Cover(1968{\natexlab{b}})]{Cover1968b}
Thomas~M Cover.
\newblock Rates of convergence for nearest neighbor procedures.
\newblock In \emph{Proc. Hawaii Int. Conf. Sys. Sci.}, volume 415, 1968{\natexlab{b}}.

\bibitem[Cover and Hart(1967)]{Cover--Hart1967}
Thomas~M Cover and Peter Hart.
\newblock Nearest neighbor pattern classification.
\newblock \emph{{IEEE} Trans. Inf. Theory}, 13\penalty0 (1):\penalty0 21--27, 1967.

\bibitem[Dasgupta and Kpotufe(2014)]{Dasgupta--Kpotufe2014}
Sanjoy Dasgupta and Samory Kpotufe.
\newblock Optimal rates for $k$-{NN} density and mode estimation.
\newblock In \emph{Adv. Neural Inf. Process. Syst.}, volume~27, pages 2555--2563. Curran Associates, Inc., 2014.

\bibitem[Dasgupta and Kpotufe(2019)]{Dasgupta--Kpotufe2019}
Sanjoy Dasgupta and Samory Kpotufe.
\newblock Nearest-neighbor classification and search.
\newblock In Tim Roughgarden, editor, \emph{Beyond Worst-Case Analysis}, chapter~1. Cambridge University Press, 2019.

\bibitem[Datar et~al.(2004)Datar, Immorlica, Indyk, and Mirrokni]{Datar--Immorlica--Indyk--Mirrokni2004}
Mayur Datar, Nicole Immorlica, Piotr Indyk, and Vahab~S Mirrokni.
\newblock Locality-sensitive hashing scheme based on p-stable distributions.
\newblock In \emph{Proc. 20th Ann. Symp. Comput. Geom.}, pages 253--262, 2004.

\bibitem[Devroye et~al.(1994)Devroye, Gyorfi, Krzyzak, and Lugosi]{Devroye--Gyorfi--Krzyzak--Lugosi1994}
Luc Devroye, Laszlo Gyorfi, Adam Krzyzak, and G{\'a}bor Lugosi.
\newblock On the strong universal consistency of nearest neighbor regression function estimates.
\newblock \emph{Ann. Statist.}, pages 1371--1385, 1994.

\bibitem[Devroye et~al.(1996)Devroye, Gy{\"o}rfi, and Lugosi]{Devroye--Gyorfi--Lugosi1996}
Luc Devroye, L{\'a}szl{\'o} Gy{\"o}rfi, and G{\'a}bor Lugosi.
\newblock \emph{A probabilistic theory of pattern recognition}, volume~31.
\newblock Springer Science \& Business Media, 1996.

\bibitem[Devroye and Wagner(1977)]{Devroye--Wagner1977}
Luc~P Devroye and Terry~J Wagner.
\newblock The strong uniform consistency of nearest neighbor density estimates.
\newblock \emph{Ann. Statist.}, pages 536--540, 1977.

\bibitem[Dua and Graff(2019)]{UCI2019}
Dheeru Dua and Casey Graff.
\newblock {UCI} {M}achine {L}earning {R}epository, 2019.
\newblock URL \url{http://archive.ics.uci.edu/ml}.

\bibitem[Duan et~al.(2020)Duan, Qiao, and Cheng]{Duan--Qiao--Cheng2020}
Jiexin Duan, Xingye Qiao, and Guang Cheng.
\newblock Statistical guarantees of distributed nearest neighbor classification.
\newblock In \emph{Adv. Neural Inf. Process. Syst.}, volume~33. Curran Associates, Inc., 2020.

\bibitem[Efremenko et~al.(2020)Efremenko, Kontorovich, and Noivirt]{Efremenko--Kontorovich--Noivirt2020}
Klim Efremenko, Aryeh Kontorovich, and Moshe Noivirt.
\newblock Fast and {B}ayes-consistent nearest neighbors.
\newblock In \emph{Proc. Int. Conf. Artif. Int. Statist.}, pages 1276--1286. PMLR, 2020.

\bibitem[Fix and Hodges(1951)]{Fix--Hodges1951}
Evelyn Fix and J.L. Hodges.
\newblock Discriminatory analysis: {N}onparametric discrimination, consistency properties.
\newblock Technical Report 4; 21-49-004, USAF School of Aviation Medicine, 1951.

\bibitem[Folland(2013)]{Folland2013}
Gerald~B Folland.
\newblock \emph{Real Analysis: Modern Techniques and Their Applications}.
\newblock John Wiley \& Sons, 2013.

\bibitem[Fritz(1975)]{Fritz1975}
Jozsef Fritz.
\newblock Distribution-free exponential error bound for nearest neighbor pattern classification.
\newblock \emph{{IEEE} Trans. Inf. Theory}, 21\penalty0 (5):\penalty0 552--557, 1975.

\bibitem[Fukunaga and Hostetler(1973)]{Fukunaga--Hostetler1973}
Keinosuke Fukunaga and L~Hostetler.
\newblock Optimization of $k$ nearest neighbor density estimates.
\newblock \emph{{IEEE} Trans. Inf. Theory}, 19\penalty0 (3):\penalty0 320--326, 1973.

\bibitem[Gadat et~al.(2016)Gadat, Klein, Marteau, et~al.]{Gadat--Klein--Marteau2016}
S{\'e}bastien Gadat, Thierry Klein, Cl{\'e}ment Marteau, et~al.
\newblock Classification in general finite dimensional spaces with the $k$-nearest neighbor rule.
\newblock \emph{Ann. Statist.}, 44\penalty0 (3):\penalty0 982--1009, 2016.

\bibitem[Gottlieb et~al.(2014)Gottlieb, Kontorovich, and Krauthgamer]{Gottlieb--Kontorovich--Krauthgamer2014}
Lee-Ad Gottlieb, Aryeh Kontorovich, and Robert Krauthgamer.
\newblock Efficient classification for metric data (extended abstract {COLT} 2010).
\newblock \emph{{IEEE} Trans. Inf. Theory}, 60\penalty0 (9):\penalty0 5750--5759, 2014.

\bibitem[Gottlieb et~al.(2018)Gottlieb, Kontorovich, and Nisnevitch]{Gottlieb--Kontorovich--Nisnevitch2018}
Lee-Ad Gottlieb, Aryeh Kontorovich, and Pinhas Nisnevitch.
\newblock Near-optimal sample compression for nearest neighbors.
\newblock \emph{{IEEE} Trans. Inf. Theory}, 64\penalty0 (6):\penalty0 4120--4128, 2018.

\bibitem[Gu et~al.(2019)Gu, Akoglu, and Rinaldo]{Gu--Akoglu--Rinaldo2019}
Xiaoyi Gu, Leman Akoglu, and Alessandro Rinaldo.
\newblock Statistical analysis of nearest neighbor methods for anomaly detection.
\newblock In \emph{Adv. Neural Inf. Process. Syst.}, volume~32. Curran Associates, Inc., 2019.

\bibitem[Guyon et~al.(2004)Guyon, Gunn, Ben-Hur, and Dror]{Guyon--Gunn--Ben-Hur--Dror2004}
Isabelle Guyon, Steve~R Gunn, Asa Ben-Hur, and Gideon Dror.
\newblock Result {A}nalysis of the {NIPS} 2003 {F}eature {S}election {C}hallenge.
\newblock In \emph{Adv. Neural Inf. Process. Syst.}, volume~4, pages 545--552. Curran Associates, Inc., 2004.

\bibitem[Gyorfi(1981)]{Gyorfi1981}
L~Gyorfi.
\newblock The rate of convergence of $k_n$-nn regression estimates and classification rules.
\newblock \emph{{IEEE} Trans. Inf. Theory}, 27\penalty0 (3):\penalty0 362--364, 1981.

\bibitem[Gy{\"o}rfi and Weiss(2021)]{Gyorfi--Weiss2021}
L{\'a}szl{\'o} Gy{\"o}rfi and Roi Weiss.
\newblock Universal consistency and rates of convergence of multiclass prototype algorithms in metric spaces.
\newblock \emph{J. Mach. Learn. Res.}, 22\penalty0 (151):\penalty0 1--25, 2021.

\bibitem[Hall and Samworth(2005)]{Hall--Samworth2005}
Peter Hall and Richard~J Samworth.
\newblock Properties of bagged nearest neighbour classifiers.
\newblock \emph{J. R. Stat. Soc. B}, 67\penalty0 (3):\penalty0 363--379, 2005.

\bibitem[Hanneke et~al.(2020)Hanneke, Kontorovich, Sabato, and Weiss]{Hanneke--Kontorovich--Sabato--Weiss2020}
Steve Hanneke, Aryeh Kontorovich, Sivan Sabato, and Roi Weiss.
\newblock Universal {B}ayes consistency in metric spaces.
\newblock \emph{Ann. Statist.}, page to appear, 2020.

\bibitem[Har-Peled et~al.(2012)Har-Peled, Indyk, and Motwani]{Har-Peled--Indyk--Motwani2012}
Sariel Har-Peled, Piotr Indyk, and Rajeev Motwani.
\newblock Approximate nearest neighbor: {T}owards removing the curse of dimensionality.
\newblock \emph{Theory Comput.}, 8\penalty0 (1):\penalty0 321--350, 2012.

\bibitem[Hastie et~al.(2009)Hastie, Tibshirani, and Friedman]{Hastie--Tibshirani--Friedman2009}
Trevor Hastie, Robert Tibshirani, and Jerome Friedman.
\newblock \emph{The elements of statistical learning: data mining, inference, and prediction}.
\newblock Springer Science \& Business Media, 2009.

\bibitem[Indyk and Motwani(1998)]{Indyk1998approximate}
Piotr Indyk and Rajeev Motwani.
\newblock Approximate nearest neighbors: towards removing the curse of dimensionality.
\newblock In \emph{Proc. Symp. Theory Comput.}, pages 604--613, 1998.

\bibitem[Kibriya and Frank(2007)]{Kibriya--Frank2007}
Ashraf~M Kibriya and Eibe Frank.
\newblock An empirical comparison of exact nearest neighbour algorithms.
\newblock In \emph{Euro. Conf. Princ. Data Mining Knowledge Discov.}, pages 140--151. Springer, 2007.

\bibitem[Kontorovich and Weiss(2015)]{Kontorovich--Weiss2015}
Aryeh Kontorovich and Roi Weiss.
\newblock {A Bayes consistent 1-{NN} classifier}.
\newblock In Guy Lebanon and S.~V.~N. Vishwanathan, editors, \emph{Proc. Int. Conf. Artif. Int. Statist.}, volume~38 of \emph{Proc. Mach. Learn. Res.}, pages 480--488, San Diego, California, USA, 09--12 May 2015. PMLR.

\bibitem[Kontorovich et~al.(2017)Kontorovich, Sabato, and Weiss]{Kontorovich--Sabato--Weiss2017}
Aryeh Kontorovich, Sivan Sabato, and Roi Weiss.
\newblock Nearest-neighbor sample compression: Efficiency, consistency, infinite dimensions.
\newblock In \emph{Adv. Neural Inf. Process. Syst.}, volume~30, pages 1572--1582. Curran Associates, Inc., 2017.

\bibitem[Korn and Korn(2000)]{Korn--Korn2000}
Granino~Arthur Korn and Theresa~M Korn.
\newblock \emph{{Mathematical Handbook for Scientists and Engineers: Definitions, Theorems, and Formulas for Reference and Review}}.
\newblock Courier Corporation, 2000.

\bibitem[Kozachenko and Leonenko(1987)]{Kozachenko--Leonenko1987}
L~F Kozachenko and Nikolai~N Leonenko.
\newblock {Sample estimate of the entropy of a random vector}.
\newblock \emph{Probl. Inf. Transm.}, 23\penalty0 (2):\penalty0 9--16, 1987.
\newblock (Russian).

\bibitem[Kpotufe and Verma(2017)]{Kpotufe--Verma2017}
Samory Kpotufe and Nakul Verma.
\newblock Time-accuracy tradeoffs in kernel prediction: controlling prediction quality.
\newblock \emph{J. Mach. Learn. Res.}, 18\penalty0 (1):\penalty0 1443--1471, 2017.

\bibitem[Kulkarni and Posner(1995)]{Kulkarni--Posner1995}
Sanjeev~R Kulkarni and Steven~E Posner.
\newblock Rates of convergence of nearest neighbor estimation under arbitrary sampling.
\newblock \emph{{IEEE} Trans. Inf. Theory}, 41\penalty0 (4):\penalty0 1028--1039, 1995.

\bibitem[Leonenko et~al.(2008)Leonenko, Pronzato, and Savani]{Leonenko--Pronzato--Savani2008}
Nikolai Leonenko, Luc Pronzato, and Vippal Savani.
\newblock {A class of R{\'{e}}nyi information estimators for multidimensional densities}.
\newblock \emph{Ann. Statist.}, 36\penalty0 (5):\penalty0 2153--2182, October 2008.

\bibitem[Liu et~al.(2021)Liu, Xu, and Shang]{Liu--Xu--Shang2021}
Ruiqi Liu, Ganggang Xu, and Zuofeng Shang.
\newblock Distributed adaptive nearest neighbor classifier: Algorithm and theory.
\newblock \emph{arXiv preprint arXiv:2105.09788}, 2021.

\bibitem[Loftsgaarden and Quesenberry(1965)]{Loftsgaarden--Quesenberry1965}
Don~O Loftsgaarden and Charles~P Quesenberry.
\newblock A nonparametric estimate of a multivariate density function.
\newblock \emph{Ann. Math. Statist.}, 36\penalty0 (3):\penalty0 1049--1051, 1965.

\bibitem[Lyon et~al.(2016)Lyon, Stappers, Cooper, Brooke, and Knowles]{Lyon--Stappers--Cooper--Brooke--Knowles2016}
Robert~J Lyon, BW~Stappers, Sally Cooper, JM~Brooke, and JD~Knowles.
\newblock Fifty years of pulsar candidate selection: {F}rom simple filters to a new principled real-time classification approach.
\newblock \emph{Mon. Not. R. Astron. Soc}, 459\penalty0 (1):\penalty0 1104--1123, 2016.
\newblock Data doi: {10.6084/m9.figshare.3080389.v1}.

\bibitem[Mack and Rosenblatt(1979)]{Mack--Rosenblatt1979}
YP~Mack and Murray Rosenblatt.
\newblock Multivariate k-nearest neighbor density estimates.
\newblock \emph{J. Multivar. Anal.}, 9\penalty0 (1):\penalty0 1--15, 1979.

\bibitem[Moore and Yackel(1977)]{Moore--Yackel1977}
David~S Moore and James~W Yackel.
\newblock Large sample properties of nearest neighbor density function estimators.
\newblock In \emph{Statist. Decis. Theory Relat. Top.}, pages 269--279. Elsevier, 1977.

\bibitem[Pedregosa et~al.(2011)Pedregosa, Varoquaux, Gramfort, Michel, Thirion, Grisel, Blondel, Prettenhofer, Weiss, Dubourg, Vanderplas, Passos, Cournapeau, Brucher, Perrot, and Duchesnay]{scikit-learn}
F.~Pedregosa, G.~Varoquaux, A.~Gramfort, V.~Michel, B.~Thirion, O.~Grisel, M.~Blondel, P.~Prettenhofer, R.~Weiss, V.~Dubourg, J.~Vanderplas, A.~Passos, D.~Cournapeau, M.~Brucher, M.~Perrot, and E.~Duchesnay.
\newblock Scikit-learn: Machine learning in {P}ython.
\newblock \emph{J. Mach. Learn. Res.}, 12:\penalty0 2825--2830, 2011.

\bibitem[Qiao et~al.(2019)Qiao, Duan, and Cheng]{Qiao--Duan--Cheng2019}
Xingye Qiao, Jiexin Duan, and Guang Cheng.
\newblock Rates of convergence for large-scale nearest neighbor classification.
\newblock In \emph{Adv. Neural Inf. Process. Syst.}, volume~32, pages 10768--10779. Curran Associates, Inc., 2019.

\bibitem[Ramaswamy et~al.(2000)Ramaswamy, Rastogi, and Shim]{Ramaswamy--Rastogi--Shim2000}
Sridhar Ramaswamy, Rajeev Rastogi, and Kyuseok Shim.
\newblock Efficient algorithms for mining outliers from large data sets.
\newblock In \emph{Proc. ACM Int. Conf. Manag. Data}, pages 427--438, 2000.

\bibitem[Rudin(1987)]{Rudin1987}
Walter Rudin.
\newblock \emph{Real and Complex Analysis}.
\newblock McGraw-Hill Education, 1987.

\bibitem[Ryu et~al.(2022)Ryu, Ganguly, Kim, Noh, and Lee]{Ryu--Ganguly--Kim--Noh--Lee2022}
J.~Jon Ryu, Shouvik Ganguly, Young-Han Kim, Yung-Kyun Noh, and Daniel~D. Lee.
\newblock Nearest neighbor density functional estimation from inverse {L}aplace transform.
\newblock \emph{{IEEE} Trans. Inf. Theory}, 68\penalty0 (6):\penalty0 3511--3551, 2022.
\newblock \doi{10.1109/TIT.2022.3151231}.

\bibitem[Samworth(2012)]{Samworth2012}
Richard~J Samworth.
\newblock Optimal weighted nearest neighbour classifiers.
\newblock \emph{Ann. Statist.}, 40\penalty0 (5):\penalty0 2733--2763, 2012.

\bibitem[Singh and P{\'{o}}czos(2016)]{Singh--Poczos2016}
Shashank Singh and Barnab{\'{a}}s P{\'{o}}czos.
\newblock Finite-sample analysis of fixed-k nearest neighbor density functional estimators.
\newblock In \emph{Adv. Neural Inf. Process. Syst.}, volume~29, pages 1217--1225. Curran Associates, Inc., 2016.

\bibitem[Slaney and Casey(2008)]{Slaney--Casey2008}
Malcolm Slaney and Michael Casey.
\newblock Locality-sensitive hashing for finding nearest neighbors.
\newblock \emph{{IEEE} Signal Process. Mag.}, 25\penalty0 (2):\penalty0 128--131, 2008.

\bibitem[Sun et~al.(2016)Sun, Qiao, and Cheng]{Sun--Qiao--Cheng2016}
Will~Wei Sun, Xingye Qiao, and Guang Cheng.
\newblock Stabilized nearest neighbor classifier and its statistical properties.
\newblock \emph{J. Am. Statist. Assoc.}, 111\penalty0 (515):\penalty0 1254--1265, 2016.

\bibitem[Uhlmann(1991)]{Uhlmann1991}
Jeffrey~K Uhlmann.
\newblock Satisfying general proximity/similarity queries with metric trees.
\newblock \emph{Inf. Process. Lett.}, 40\penalty0 (4):\penalty0 175--179, 1991.

\bibitem[Vanschoren et~al.(2013)Vanschoren, van Rijn, Bischl, and Torgo]{OpenML2013}
Joaquin Vanschoren, Jan~N. van Rijn, Bernd Bischl, and Luis Torgo.
\newblock Open{ML}: {N}etworked {S}cience in {M}achine {L}earning.
\newblock \emph{SIGKDD Explor.}, 15\penalty0 (2):\penalty0 49--60, 2013.
\newblock \doi{10.1145/2641190.2641198}.

\bibitem[Wagner(1971)]{Wagner1971}
T~Wagner.
\newblock Convergence of the nearest neighbor rule.
\newblock \emph{{IEEE} Trans. Inf. Theory}, 17\penalty0 (5):\penalty0 566--571, 1971.

\bibitem[Wang et~al.(2018)Wang, Jha, and Chaudhuri]{Wang--Jha--Chaudhuri2018}
Yizhen Wang, Somesh Jha, and Kamalika Chaudhuri.
\newblock Analyzing the robustness of nearest neighbors to adversarial examples.
\newblock In \emph{Proc. Int. Conf. Mach. Learn.}, pages 5133--5142, 2018.

\bibitem[Xue and Kpotufe(2018)]{Xue--Kpotufe2018}
Lirong Xue and Samory Kpotufe.
\newblock Achieving the time of 1-{NN}, but the accuracy of $k$-{NN}.
\newblock In \emph{Proc. Int. Conf. Artif. Int. Statist.}, pages 1628--1636, 2018.

\end{thebibliography}

\end{document}